\documentclass[11pt,reqno]{amsart}
\usepackage[margin=1.2in,letterpaper]{geometry}                
\usepackage{graphicx}

\usepackage{amsmath,amssymb}
\usepackage{enumerate}
\usepackage{epstopdf}
\usepackage[colorlinks=true, pdfstartview=FitV, linkcolor=blue, citecolor=blue, urlcolor=blue]{hyperref}
\usepackage{xcolor}
\usepackage{tikz-cd}
\usepackage[all]{xy}
\usepackage{appendix}
\usepackage[nameinlink, capitalize]{cleveref}
    \crefname{conj}{conjecture}{conjectures}
    \crefname{conj}{Conjecture}{Conjectures}
    
\numberwithin{equation}{section}

\newtheorem{thm}{Theorem}[section]

\newtheorem*{introthm*}{Main Theorem}

\newtheorem{cor}[thm]{Corollary}
\newtheorem{lem}[thm]{Lemma}
\newtheorem{prop}[thm]{Proposition}

\theoremstyle{definition}
\newtheorem{defn}[thm]{Definition}
\newtheorem*{defn*}{Definition}

\newtheorem{ex}[thm]{Example}

\newtheorem{obs}[thm]{Observation}
\newtheorem{constr}[thm]{Construction}

\newtheorem{fact}[thm]{Fact}

\newtheorem{rem}[thm]{Remark}
\newtheorem{notation}[thm]{Notation}

\theoremstyle{remark}

\newtheorem*{claim}{Claim}

\newcommand{\kk}{{\sf{k}}}

\newcommand{\A}{{\mathbb A}}
\newcommand{\C}{{\mathbb C}}

\newcommand{\N}{{\mathbb N}}
\renewcommand{\P}{{\mathbb P}}
\newcommand{\Q}{{\mathbb Q}}
\newcommand{\Z}{{\mathbb Z}}
\newcommand{\bu}{\mathbf{u}}

\def\Sp{\operatorname{Sp}}

\newcommand{\bw}{{\mathsf\Lambda}}

\newcommand{\R}{{\mathcal R}}

\newcommand{\m}{{\mathfrak m}}

\newcommand{\sym}{\mathfrak{S}}
\renewcommand{\SS}{\mathbb{S}}

\renewcommand{\l}{\lambda}

\def\Aut{\operatorname{Aut}}
\def\coker{\operatorname{coker}}
\def\ch{\operatorname{char}}

\def\mod{\operatorname{mod}}

\def\Res{\operatorname{Res}}

\def\GL{\operatorname{GL}}

\def\Ann{{\operatorname{Ann}}}
\def\Tor{{\operatorname{Tor}}}
\def\Ext{{\operatorname{Ext}}}
\def\HF{{\operatorname{\sf{HF}}}}
\def\H{{\operatorname{\sf{H}}}}
\def\Ker{\operatorname{Ker}}
\def\Hom{\operatorname{Hom}}
\def\Soc{\operatorname{Soc}}

\def\Supp{\operatorname{Supp}}

\newcommand{\pa}{\operatorname{part}}

\def\lex{{\operatorname{Lex}}}

\def\Sym{{\operatorname{Sym}}}

\def\l{\lambda}
\def\type{\operatorname{type}}

\def\diff{\operatorname{diff}}

\newcommand{\fm}{\mathfrak{m}}
\DeclareMathOperator{\Mod}{\operatorname{Mod}}

\DeclareMathOperator{\MG}{\mod_{\sym_n}}

\title[The resolution of a general principal symmetric ideal ]{The minimal free resolution \\
of a general principal symmetric ideal}

\author[]{Megumi Harada, Alexandra Seceleanu, and Liana M.~ \c{S}ega}

\address[Megumi Harada]{McMaster University}
\email{haradam@mcmaster.ca} 

\address[Alexandra Seceleanu]{University of Nebraska-Lincoln}
\email{aseceleanu@unl.edu}

\address[Liana M.~\c{S}ega]{University of Missouri Kansas City}
\email{segal@umkc.edu}

\date{\today}


\begin{document}

\begin{abstract}
We introduce the class of principal symmetric ideals, which are ideals generated by the orbit of a single polynomial under the action of the symmetric group. Fixing the degree of the generating polynomial, this class of ideals is parametrized by points in a suitable projective space. We show that the minimal free resolution of a principal symmetric ideal is constant on a nonempty Zariski open subset of this projective space and we determine this resolution explicitly. Along the way, we study two classes of graded algebras which we term narrow and extremely narrow; both of which are instances of compressed artinian algebras.
\end{abstract}

\maketitle 
\tableofcontents

\section{Introduction}
\label{s:prelim}
Let $\kk$ be a field and let $R=\kk[x_1,\ldots, x_n]$ denote the polynomial ring in $n$ variables. The ring $R$ comes equipped  with a natural action of the symmetric group $\sym_n$ given by 
$$\sigma \cdot f(x_1,\ldots, x_n)=f(x_{\sigma(1)}, \ldots, x_{\sigma(n)})\quad \text{for}\quad  \sigma\in \sym_n\,. 
$$
 We are interested in homogeneous ideals which acquire an induced action from the action of $\sym_n$ on $R$, which we refer to as symmetric ideals. The study of such ideals is both a classical topic in commutative algebra and one of current interest. In particular, the graded betti numbers for (certain families of) monomial symmetric ideals are the focus of the recent works  \cite{Galetto}, \cite{BDGMNORS} and \cite{MR}. 
 
 In this paper, we introduce a different class of symmetric ideals which we term   principal symmetric ideals.  We determine the graded betti numbers for general such ideals, that is, when their generator belongs to a non-empty Zariski-open set in its natural parameter space. This constitutes a symmetric counterpart to the study of general graded artinian  algebras parametrized by their Macaulay dual generators, an endeavor which was undertaken primarily in \cite{Iar} and \cite{Boij}.
 
We now introduce the main characters of our work in more detail. Let $f_1, \ldots, f_a\in R$. Define the {\em symmetric ideal} generated by $f_1, \ldots f_a$ to be 
\[
(f_1, \ldots f_a)_{\sym_n}=(\sigma \cdot f_i : \sigma \in \sym_n, 1\leq i\leq a).
\]
We are particularly interested in the case of homogeneous principal symmetric ideals.
A {\em principal symmetric ideal} is an ideal of the form $(f)_{\sym_n}$, where $f\in R$ is a homogeneous polynomial.

Fix $d\in \mathbb N$. A parameter space for the set of principal symmetric ideals generated in degree $d$ is $\P^{N-1}$, where $N=\binom{n+d-1}{d}$.
Indeed, let $M_d$ denote the set of monomials of degree $d$ in $R$ listed as $M_d=\{m_1, \ldots, m_N\}$ in an arbitrary order (for example, lexicographical).
Points in $\P^{N-1}$ parametrize principal symmetric ideals generated in degree $d$ via the assignment 
\begin{align*}
&\Phi: \P^{N-1} \to \text{ principal symmetric ideals} \\
& \Phi(c_1:\cdots: c_N) :=(f_c)_{\sym_n}\,,\qquad \text{where}\quad f_c=\sum_{i=1}^Nc_im_i\,.\qedhere
\end{align*}

The map $\Phi$ above is onto, but not one-to-one. It allows us to formulate a notion of a general principal symmetric ideal, which we now formalize. 
We will say that \emph{a general principal symmetric ideal generated in degree $d$ satisfies property $\mathcal P$} if there exists a non-empty Zariski-open set $U$ of $\P^N$ so that for each  $c\in U$, the principal symmetric ideal $\Phi(c)=(f_c)_{\sym_n}$ satisfies the property $\mathcal P$. 

With this terminology in place, we inquire into properties satisfied by general principal symmetric ideals. Our initial interest in this question was spurred by a result in \cite[Theorem 1]{Kre}, where Kretschmer proves (in a different language)  a more general version of the following result. 

\begin{thm}[{\cite[Theorem 1\,(i)]{Kre}}]
\label{thm:Kre}
For each $d\in \N$,  a general principal symmetric ideal $I$ generated in degree $d$ has the property that  $R/I$ is artinian. 
\end{thm}

The main goal of this paper is to describe general symmetric ideals in more detail, specifically in terms of their numerical and homological invariants such as the Hilbert series, socle type, Macaulay inverse systems, and minimal graded free resolutions. Recall that the graded betti numbers of a graded $R$-module $U$ are 
$\beta_{i,j}^R(U) :=\dim_\kk \Tor_i^R(U, \kk)_j$ for $i,j\in \mathbb Z$.  They are summarized in a betti table which displays $\beta_{i,j}(U)$ in column $i$ and row $j-i$.

The following is our main result. It employs the notation $P(d)$ for the number of partitions of $d$, that is,
\[
P(d) =\# \left \{ (p_1, \ldots, p_t) \mid  p_1\geq  \ldots \geq  p_t, p_1+\cdots+p_t=d, t\in \N \right \}.
\]

\begin{introthm*}[\Cref{introthm}]
Suppose $\kk$ is  infinite  with ${\rm char}(\kk)=0$ and fix an integer $d\geq 2$. For sufficiently large  $n$, a general principal symmetric ideal $I$ of ${\kk}[x_1,\ldots, x_n]$ generated in degree $d$ satisfies the following. 
\begin{enumerate}[\quad\rm(1)]
\item The Hilbert function of $A=R/I$ is given by
\[
 \HF_A(i):=\dim_\kk A_i = \begin{cases}\dim_\kk R_i &\text{if $i\le d-1$}\\
 P(d)-1&\text{if $i=d$}\\
 0 &\text{if $i>d$.}\
 \end{cases}
\]
\item The betti table of $A$ has the form
\[
\begin{matrix}
     &0&1&2&\cdots&i &\cdots & n-2 &n-1&n\\
     \hline
     \text{total:}&1&u_1&u_2&\cdots&u_i&\cdots&u_{n-2} &u_{n-1}+\ell&a+b\\
     \hline
     \text{0:}&1&\text{.}&\text{.}&\text{.}&\text{.}&\text{.}&\text{.}&\text{.}&\text{.}\\
     \text{\vdots}&\text{.}&\text{.}&\text{.}&\text{.}&\text{.}&\text{.}&\text{.}&\text{.} &\text{.}\\
     \text{d-1:}&\text{.}&u_1&u_2&\cdots&u_i&\cdots&u_{n-2} &u_{n-1}&b\\
     \text{d:}&\text{.}&\text{.}&\text{.}&\text{.}&\text{.}&\text{.}&\text{.}&\ell & a \\
     \end{matrix}
\]
with 
\begin{eqnarray*}
a &=&P(d)-1,\\
 b&=&
 \dim_\kk R_{d-1}-(P(d)-1)(n-1)+P(d-1) \\
  u_{i+1} &=& \binom{n+d-1}{d+i}\binom{d+i-1}{i}-\left(P(d)-1\right)\binom{n}{i} \\
    \ell &=& P(d)-P(d-1)-1.
  \end{eqnarray*}
\item The graded minimal free resolution of $A$ has $\sym_n$-equivariant structure described by the $\sym_n$-irreducible decompositions for the modules $\Tor_i^R(A,\kk)$ given in 
\Cref{thm:equivariantTors}.
\end{enumerate}

Moreover,  the Poincar\'e series of all finitely generated graded $A$-modules are rational, sharing a common denominator. When $d>2$, $A$ is Golod. When $d=2$, $A$ is Gorenstein and Koszul. 
\end{introthm*}

In order to prove the above theorem, we single out the class of {\em narrow graded algebras}, which we believe to be interesting in its own right. A narrow algebra is an artinian quotient algebra $A=R/I$ with $I$ a homogeneous ideal,  so that the initial degree of $I$ coincides with the top socle degree of $A$ cf.~\Cref{def:narrow}. It is a consequence of the definition and in fact equivalent to it, that the betti table of a narrow algebra is concentrated in two adjacent rows, except for its $0$-th homological degree. In view of this, \Cref{introthm} (2) yields that quotient algebras defined by principal symmetric ideals in sufficiently many variables are narrow. In fact they are {\em extremely narrow}, cf.~\Cref{def:extremelynarrow}, a property that forces additionally that the betti table is concentrated in a single row except for the $0$-th, $(n-1)$-st and $n$-th homological degree according to \Cref{en-props}. 

Algebras with Hilbert function as in part (1) of the above Theorem are instances of {\em compressed algebras}; see \Cref{def:compressed}. This class of algebras has been studied extensively in the graded case most notably by Iarrobino \cite{Iar} and Boij \cite{Boij} and in the local case by  Kustin-\c{S}ega-Vraciu \cite{KSV}. Compressed level artinian algebras are often studied using the {\em Macaulay inverse system}, which parametrizes this class of algebras. Our \Cref{introthm} shares some similarities with results regarding general level artinian algebras: indeed, these are known to be compressed, and their betti tables are also concentrated in two adjacent rows with the exception of the 0-th and $n$-th homological degrees \cite{Boij}. Unlike our results, those for general artinian level algebras hold for an arbitrary number of variables. However \Cref{introthm} does not hold when there are not enough variables; see \Cref{ex:fewvar}. 

The equivariant structure for the minimal free resolution of a symmetric ideal incorporates the action of the symmetric group in describing the resolution. An action of $\sym_n$ on $R/I$ extends to an action on the minimal free resolution of $R/I$ making each of the graded vector spaces $\Tor^R_i(R/I,\kk)_j$ into a finite-dimensional representation of the symmetric group as recalled in \Cref{lem:A5}. We describe the $\sym_n$-equivariant structure of the minimal free resolution of a general principal symmetric ideal by finding the decomposition of $\Tor^R_i(R/I,\kk)_j$ into irreducible $\sym_n$-modules (Specht modules) in \Cref{thm:equivariantTors}.

Equivariant structures on the resolutions of certain families of symmetric ideals have been previously described by Galetto \cite{Galetto} for the ideals generated by all squarefree monomials of a given degree, by Biermann-de Alba-Galetto-Murai-Nagel-O'Keefe-R\"{o}mer-Seceleanu \cite{BDGMNORS} for symmetric shifted ideals, and by Murai-Raicu \cite{MR} for all symmetric monomial ideals. Our work does not exhibit significant overlap with these contributions. The graded free resolutions for the  powers of the homogeneous maximal ideal can be deduced from work of Buchsbaum-Eisenbud \cite{BE}, which is an important ingredient in our proof. The focus of \cite{BE} is however on equivariant resolutions with respect to an action of the general linear group whence a translation to the induced action of the symmetric subgroup is necessary; see \Cref{lem:Lirreps}.

Our paper is organized as follows. \Cref{s:rep} introduces some representation-theoretic tools, while \cref{s:duality} recalls Macaulay duality and its homological consequences. \cref{s:narrow} focuses on properties of narrow algebras, including equivalent descriptions in terms of their minimal free resolutions in  \Cref{l:narrow-equiv2} and the Macaulay inverse system in \Cref{l:narrow-equiv1}. The subclass dubbed extremely narrow algebras is also discussed therein, with a focus on determining the betti numbers of these algebras in \Cref{cor:betti}. \cref{sec: quadratic symmetric} constitutes an introduction to employing the methods presented so far for quotient algebras defined by principal symmetric ideals. More precisely,  in \cref{sec: quadratic symmetric}  we analyze the case of quadratic principal symmetric ideals, which is less involved, but sets forth a blueprint for our general investigation.  In  \cref{s:example} we construct sufficiently general principal symmetric ideals generated in arbitrary degree, for which we determine the Macaulay inverse system in \cref{sec: linear relations}. This allows us to show that quotient algebras defined by principal symmetric ideals are extremely narrow. Our main theorem comes together in \cref{sec:main}, by combining all the ingredients prepared in advance.  Finally, in the appendix, we make brief remarks about the category of finitely generated graded $R$-modules with group actions. 
\medskip

\noindent \textbf{Acknowledgements.} 
We thank the Fields Institute for sponsoring the two ``Women in Commutative Algebra and Algebraic Geometry'' virtual workshops, in 2022 and 2023, where this collaboration began.  We also thank Claudia Miller for co-organizing these Fields Institute WCAAG workshops, together with the first author, and we thank Martha Precup for her contributions in early discussions on the topic of this paper. We also thank Tony Iarrobino for his extensive comments on our work.

Harada is supported in part by NSERC Discovery Grant 2019-06567 and a Canada Research Chair Tier 2 award.  Seceleanu is supported by NSF DMS--2101225 and DMS--2401482.  \c{S}ega is supported in part by a  Simons Foundation grant  (\#354594).

\section{Representation theory for the symmetric group}
\label{s:rep}

The irreducible representations of the symmetric group $\sym_n$ were worked out by Young and Specht in the early 20th century. Recognizing the latter mathematician's contribution, the irreducible representations of $\sym_n$ are now called Specht modules. 

Let $n$ be a positive integer. We denote $[n]=\{1,\ldots, n\}$. Let $\l$ be a partition of $n$, that is, a tuple $\l=(\l_1, \ldots, \l_p)$ such that $\l_1\geq \l_2\geq \cdots \geq \l_p$ and   $|\l|=\l_1+\cdots+\l_p=n$. We denote the fact that $\l$ is a partition of $n$ by $\l \vdash n$. For a partition $\l$ we write $\#\l$ for the number of parts, not to be confused with $|\l|$, the sum of the parts of $\l$.  When it is convenient to do so, we view $\l$ as a $t$-tuple ($t>\#\l$) by appending $t-\#\l$ entries equal to zero.

Specht $\sym_n$-modules are in bijection with partitions of $n$. For each $\l\vdash n$ the corresponding Specht module is denoted $\Sp_\l$.  
 We adopt the convention that if $\alpha$ is a tuple which is {\em not} a partition then $\Sp_\alpha=0$.
The Specht modules $\Sp_\l$, where $\l$ varies over all partitions of $n$, form a complete set of finite-dimensional irreducible representations for $\sym_n$ over $\mathbb{\C}$ (see \cite[Theorem 4.12]{James}). In fact, since Specht modules are defined over $\Z$, they continue to form a complete set of irreducible representations over any field of characteristic zero.

\begin{ex}\label{ex:altrep}
The Specht modules $\Sp_{(n)}$ and $\Sp_{(1^n)}$ are the trivial and sign representation of $\sym_n$, respectively. 
\end{ex}

The next result collects a few useful properties of Specht modules used in our proofs.

\begin{fact}
\label{prop:propertiesofSpecht}
The following are known properties of the Specht modules. 
\begin{enumerate}
\item The natural representation of $\sym_n$, acting by permutation matrices on $V={\kk}^n$, decomposes as
$V=\Sp_{(n)} \oplus \Sp_{(n-1,1)}$, where $\Sp_{(n-1,1)}$ is the standard representation.
\item The exterior powers of the standard representation are $\bw^i \Sp_{(n-1,1)} \cong \Sp_{(n-i,1^i)}$ cf.~{\color{blue} \cite[Exercise 4.6]{FultonHarris}.} 
\item  For each partition $\l$ of $n$ there is  an identity cf.~\cite[p. 257-258]{Hamermesh}
\[
\Sp_{\l} \otimes_{\kk} \Sp_{(n-1,1)}= \bigoplus_{\mu \in \mathcal{I}(\l)} {\Sp_{\mu}}^{c_\mu},\]
 where: $\mathcal{I}(\l)$ is the set of partitions whose Young diagram is obtained from that of  $\l$ by removing a box from the Young diagram of $\l$ and adding a box to the resulting diagram, $c_\mu=1$ whenever $\mu\neq \l$, and $c_\l$ is the number of distinct parts of $\l$ minus one. 
\end{enumerate}
\end{fact}

\subsection{The category $\MG R$}

The polynomial ring $R=\kk[x_1, \ldots, x_n]=\Sym_\kk(R_1)$ with homogeneous maximal ideal $\m=(x_1,\ldots,x_n)$ carries the structure of an $\sym_n$-representation, given by extending the {\em permutation representation} structure of $R_1$ (obtained by permuting variables) to the symmetric algebra it generates. This coincides with the explicit action of $\sym_n$ on polynomials given in the beginning of the introduction. 

An {\em $\sym_n$-equivariant} graded $R$-module $M$ is an $R$-module endowed with an $\sym_n$-action which is degree-preserving and for which  the multiplication map $R\otimes M \to M$ is $\sym_n$-equivariant.  Let us denote by $\mod_{\sym_n} R$ the category of finitely generated graded $\sym_n$-equivariant $R$-modules.  
This category is discussed in more detail in \Cref{appendix}. 

If $M$ is an object in $\mod_{\sym_n} R$, then the action of $\sym_n$ extends to the entire minimal free resolution $F$  of $M$; see \cite{Galetto}. If $F_i$ is a free module appearing in $F$, then  $F_i$ is isomorphic in $\MG R$ to $F_i/\m F_i\otimes_\kk R$, hence the action of $\sym_n$ on $F_i$ is determined by its action on the vector spaces $F_i/\m F_i \cong \Tor_i^R(M,\kk)$.

In this paper we will be interested in quotient rings given by principal symmetric ideals, viewed as members of $\mod_{\sym_n} R$. In particular we consider how the $\sym_n$-representation structure of such quotient algebras $R/(f_c)_{\sym_n}$ varies in families parametrized by $c\in \P^{N-1}$ as in  the introduction.  Our approach will be via determining the equivariant structure of the modules $\Tor_i^R(R/(f_c)_{\sym_n},\kk)$.  We show in \Cref{thm:equivariantTors} that the isomorphism class of $\Tor_i^R(R/(f_c)_{\sym_n},\kk)$ in $\MG R$ is constant on a Zariski open set of $\P^{N-1}$ for $n\gg 0$  and we determine its value. 

\begin{ex}
\label{ex:resk}
Set $M=\kk$ with trivial $\sym_n$-action. It is well known that the minimal free resolution of $\kk$ is  the Koszul complex  $K(x_1, \ldots, x_n)$ and that the free modules in this complex are described by $K_i=\bw_\kk^i R_1 \otimes_\kk R$. We conclude that 
\[
\Tor_i^R({{\kk}}, {\kk})  \cong K_i/\m K_i=\bw_\kk^i R_1\otimes_\kk (R/\m R)\cong\bw_\kk^i R_1.
\]
\Cref{prop:propertiesofSpecht}~(1) and (2) yield (noting also that $\Sp_{(n)}$ is the trivial $1$-dimensional representation) that 
\begin{eqnarray}
\label{eq:equivTorkk}
\begin{split}
\Tor_i^R({{\kk}}, {\kk}) & \cong & \bw_\kk^i R_1= \bw_\kk^i \left(  \Sp_{(n)} \oplus \Sp_{(n-1,1)}\right)=\bigoplus_{j=0}^i \left(\bw_\kk^j \Sp_{(n)}\otimes_{\kk} \bw_\kk^{i-j} \Sp_{(n-1,1)} \right) \\
&=&  \bw_\kk^{i} \Sp_{(n-1,1)} \oplus  \bw_\kk^{i-1} \Sp_{(n-1,1)} = \Sp_{(n-i,1^i)} \oplus  \Sp_{(n-i+1,1^{i-1})}.
\end{split}
\end{eqnarray}
\end{ex}

\begin{ex}
\label{ex:resmd}
Fix an integer $d\geq 1$ and set $M=\m^d$. The minimal free resolution $\mathcal L$ of $M$ over $R$ can be deduced from the more general construction given in \cite{BE}. 
The Koszul complex $K(x_1, \ldots, x_n)$ splits into graded strands of the form
\begin{equation}
\label{eq:strand}
\cdots \to \bw_\kk^{i+1} R_1 \otimes_\kk R_{j-1} \to \bw_\kk^{i} R_1 \otimes_\kk R_{j} \to \bw_\kk^{i-1} R_1 \otimes_\kk R_{j+1}\to \cdots.
\end{equation}
For our purposes it suffices to note that the free module appearing in homological degree $i$ the resolution $\mathcal L$ of $\m^d$ can be described as   
\begin{eqnarray}
\label{eq:Lasker}
\mathcal L_{i,d} &=& \ker\left (\bw^{i}_\kk R_1 \otimes_\kk R_{d} \to \bw_\kk ^{i-1} R_1 \otimes_\kk R_{d+1}\right)\otimes_\kk R\\
&=&{\rm im}\left (\bw^{i+1}_\kk R_1 \otimes_\kk R_{d-1} \to \bw_\kk ^{i} R_1 \otimes_\kk R_{d}\right)\otimes_\kk R.
\end{eqnarray}
where the map above comes from the Koszul complex~\eqref{eq:strand}. 
The differentials in the complex $\mathcal L$ are induced, viewing  $\mathcal L_{i,d}$ as a submodule of  by $ \bw_\kk ^{i} R_1 \otimes_\kk R_{d}\otimes_\kk R$, by $\partial\otimes {\rm id}$ where $\partial_i:\bw_\kk ^{i} R_1 \to \bw_\kk ^{i} R_1$ is the differential of the Koszul complex and ${\rm id}$ is the identity map of $R_{d}\otimes_\kk R$.
The module $M$ and indeed the entire resolution $\mathcal L$ are equipped with an action of the general linear group $\GL_n(\kk)$ extended from the natural action of $\GL_n(\kk)$ on $R_1$.  The $\GL_n(\kk)$-module 
\[ 
\SS_{(j, 1^{i})}=\ker\left (\bw_\kk^{i} R_1 \otimes_\kk R_{j} \to \bw_\kk^{i-1} R_1 \otimes_\kk R_{j+1}\right), 
\]
where the map is the differential of the Koszul complex, is a $\GL_n(\kk)$-representation known as the {\em Schur module} corresponding to the partition $(j, 1^i)$; see \cite[Exercise 6.20 (c)]{FultonHarris}.\footnote{The indexing convention for Schur modules used here differs from that in  \cite[Example (2.1.3)(h)]{Weyman} by transposing partitions.} 

From \eqref{eq:Lasker} we obtain the isomorphisms
\begin{equation}
\label{eq:LasSchur}
\Tor_i^R(\m^d,{\kk}) \cong \mathcal L_{i,d}/\m \mathcal L_{i,d} \cong \SS_{(d, 1^i)}
\end{equation}
of $\GL_n(\kk)$-modules.
The particular case $\mathcal L_0/\m \mathcal L_0=\SS_{(d)}$ identifies $R_d\cong \SS_{(d)}$ in $\MG R$. 
\end{ex}

While \eqref{eq:LasSchur} identifies the structure of $\mathcal L$ as a $\GL_n(\kk)$-representation, we are instead interested in its structure as a $\sym_n$-representation, where we view $\sym_n$ as a subgroup of $\GL_n(\kk)$ by its usual embedding  as the subgroup of permutation matrices. To describe this, we need additional notation. Let $\SS_\l$ denote the Schur module determined by a partition $\l$. Up to isomorphism, the Schur modules $\SS_\l$ for $\l$ a partition with length $\leq n$ parametrize the irreducible polynomial representations of $\GL_n(\kk)$. The construction of Schur modules in this generality can be found in e.g.~\cite[\S 6.1]{FultonHarris}.  If  a partition $\nu = (\nu_1,\nu_2,\ldots ,\nu_s)$ has the property that $n \geq \nu_1 +\lvert \nu \rvert$ (note in particular this means $\nu$ is a partition of an integer strictly smaller than $n$), we may define $\nu(n) :=(n-\lvert \nu \rvert, \nu_1,\nu_2,\ldots ,\nu_s)$, which is a partition of $n$. 
We denote by $\Res^{\GL_n(\kk)}_{\sym_n}$ the restriction operator which takes a $\GL_n(\kk)$-representation and considers it as a $\sym_n$-representation. 
Let $\l$ be a partition. Applying the restriction to the corresponding Schur module $\SS_\l$ we obtain a formula 
\begin{equation}
\label{eq:res}
\Res^{\GL_n(\kk)}_{\sym_n} \SS_\l =\bigoplus_{\nu_1 + \lvert \nu \rvert \leq n} \left(\Sp_{\nu(n)}\right)^{\oplus a_\l^{\nu(n)}}\,,
\end{equation}
where the $a_\l^{\nu(n)}$ are by definition the multiplicities with which the relevant Specht modules appear in the decomposition of the LHS of~\eqref{eq:res} into irreducible $\sym_n$-representations. In particular,  $a_\l^{\nu(n)}$ are non-negative integers by definition. Littlewood \cite{Littlewood} (see also \cite{ST} and \cite[Exercise 7.74]{Stanley}) gave a formula for the multiplicities $a_\l^{\nu(n)}$ by means of plethysm, recorded below in  \eqref{eq:plethysm}.

\begin{fact}({\color{blue} \cite{Littlewood}}) 
\label{fact:2}
\begin{enumerate}
\item The  multiplicities for the irreducible decomposition of the restricted representation in \eqref{eq:res} are given by
\begin{equation}
\label{eq:plethysm}
 a_\l^{\nu(n)} =\langle s_\l,s_{\nu(n)}[1+h_1 +h_2 +\cdots]\rangle\,,
 \end{equation}
where $s_\l(x_1,\cdots,x_n) = {\rm char}(\SS_\l)$ is the Schur polynomial, the character of the Schur module $\SS_\l$, and, for a non-negative integer $k$, we have 
$$h_k :=\sum_{1\leq i_1\leq \cdots\leq i_k\leq n} x_{i_1}x_{i_2}\cdots x_{i_k}$$
is the complete homogeneous symmetric polynomial. The notation $s_{\nu(n)}[1+h_1+h_2+\cdots ]$ stands for plethysm. Specifically, $g[f]$ is the result of substituting the monomials in the support of $f$ into $g$. The Hall  inner product for characters is determined by $\langle s_\l, s_\mu \rangle  = \delta_{\l,\mu}$.
\item The limit  $a_\l^{\nu}:=\lim_{n\to \infty} a_\l^{\nu(n)}$ exists. 
\item For $\l,\nu$ partitions with $|\l|\leq|\nu|$, one has $a_\l^{\nu}=\delta_{\l,\nu}$.
\end{enumerate}
\end{fact}

\begin{rem} 
It is a long-standing question to give a combinatorial formula for the coefficients $a_\lambda^{\nu(n)}$ appearing above; it is sometimes called ``the restriction problem''. Littlewood's results in Fact~\ref{fact:2} reformulates the question in terms of plethysm, but a combinatorial formula is not known for the RHS of~\eqref{eq:plethysm}. However, partial results are known for very special cases. We do not give an exhaustive list here, but some recent progress for the cases of the trivial and the sign representations can be found in \cite[Theorems B and C]{NPPS2022}.  In another direction, a symmetric function approach to the problem was given by Orellana and Zabrocki \cite{OZ2021}. 
\end{rem}

 While they are not mentioned explicitly in our main theorem stated in \cref{s:prelim}, the coefficients  $a_\lambda^{\nu(n)}$ appear  in \Cref{thm:equivariantTors}, and we use them to describe the $\sym_n$-equivariant structure of the minimal free resolution of a general principal symmetric ideal. This will be done, in part,  by means of \cref{lem:Lirreps} below, which provides 
the $\sym_n$-equivariant structure of the resolutions of powers of the maximal ideal. 

\begin{lem}
\label{lem:Lirreps}
The $\sym_n$-equivariant structure for the minimal free resolution of $\m^d$  is given by
\begin{equation}
\label{eq:tormdequivsmalli}
\Tor_i^R(\m^d,\kk)= \bigoplus_{\substack{|\nu|\leq i+d\\\nu_1+|\nu|\leq n}}\left(\Sp_{\nu(n)}\right)^{\oplus a_{(d,1^i)}^{\nu(n)}},
\end{equation}
where the multiplicities are given by the plethysm \eqref{eq:plethysm}.
\end{lem}

\begin{proof}
The desired  $\sym_n$-irreducible decomposition \eqref{eq:tormdequivsmalli} follows from \eqref{eq:LasSchur} and  \eqref{eq:res}. To see the restrictions on $\nu$ in the summation, note that if $|(d, 1^{i})|=i+d<|\nu|$ then \Cref{fact:2}~(3) implies that the multiplicity of  $\Sp_{\nu(n)}$  in   \eqref{eq:tormdequivsmalli} is zero.
\end{proof}

\section{Duality and compressed algebras}
\label{s:duality}

We give here  an exposition of  the Macaulay inverse system and Matlis/local duality, as needed toward our main results. We use Bruns and Herzog \cite{BH} as a reference for standard duality facts.  

In this section, $R=\kk[x_1, \ldots, x_n]$ is always taken to be a polynomial ring over a field $\kk$, with $n$ variables in degree $1$.  When regarded as a graded ring itself, $\kk$ is considered to be concentrated in degree zero.  We work with graded $R$ or  $\kk$-modules.  When we write $\Hom_R(-,-)$, $\Hom_\kk(-,-)$, $\Ext_R(-,-)$ or $\Tor^R(-,-)$, we understand the graded versions of the functors, as described in \cite[Section 1.5]{BH} and explained further in \Cref{appendix}. 
In particular, if $U$, $V$ are graded $R$-modules, then $\Tor^R(U,V)$ is a bigraded module, and $\Tor^R_i(U,V)_j$  denotes the component in homological degree $i$ and internal degree $j$. 

\subsection{Duality}
If $U$ is graded $R$-module, we define the graded $R$-module
$$
U^* :=\Hom_R(U,R) \,.
$$
On the other hand, viewing $U$ as a  graded $\kk$-module, we may also define the graded $\kk$-module
$$
U^\vee :=\Hom_\kk(U,\kk) \qquad \text{where}\qquad  (U^\vee)_i=\Hom_\kk(U_{-i},\kk)
$$
for all $i$.  Note that $U^\vee$ has a graded $R$-module structure given by 
\begin{equation}
\label{e:module-structure}
(a \, \varphi)(u) := \varphi(au)\qquad\text{for $a\in R$, $\varphi\in U^\vee$, $u\in U$\,.}
\end{equation}

We recall below the results that allow us to use (graded) Matlis duality.  Since $R = \kk[x_1,\ldots,x_n]$, considered with its (graded) maximal ideal $\mathfrak{m} = (x_1,\ldots,x_n)$, is a Noetherian complete $^*$local ring (in the sense of \cite[1.5.13]{BH}) and is Cohen-Macaulay of dimension $n$, the following holds. 

\begin{prop}[{\cite[Proposition 3.6.16, Example 3.6.10, Theorem 3.6.17(c), Theorem 3.6.19]{BH}}
\label{p:BHresults}]
With notation as above, 
\begin{enumerate}[\quad\rm(1)]
\item $R^\vee$ is the (graded) injective hull of $\kk$.
\item  $U^\vee\cong \Hom_R(U, R^\vee)$. 
\item The functor $(-)^\vee$ establishes an anti-equivalence between the category of graded artinian $R$-modules and that of finitely generated graded $R$-modules. In particular, the functor $(-)^\vee$ is exact,  and $U$ is artinian if and only if $U^\vee$ is finitely generated. 
\item If  $U$ is artinian,  then  $U^\vee\cong \Ext^n_R(U, \omega_R)$, where $\omega_R=R(-n)$ is the (graded) canonical module of $R$.
\end{enumerate} 
\end{prop} 

We set $S =R^\vee$. Recall that $S_i=\Hom_\kk(R_{-i},\kk)$ for all $i$. 
Let  $y_1^{(e_1)}y_2^{(e_2)}\cdots y_n^{(e_n)}$ denote the basis element of $S$ that is dual to the monomial $x_1^{e_1}x_2^{e_2}\dots x_n^{e_n}$ with $e_i\ge 0$ for all $i$, with the convention that we will write $y_i$ instead of $y_i^{(1)}$ and we may omit $y_i^{(e_i)}$ when $e_i=0$, unless $e_i=0$ for all $i$, in which case we will write $1$ for the basis element dual to  $1\in R$ (e.g.~ we may write $y_1y_3$ instead of $y_1^{(1)}y_2^{(0)}y_3^{(1)}$).\footnote{
Although we are not making use of a multiplicative structure on $S$, we note that $S$ can be identified to either the ring of inverse polynomials $\kk[x_1^{-1}, \dots, x_n^{-1}]$ or to a divided power algebra $D^\kk(V^*)$ where $V=\text{Span}_\kk(x_1, \dots, x_n)$.  The notation that we adopted for the dual basis is in line with the latter. } The $R$-module structure of $S$ can be described equivalently  as a {\it contraction}, defined by extending the following action on monomials by linearity in both arguments: 
\begin{equation}
\label{eq:contraction}
x_1^{d_1}\cdots x_n^{d_n}\circ y_1^{(e_1)}\cdots y_n^{(e_n)}=
\begin{cases}
y_1^{(e_1-d_1)}\cdots y_n^{(e_n-d_n)} & \text{if } e_i-d_i\geq 0 \text{ for all } i\\
0 & \text{otherwise\,. }
\end{cases}
\end{equation}
We emphasize that we regard $S$ as a  {\it graded} injective hull of $\kk$, and, as such, the basis element $y_1^{(e_1)}\cdots y_n^{(e_n)}$ has degree equal to $-\sum_{i=1}^n e_i$. 

As discussed in \cref{rem:MGRproperties}(4), the natural action of $\sym_n$ on $S$ is given by 
$$
\sigma\cdot s(t)=s(\sigma^{-1}\cdot t)\qquad\text{for $\sigma\in \sym_n$, $t\in R$, $s\in S$\,.}
$$
We record two important aspects of this action below.

\begin{lem}
\label{l:actionS}
With notation as above, the following hold. 
\begin{enumerate}[\quad\rm(1)]
\item For each dual basis element $y_1^{(e_1)}\cdots y_n^{(e_n)}\in S$ and $\sigma\in \sym_n$, 
\begin{equation*}
\label{e:Spermute}
\sigma\cdot  (y_1^{(e_1)}\cdots y_n^{(e_n)})= y_{\sigma(1)}^{(e_1)}\cdots y_{\sigma(n)}^{(e_n)}\,.
\end{equation*}
\item For all $i\ge 0$, $\sigma\in\sym_n$, $r\in R_{i}$ and $s\in S_j$, we have: 
$$
\sigma\cdot (r\circ s)=r\circ s   \qquad\text{when $i+j\ge 0$}\,.
$$
\end{enumerate}
\end{lem}
\begin{proof}
Indeed, to show (1), it suffices to prove the equality when evaluating at an arbitrary basis element $x_1^{d_1}\cdots x_n^{d_n}$ of $R$. We have: 
\begin{align*}
(\sigma\cdot y_1^{(e_1)}\cdots y_n^{(e_n)})(x_1^{d_1}\cdots x_n^{d_n})&= ( y_1^{(e_1)}\cdots y_n^{(e_n)})(\sigma^{-1}\cdot x_1^{d_1}\cdots x_n^{d_n})\\
&=(y_1^{(e_1)}\cdots y_n^{(e_n)})( x_{\sigma^{-1}(1)}^{d_1}\cdots x_{\sigma^{-1}(n)}^{d_n})\\
&=(y_1^{(e_1)}\cdots y_n^{(e_n)})( x_{1}^{d_{\sigma(1)}}\cdots x_{n}^{d_{\sigma(n)}})\\
&=\begin{cases}
1& \text{if $e_i=d_{\sigma(i)}$  for all $i$}\\
0 &\text{otherwise. }
\end{cases}
\end{align*}
On the other hand, we have 
\begin{align*}
( y_{\sigma(1)}^{(e_1)}\cdots y_{\sigma(n)}^{(e_n)})(x_1^{d_1}\cdots x_n^{d_n})&=( y_1^{(e_{\sigma^{-1}(1)})}\cdots y_{n}^{(e_{\sigma^{-1}(n)})})(x_1^{d_1}\cdots x_n^{d_n})=\begin{cases}
1& \text{if $e_{\sigma^{-1}(i)}=d_{i}$  for all $i$}\\
0 &\text{otherwise. }
\end{cases}
\end{align*}
Since $e_{\sigma^{-1}(i)}=d_{i}$ for all $i$ if and only if $e_i=d_{\sigma(i)}$ for all $i$, these two computations prove (1).  

To explain (2), observe that $r\circ s\in S_{i+j}$ and $S_{i+j}=0$ when $i+j>0$, while $S_{i+j}=\kk$ when $i+j=0$. Thus the action of $\sym_n$ on $S_{i+j}$ is trivial in these cases. 
\end{proof}

We expand the notation introduced for modules to complexes, as well. 

\begin{notation} For any commutative ring $Q$, if  $X$ is a complex of $Q$-modules and $M$ is a $Q$-module,  then we denote by $\Hom_Q(X, M)$ the complex $Y$ of $Q$-modules, where 
\[
Y_i :=\Hom_Q(X_{-i},M)\qquad\text{and}\qquad (\partial^{Y}(f))(x) :=(-1)^i f(\partial^X(x))
\]
for $f\in Y_i$ and $x\in X_{-i+1}$.  We set 
\[
{X}^* :=\Hom_R(X, R)\qquad\text{and}\qquad {X}^\vee :=\Hom_\kk(X, \kk)\,
\]
as complexes, 
noting that the latter is a complex of $R$-modules as well, and \cref{p:BHresults}(2) gives
\begin{equation}
\label{e:iso}
{X}^\vee \cong \Hom_R(X, S)\,.
\end{equation}
Also, for $a\in \mathbb Z$, we let $\Sigma^a X$ denote the complex with 
\[
(\Sigma^a X)_i:=X_{i-a}\qquad\text{and}\qquad \partial^{\Sigma^a X} :=(-1)^a\partial^{X}.
\]
Observe that 
\begin{equation}
\label{e:switch}
(\Sigma^aX)^\vee = \Sigma^{-a}({X}^\vee)\,. 
\end{equation}
\end{notation}

We reproduce below, with more detail,  part of the statement of \cite[Proposition 2.2]{Boij}, which is a direct consequence of the isomorphisms in \cref{p:BHresults}. 

\begin{prop}{\cite[Proposition 2.2 (iv),(v)]{Boij}}
\label{Boij}
Let $U$ be an artinian $R$-module and suppose 
$$
F: \quad F_0\xleftarrow{\partial_0} F_1 \cdots \leftarrow F_i\xleftarrow{\partial_i} F_{i+1}\leftarrow\cdots\leftarrow F_n\leftarrow 0
$$
is a graded minimal free resolution of $U$ over $R$. Then 
$$
\Sigma^{n}\Hom_R(F,\omega_R) :\quad  {F_n}^*(-n)\leftarrow\cdots \leftarrow {F_{i+1}}^*(-n)\leftarrow {F_{i}}^*(-n)\leftarrow\cdots\leftarrow {F_0}^*(-n)\leftarrow 0
$$
is a graded minimal free resolution of $U^\vee$ over $R$. In particular 
$$
\Tor_i^R(U,\kk)_j \cong 
\Hom_\kk(\Tor_{n-i}^R(U^\vee,\kk)_{n-j},\kk)
$$
for all $i$ with $0\le i\le n$ and all $j\in \mathbb Z$. 
\end{prop}

We will need an equivariant version of the previous result which we state in \Cref{equivariantBoij}.  This differs from \Cref{Boij} by a twist given by the 1-dimensional representation $ \bw^n_{\kk} R_1$, denoted $\Sp_{(1^n)}$ in~ \Cref{ex:altrep}. 

\begin{prop}
\label{equivariantBoij}
Let $R={\kk}[x_1,\ldots, x_n]$ and assume ${\kk}$ is a field of characteristic zero. Suppose that $U$ is an object of $\MG R$ which is an artinian $R$-module. Then there is an isomorphism in $\MG R$
\[ \Tor_i^R(\kk,U)_j \cong
 \bw^n_{\kk} R_1  \otimes_{\kk}\left(\Tor_{n-i}^R( \kk, U^\vee)_{n-j}\right)^\vee\,.\]
\end{prop}

\begin{proof}
As noted above, a projective resolution of $\kk$ in $\MG R$ is given by the Koszul complex $K$.  Since $R$ is Cohen-Macaulay of dimension $n$, we know that ${K}^*$ is exact, except in  homological degree $-n$, see  \cite[Theorem 1.2.8]{BH}. Since ${K_i}^*=\Hom_R(K_i,R)$ is a free $R$-module, we  conclude that the complex  $\Sigma^{n}{K}^*$ is a projective resolution of $\Ext^n_R(\kk,R)$ in $\MG R$.  

Moreover, there is an isomorphism $\Ext^n_R(\kk,R)\cong (\bw^n_{\kk} (R_1))^\vee$ in $\MG R$, since the former module fits into the equivariant exact sequence
\[
\Hom_R(K_{n-1},R)\xrightarrow{\eta}\Hom_R(K_{n},R)\to \Ext^n_R(\kk,R)\to 0\,,
\]
where, with respect to appropriate bases, $\eta=[ x_1 \,\, -x_2 \, \,\, x_3 \,\,  -x_4 \,\, \ldots \,\, (-1)^{n+1} x_n]$, and hence
\[
\Ext^n_R(\kk,R)\cong {K_{n}}^*/\fm {K_n}^*\cong \left(\bw^n_{\kk} (R_1)\right)^\vee.
\]
Consider now $({K_i}^*  \otimes_R U^\vee)^\vee$. We have isomorphisms in $\MG R$
\begin{eqnarray}
({K_i}^* \otimes_R U^\vee)^\vee & \cong&\Hom_R({K_i}^* \otimes_R \Hom_R(U,S), S)  \label{E0} \\
&\cong& \Hom_R({K_i}^*, \Hom_R(\Hom_R(U,S), S)) \label{E1}\\
&\cong& \Hom_R({K_i}^*, U) \label{E2}\\
 & \cong& {{K_i}^*}^*\otimes _R U \label{E3}\\
 &\cong & K_i\otimes_R U  \label{E4}
\end{eqnarray}
that are justified as follows: \eqref{E0}  follows from \eqref{e:iso}, \eqref{E1} is by hom-tensor adjunction, \eqref{E2} is by \eqref{e:iso} and equivariant Matlis duality \cite[Lemma 3.8]{Witt}, \eqref{E3} is by \Cref{lem:freeisos} (1) and \eqref{E4} is by \Cref{lem:freeisos} (2). 
All these isomorphisms are natural, so they induce an equivariant isomorphism of complexes $({K}^*  \otimes_R U^\vee)^\vee\cong K \otimes_R U $. In view of \eqref{e:switch}, we also have  $(\Sigma^n{K}^*  \otimes_R U^\vee)^\vee\cong \Sigma^{-n}K \otimes_R U$. 
Since the functor $(-)^\vee$ is exact, we obtain 
\begin{align*}
\left(\Tor_{n-i}^R((\bw^n_{\kk} R_1)^\vee, U^\vee)\right)^\vee\cong H_{i-n}(( \Sigma^{n}{K}^* \otimes_R U^\vee )^\vee)\cong H_{i-n}(\Sigma^{-n}K \otimes_R U)=\Tor_i^R(\kk,U). 
\end{align*}
Since $(\bw^n_{\kk} (R_1))^\vee\otimes_{\kk} -$  is an exact functor, we further have 
\[
\Tor_{n-i}^R( (\bw^n_{\kk} R_1)^\vee, U^\vee)=\Tor_{n-i}^R( (\bw^n_{\kk} R_1)^\vee\otimes_{\kk} \kk, U^\vee)=(\bw^n_{\kk} R_1)^\vee\otimes_{\kk} \Tor_{n-i}^R(  \kk, U^\vee)
\]
whence 
\begin{eqnarray*}
\left(\Tor_{n-i}^R( (\bw^n_{\kk} R_1)^\vee, U^\vee)\right)^\vee 
&\cong&
 {\left(\bw^n_{\kk} R_1\right)^\vee}^\vee \otimes_{\kk} \left(\Tor_{n-i}^R(  \kk, U^\vee)\right)^\vee \\
 &\cong& \bw^n_{\kk} R_1 \otimes_{\kk} \left(\Tor_{n-i}^R(  \kk, U^\vee)\right)^\vee.
\end{eqnarray*}

Tracing through the isomorphisms above, where $\bw^n_{\kk} R_1$ is concentrated in internal degree $n$, we see that the degree $j$ component of $\left(\Tor_{n-i}^R( (\bw^n_{\kk} R_1)^\vee, U^\vee)\right)^\vee $ is isomorphic to  
\[
\left(\bw^n_{\kk} R_1 \right)_n \otimes_{\kk}\left( \left(\Tor_{n-i}^R(  \kk, U^\vee)\right)^\vee\right)_{j-n}=\bw^n_{\kk} R_1\otimes_{\kk} \left(\Tor_{n-i}^R(  \kk, U^\vee)_{n-j}\right)^\vee.  \qedhere
\] 
\end{proof}

Another way to recover the twist by $\bw^n_{\kk} R_1$ in \cref{equivariantBoij} is by using \cref{Boij} and the canonical description of $\omega_R$ as $\bw^n_{\kk} R_1\otimes_\kk R$ .

\subsection{Macaulay inverse systems}
\begin{defn}
Let $I$ be a homogeneous ideal of $R$. The {\it Macaulay inverse system} of $I$ is defined as the graded $R$-submodule of $S$ given by 
\begin{equation}
\label{eq:Iperp}
I^\perp:=\Ann_S(I)=\{ g\in S : f\circ g=0 \text{ for all } f\in I\}. 
\end{equation}
\end{defn}

\begin{rem} 
Since $\Ann_S(I)\cong \Hom_R(R/I, S)$, we see from \cref{p:BHresults} that 
\[
I^\perp\cong (R/I)^\vee
\]
and $I^\perp$ is finitely generated if and only if $R/I$ is artinian. 

If $W$ is a graded $R$-submodule of $S$, then $W^\vee\cong \Hom_R(W,S)\cong R/\Ann_R(W)$, where  
$
\Ann_R(W)=\{f\in R: f\circ g=0 \mbox{ for all } g\in W\}
$ is a homogeneous  ideal of $R$. 
It is thus a consequence of (graded) Matlis duality that 
\begin{equation}
\label{e:duality}
 \Ann_R(I^\perp)=I \qquad \text{and}\qquad (\Ann_R(W))^\perp=W\,.
 \end{equation}
\end{rem}

Let $U$ be a finitely generated graded $R$-module. The {\it Hilbert function} of $U$ is the  function 
$$\HF_U\colon \mathbb Z\to \mathbb N\,,\qquad \text{with}\qquad \HF_U(i)=\dim_\kk U_i\quad\text {for all $i\in \mathbb Z$}\,.
$$ 
We also use the {\it Hilbert series} of $U$, defined as 
$$
\H_U(z)=\sum_{i\in \mathbb Z}(\HF_U(i))z^i\,.
$$
The {\it generator type} of $U$ is the polynomial $G_U=\HF_{U/R_{\geqslant 1}U}$.

\begin{defn}
\label{def:t}
We define the {\it initial degree} of $U$ to be the integer
\[
t(U) :=\min\{i\in \mathbb Z\colon \HF_U(i)\ne 0\}\,.
\]
\end{defn}

Assume $I$ is a homogeneous ideal  of $R$ and set $A=R/I$. 
We refer to $A$ as an {\it algebra quotient} of $R$.  When using this terminology, it is implied that $I$ is homogeneous.

\subsection{Compressed artinian algebras}
Assume further, for the remainder of the section,  that $A$ is  an artinian $R$-algebra. 
We let $\m$ denote the ideal $R_{\geqslant 1}=(x_1, \ldots, x_n)$ and $\m_A$ its image in $A$.  The {\it socle} of $A$ is the graded $A$-module $$\Soc(A) :=\Ann_A(\m_A)\cong \Hom_A(\kk, A)\,.$$
The Hilbert function of the socle is called the {\it socle type of $A$.} We refer to the Hilbert series $\H_{\Soc(A)}(z)$ of $\Soc(A)$ as the {\it socle polynomial} of $A$. Since $\Soc(A)$ is a finite dimensional $\kk$-vector space, this is indeed a polynomial. 
The Hilbert function and the socle type of $A$ can be read off the Macaulay inverse system of $I$.
Namely, we have
\begin{gather}
\HF_{I^\perp}(-i)=\dim_\kk (I^\perp)_{-i}=\dim_\kk A_i =\HF_A(i)\\
\label{eq:GIperp}
G_{I^\perp}(-i)=\dim_\kk\left(\frac{I^\perp}{\m\circ I^\perp}\right)_{-i}=\dim_\kk \Soc(A)_i=\HF_{\Soc(A)}(i)
\end{gather}
where the first equation follows from the isomorphism $I^\perp\cong A^\vee = (R/I)^\vee$ and the second is due to the isomorphisms 
$$(I^\perp/\m\circ  I^\perp)^\vee \cong \Soc((I^\perp)^\vee)\cong \Soc(A)\,,$$
which follow from \cite[Proposition 3.6.17]{BH} and (the graded version of) the proof of \cite[Proposition 3.2.12(d)]{BH}. 

\begin{rem}
\label{r:cyclic}
$A$ is Gorenstein if and only if $I^\perp$ is a cyclic $R$-module. Indeed, since $A$ is artinian, it is Gorenstein if and only if its socle is $1$-dimensional, and the statement follows from \eqref{eq:GIperp}. 
\end{rem}

\begin{defn}
\label{def:socle}
The {\it top socle degree} of $A$ is defined to be the integer
\[
s(A) :=\max\{i\ge 0\colon\Soc(A)_i\ne 0\}\,.
\]
\end{defn}
Set $s=s(A)$ and let $\sum_{i=0}^se_iz^i$ be the socle polynomial of $A$. In view of \eqref{eq:GIperp}, the isomorphism $I^\perp\cong A^\vee$ gives an inequality: 
\begin{equation}
\label{e:comp}
H_A(i)\le \min\big\{\dim_\kk R_i, \sum_{j=i}^se_j\dim_\kk R_{j-i} \big\}\qquad\text{for all $i$ with $0\le i\le s$.}
\end{equation}
Compressed algebras were defined and first studied in \cite{Iar} as algebras whose Hilbert series is largest that is allowed by the socle polynomial. 
\begin{defn}
\label{def:compressed}
Set $s=s(A)$ and let $\sum_{i=0}^se_iz^i$ be the socle polynomial of $A$. We say that the algebra $A$ is {\it compressed}  if equality holds in \eqref{e:comp}. 
\end{defn}
There need not always exist an algebra having a given socle polynomial that is compressed in the sense of \Cref{def:compressed}. A sufficient (but not necessary) condition for the existence of such algebras is phrased in terms of permissible socle polynomials.

\begin{defn}
\label{d:permissible}Following \cite{Iar}, we say that a polynomial $p(z)=\sum_{i=t-1}^s e_iz^i$  is {\it permissible}  for a socle polynomial of a compressed algebra $A=R/I$, where $t(I)=t$ and $s(A)=s$ if the following conditions are satisfied: 
\begin{enumerate}[\quad(i)]
\item $e_i\ge 0$ for all $i$ and $e_s>0$;
\item $\sum_{i\geqslant t} e_i\dim_\kk R_{i-t}<\dim_\kk R_t$;
\item $e_{t-1}=\max(0, R_{t-1}-\sum_{i\geqslant t}e_i\dim_\kk R_{i-(t-1)})$.
\end{enumerate}
\end{defn}

If $p(z)$ is permissible for a socle polynomial,  then all artinian algebra quotients $A$ of $R$ that have $p(z)$ as their socle polynomial can be parametrized by a certain space, and a general $A$ (with respect to this parameter space) is known to be compressed. Consult  the literature, e.g.~\cite{Iar} or ~\cite{FL}, for more details.

When $A$ is compressed, one can use the socle polynomial of $A$ and duality  to derive information about the betti numbers of $A$ over $R$. For example, when the socle is concentrated in a single degree (i.e.~the algebra is {\it level}), the free resolution of a compressed algebra is concentrated in two strands, except for the beginning and the end, see \cite{Boij}. Also, if $A$ is compressed with permissible socle polynomial $p(z)=\sum_{i=t-1}^s e_iz^i$ as above, and 
$$
\sum_{i}e_i\dim_\kk R_{i-t-1}=\dim_\kk R_{t-1}\,,
$$
then the minimal free resolution of $A$ over $R$ is {\it almost linear}, meaning that all matrices in a minimal free resolution of $A$  have entries in degree at most 1, except for the first and the last, see \cite[Proposition 4.1A]{Iar}. 

Finally, observe that the invariants $s(A)$ and $t(I)$ impose bounds on the betti numbers of $A$ as follows. 
\begin{lem}
\label{st}
If $A$ is an artinian algebra quotient of $R$, then for all $i>0$: 
\begin{enumerate}[\quad\rm(1)]
\item $\beta_{i,j}^R(A,\kk)=0$ for all $j$ with $j<i-1+t(I)$, and $\beta_{1,t(I)}^R(A,\kk)\ne 0$; 
\item $\beta_{i,j}^R(A,\kk)=0$ for all $i,j$ with $j>i+s(A)$, and  $\beta_{n,n+s(A)}^R(A,\kk)\ne 0$. 
\end{enumerate}
In particular, $t(I)\le s(A)+1$, where $s(A)$ equals the regularity of $A$. 
\end{lem}
\begin{proof}
If $M$ is a finitely generated graded $R$-module,  observe that  $\Tor_i^R(M,\kk)_j=0$ for all $j<i+t(M)$, and $\Tor_0^R(M,\kk)_j\ne 0$ for $j=t(M)$. This can be seen by inspecting the degree shifts in a graded minimal free resolution of $M$ over $R$. Part (1) follows by using this observation with $M=I$ and the isomorphism $\Tor_i^R(A,\kk)_j\cong \Tor_{i-1}^R(I, \kk)_j$. 

Using \cref{Boij} and the isomorphism $A^\vee \cong I^\perp$, we have 
$$
\Tor_i^R(A,\kk)_j\cong \Tor_{n-i}^R(I^\perp,\kk)_{n-j}\,.
$$
Since $t(I^\perp)=s(A)$ by \eqref{eq:GIperp}, (2) follows from the same observation, with $M=I^\perp$. 
\end{proof}

\section{Narrow and extremely narrow algebras}
\label{s:narrow}

We use the notation introduced in Section \ref{s:duality}. While we do not make a blanket assumption that $A=R/I$ is artinian, this will be implied whenever $s(A)<\infty$. We discuss  a class of compressed $R$-algebras for which we can use duality to determine their betti numbers over $R$.  In general, \Cref{st} gives $t(I)\le s(A)+1$. We study algebras that are extremal with respect to this inequality.

\begin{defn}
\label{def:narrow}
Let $A=R/I$ be an artinian algebra quotient of $R$. We say that $A$ is {\it narrow} if $t(I)\ge s(A)$, that is, the top socle degree of $A$ is at most the initial degree of $I$.
\end{defn}

While we use different terminology, the idea of studying  ideals that are extremal  with respect to the inequality $t(I)\le s(A)+1$ is not new. When $A$ is artinian, $t(I)=s(A)+1$ means  that $A$ is  {\it extremal Cohen-Macaulay}, and  $t(I)=s(A)$ means that $A$ is {\it nearly extremal Cohen-Macaulay}, in the terminology of \cite{KSK}. In particular, some of our results below share some overlap with those in \cite{KSK}.

We start with equivalent characterizations of narrowness in terms of the Macaulay inverse system and the betti numbers. \cref{l:narrow-equiv1} indicates, in particular,  how one can construct examples, using the inverse system. In what follows, $\langle F_1, \dots, F_a\rangle$ denotes the $\kk$-vector space spanned by $F_1, \dots, F_a$ and $(F_1, \dots, F_a)$ denotes the $R$-submodule of $S$ generated by $F_1, \dots, F_a$. 

\begin{lem}
\label{l:narrow-equiv1}
Let $d\ge 1$, $a\ge 1$ and $F_1, \dots, F_a\in S_{-d}$. Let $A=R/I$ be an algebra quotient of $R$.  The following are equivalent: 
\begin{enumerate}[\quad\rm(1)]
\item $A$ is narrow with $s(A)=d$ and $(I^\perp)_{-d}=\langle F_1, \dots, F_a\rangle$;
\item $I^\perp=(F_1 \dots, F_a)+S_{\geqslant-(d-1)}$; 
\item $I=\Ann_R\left((F_1 \dots, F_a)+S_{\geqslant-(d-1)}\right)$.
\end{enumerate}
 If these conditions hold, then we also have 
\[
I=\Ann_R(F_1, \dots, F_a)\cap \m^d=(V)+\m^{d+1}\,
\]
where $V=\Ann_R(F_1, \dots, F_a)\cap R_d$.
\end{lem}
\begin{proof}
Note that $(I^\perp)_{-d}=\langle F_1, \dots, F_a\rangle$ is a consequence of both (1) and (2), and thus the equivalence $(1)\iff (2)$ is established through the following sequence of equivalent statements: 
\begin{align*}
t(I)\ge d&\iff \dim_\kk A_{d-1}=\dim_\kk R_{d-1}\\
&\iff (I^\perp)_{-(d-1)}=S_{-(d-1)}\\
&\iff (I^\perp)=(F_1, \dots, F_a)+S_{\geqslant-(d-1)}\,.
\end{align*}
The equivalence $(2)\iff(3)$ is a consequence of duality, cf.~\eqref{e:duality}. 

Further, we  have 
\begin{align*}
\Ann_R\left((F_1 \dots, F_a)+S_{\geqslant-(d-1)}\right)&=\Ann_R(F_1 \dots, F_a)\cap \Ann_R(S_{\geqslant-(d-1)})\\
&=\Ann_R(F_1 \dots, F_a)\cap \m^d\\
&=(V)+\m^{d+1}
\end{align*}
where the last equality is justified by the inclusion $\m^{d+1}\subseteq \Ann_R(F_1, \dots, F_a)$, which is in turn implied by the inclusion $(F_1, \dots, F_a)\subseteq S_{\geqslant -d}$.
\end{proof}

\cref{l:narrow-equiv2} justifies the terminology ``narrow" as applicable to algebras whose nonzero betti numbers are concentrated in only two adjacent degrees with the exception of the $0$th homological degree. 

\begin{lem}
\label{l:narrow-equiv2}
Let $d\ge 1$ and let $A=R/I$ be an algebra quotient of $R$ with $s(A)=d$.  The following are then equivalent: 
\begin{enumerate}[\quad\rm(1)]
\item $A$ is narrow;
\item $\beta_{i,j}(A)=0$ for all $i,j$ with $j\notin \{i+d, i+d-1\}$ and $i>0$.
\end{enumerate}
\end{lem}

\begin{proof}
This follows directly from \cref{st}. 
\end{proof}

\begin{notation}
\label{LA}
Let $F_1, \dots, F_a\in (I^\perp)_{-d}$   and define 
$$
L_{F_1, \ldots, F_a}:=\left \{(\ell_1, \ell_2, \dots, \ell_a)\in {R_1}^a\colon \sum_{i=1}^a \ell_i\circ F_i=0 \right \} \,.
$$
When $ F_1, \dots, F_a$ is a basis of $(I^\perp)_{-d}$, the $\kk$-vector space $L_{F_1, \ldots, F_a}$  is independent, up to isomorphism, of the choice of $F_1, \dots, F_a$ and we will also denote it by $L_A$ when appropriate. Note that we have an exact sequence of $\kk$-vector spaces
\begin{equation}
\label{eq:Lses}
0\to L_{F_1, \dots, F_a}\hookrightarrow{R_1}^a\xrightarrow{\psi} R_1\circ (F_1, \dots, F_a)_{-d}\to 0
\end{equation}
with $\psi(\ell_1, \ell_2, \dots, \ell_a)=\sum_{i=1}^a \ell_i\circ F_i$, and, since $\dim_\kk R_1=n$, we have  
\begin{equation}
\label{dimL_A}
\dim_\kk(L_A)=an-\dim_\kk \left( R_1\circ {(I^\perp)}_{-d} \right).
\end{equation}
\end{notation}

We now list several properties of narrow algebras.  \cref{l:narrow} establishes, in particular,  that narrow algebras are instances of compressed algebras, but their socle polynomial may not be permissible. 

\begin{lem}
\label{l:narrow}
Let $d\ge 1$, $a\ge 1$. If $A$ is a narrow algebra with $s(A)=d$ and $\dim_\kk{(I^\perp)}_{-d}=a$, then the following hold: 
\begin{enumerate}[\quad \rm(1)]
 \item The socle polynomial of $A$ is $bz^{d-1}+az^{d}$, where 
\begin{align*}
b= \dim_\kk R_{d-1}&-\dim_{\kk} \left(R_1\circ (I^\perp)_{-d}\right)=\dim_\kk R_{d-1}-an+\dim_\kk L_{A}\quad\text{and}\\
&\dim_\kk R_{d-1}>b\ge \dim_\kk R_{d-1}-an\,.
\end{align*}
 \item $A$ is compressed and has Hilbert function
  $$
 \HF_A(i)=\begin{cases}\dim_\kk R_i &\text{if $i\le d-1$.}\\
 a&\text{if $i=d$}\\
 0& \text{if $i>d$}.
 \end{cases}
 $$
 \item We have $t(I)\in \{d, d+1\}$, and
\[ t(I)=d+1\,\iff\, I=\m^{d+1} \,\iff\, a=\dim_\kk R_d\,.
\]
 \item If $I\ne \m^{d+1}$, then the following are equivalent: 
 \begin{enumerate}[\quad\rm(a)]
 \item The socle polynomial of $A$ is permissible;
 \item $b=\dim_\kk R_{d-1}-an$; 
 \item $L_A=0$. 
  \end{enumerate}
 \item Let $F_1, \dots, F_a\in S_{-d}$ such that  ${(I^\perp)}_{-d}=\langle F_1, \ldots, F_a\rangle$ and define $\varphi\colon \m^d\to \kk^a(-d)$ by 
$$
\varphi(r)=(r\circ F_1, r\circ F_2, \dots, r\circ F_a)\qquad\text{for $r\in \m^d$\,.}
$$
Then $\varphi$ is surjective and $\Ker(\varphi)=I$.  Furthermore, if $F_i$ are invariant with respect to the action of $\sym_n$ on $S$, then $\varphi$ is $\sym_n$-equivariant, when considering $\kk^a(-d)$ to be endowed with the trivial $\sym_n$-structure. 
 \end{enumerate}
\end{lem}

\begin{proof}
Write ${(I^\perp)}_{-d}=\langle F_1, \ldots, F_a\rangle$ as in (5). Since $\dim_\kk(I^\perp)_{-d}=a$, note that $F_1, \ldots, F_a$ are linearly independent. By \cref{l:narrow-equiv1}, we have  
\[I^\perp=(F_1 \dots, F_a)+S_{\geqslant-(d-1)}\,.
\]
In particular, $I^\perp$ does not have minimal generators in degrees other than $-d$ and $-(d-1)$, and hence $\Soc(A)$ is concentrated in degrees $d$ and $d-1$ by \eqref{eq:GIperp}. 
More precisely, $b :=\dim_\kk \Soc(A)_{d-1}$ is equal to the number of minimal generators of $I^\perp$ in degree $-(d-1)$, and this justifies the first equality in (1). The second equality in (1) follows from \eqref{dimL_A}. 
The inequalities in (1) follow from the preceding equalities, noting that $\dim_\kk L_A\ge 0$ and $R_1\circ (I^\perp)_{-d}\ne 0$, since $a>0$ and $d>0$.  

Since $\Soc(A)_i=0$ for $i>d$ and $\dim_\kk \Soc(A)_d=a\ne 0$, we must have $\Soc(A)_d=A_d$ and $A_i=0$ for $i>d$. Since $t(I)\ge d$, we also have $A_i=R_i$ for $i<d$. These remarks justify the expression for the Hilbert series in (2). To check that $A$ is compressed, we need to verify that equality holds in \eqref{e:comp} for all $i$ with $0\le i\le d$. When $i=d$, the equality is a consequence of the inequality $ a\le \dim_\kk R_d$. When $i\le d-1$, then equality holds in \eqref{e:comp} because $\HF_A(i)=\dim_\kk R_i$. 

 For (3), use \cref{st} to see $t(I)\in \{d,d+1\}$. If  $t(I)=d+1$, then $A_d=R_d$ and, since $A_{d+1}=0$, we must have $I_{d+1}=R_{d+1}$ and thus $I=\m^{d+1}$. If $a=\dim_\kk R_d$, then we must have $R_d=A_d$, and thus $t(I)>d$, hence $t(I)=d+1$. The remaining implications are clear. 

To prove (4), assume $I\ne \m^{d+1}$, so that $t(I)=d$ by (3). Note that condition (iii) in \cref{d:permissible} can be rewritten as
$b=\max\{0, \dim_\kk R_{d-1}-an\}$. In view of (1), this holds if and only if $L_A=0$, if and only if $b=\dim_\kk R_{d-1}-an$. To finish the proof of (4), note that condition (i) in \cref{d:permissible} holds trivially, and condition (ii) requires  $a<\dim_\kk R_d$, and this holds by (3), since $a=\dim_\kk R_d$ implies $t(I)=d+1$. 

 We now prove (5). For each $\bu=(u_1, \dots, u_n)\in \mathbb N^n$, set $x^\bu=x_1^{u_1}x_2^{u_2}\cdots x_n^{u_n}$ and $y^{(\bu)}=y_1^{(u_1)}y_2^{(u_2)}\cdots y_n^{(u_n)}$. Set 
 $$
 \mathcal N_n^d=\{\bu=(u_1, \dots, u_n)\in \mathbb N^n\colon u_1+u_2+\cdots+u_n=d\}\,.
 $$
 and write  
 \[
 F_i=\sum_{\bu\in \mathcal N_n^d } c_{\bu, i}y^{(\bu)}\qquad \text{and}\qquad 
 r=\sum_{\bu\in \mathcal N_n^d} d_{\bu}x^{\bu}\in R_d\,.
 \]
Noting that $x^\bu\circ y^{(\bu)}=1$, we have 
 $$
r\circ F_i=\sum_{\bu} d_{\bu}c_{\bu, i}\in \kk\,.
$$
 It is clear from here that $\varphi$ is surjective because  $F_1, \dots, F_a$ are linearly independent. 
 To see that $I=\Ker(\varphi)$, we first observe that $I\subseteq \Ker(\varphi)$, because $F_i\in I^\perp$ for all $i$ with $1\le i\le a$. To establish equality, we see that $\HF_I=\HF_{\Ker(\varphi)}$. 
 Indeed, we have: 
 \begin{align*}
 \HF_{\Ker(\varphi)}(i)&=\dim_\kk(\m^d)_i-\dim_\kk(\kk^a(-d))_i=\begin{cases}
 \dim_\kk R_i &\text{if $i\ge d+1$}\\
 \dim_\kk R_d-a &\text{if $i=d$}\\
 0 &\text{if $i<d$}
 \end{cases}\\
 &=\HF_R(i)-\HF_A(i)=\HF_I(i)\,
 \end{align*}
 where the third equality is by part (2) above.
 To prove the final statement, assume $F_i$ are $\sym_n$-invariant and $\kk^a(-d)$ has a trivial $\sym_n$-structure. If $r\in \m^d$, we have
 \begin{align*}
 \varphi(\sigma\cdot r)&=\left((\sigma \cdot r)\circ F_1, \cdots, (\sigma\cdot r)\circ F_a\right)=\left((\sigma\cdot r)\circ (\sigma\cdot F_1), \cdots, (\sigma\cdot r)\circ (\sigma\cdot F_a)\right)\\
 &=\left(\sigma\cdot(r\circ F_1), \cdots, \sigma\cdot (r\circ F_a)\right)=\left(r\circ F_1, \cdots, r\circ F_a\right)=\varphi(r)=\sigma\cdot \varphi(r)\,,
 \end{align*}
 where in the second equality we used that $F_i$ are $\sym_n$-invariant, in the third equality  we used fact that $S\in \mod_{\sym_n}(R)$, in the fourth equality we used \cref{l:actionS}~(2) and in the last equality we used the assumption that $\kk^a(-d)$ carries a trivial $\sym_n$-structure.
\end{proof}

We specialize next the class of narrow algebras even further, by introducing a class of algebras for which we will be able to compute explicitly  the betti numbers. 

\begin{defn}
\label{def:extremelynarrow}
Let  $d\ge 1$ and let $A=R/I$ be an algebra quotient of $R$.  We say that $A$ is $d$-{\it extremely narrow} if the following hold:
\begin{enumerate}[(i)]
\item $s(A)= t(I) =d$;
\item there exists a basis $F_1, \dots, F_a$ of  $(I^\perp)_{-d}$ and $x\in R_1$ such that 
$L_{F_1, \ldots, F_a}\subseteq x{R_0}^a$. 
\end{enumerate} 
\end{defn}

The next result is instrumental in establishing identities for the betti numbers of extremely narrow algebras. 

\begin{lem}
\label{en-props}
Assume $A$ is $d$-extremely narrow. Then the following hold:
\begin{enumerate}[\quad\rm(1)]
\item $\Tor_i^R(I^\perp, \kk)_{i-d}=0$ for all $i\ge 2$, 
\item $\Tor_1^R(I^\perp, \kk)\cong L_A$, 
\item $\Tor_{i}^R(I,\kk)_{j}=0$ for all $i\le n-3$ and all $j\ne i+d$, and 
\item If $n\geq 3$, then $I$ is generated in degree $d$. 
\end{enumerate}
\end{lem}
\begin{proof}
Let $F_1, \ldots, F_a$ be a basis of ${(I^\perp)}_{-d}$ and complete $F_1, \dots, F_a$ to a minimal generating set of $I^\perp$ by adding a finite number, say $b$, of additional homogeneous generators. Note that these $b$ generators are in degree $-(d-1)$. Let $F={R(d)}^a\oplus {R(d-1)}^b$ and let $\varphi\colon F\to I^\perp$ be the map that sends a basis of $F$ to the full generating set consisting of $a+b$ many elements. Set $N=\text{Ker}(\varphi)$ and let  $L$ be the graded $R$-submodule generated by $L_{A}$ in ${R(1)}^a$. Note that $N$ is generated in degrees $\geqslant -d+1$, and $L$ is equal to the submodule of $N$ generated by its generators in degree $-d+1$.  We have an exact sequence
$$
0\to L\hookrightarrow N\to N/L\to 0\,,
$$
which yields the exact sequence
$$
\Tor_2^R(N/L,\kk )_{-d+2}\to \Tor_{1}^R(L,\kk )_{-d+2}\to \Tor_{1}^R(N,\kk )_{-d+2}\to \Tor_{1}^R(N/L,\kk )_{-d+2}\,.
$$
The definitions of $L$ and $N$  imply $N_{-d}=0$ and $(N/L)_{-d+1}=0$, and hence 
$$
\Tor_2^R(N/L,\kk )_{-d+2}=0=\Tor_{1}^R(N/L,\kk )_{-d+2}\,.
$$
In view of the exact sequence above, we conclude
\begin{equation}\label{isom-tors}
\Tor_{1}^R(L,\kk )_{-d+2}\cong \Tor_{1}^R(N,\kk )_{-d+2}\,.
\end{equation}
Now consider the short exact sequence, with $x$ as in \cref{def:extremelynarrow}(ii):
$$
0\to L\hookrightarrow{xF}\to xF/L\to 0\,.
$$
It yields an exact  sequence 
\begin{equation}\label{les}
\Tor_2^R(xF/L,\kk )_{-d+2}\to \Tor_1^R(L,\kk )_{-d+2}\to \Tor_1^R(xF,\kk )_{-d+2}\,.
\end{equation}
Since $xF$ is a free $R$-module, we have $\Tor_1^R(xF,\kk )=0$. 
We also have $(xF)_i=0$ for all $i< -d+1$, hence $\Tor_2^R(xF/L,\kk )_{-d+2}=0$. 

Equations \eqref{les} and \eqref{isom-tors} give then $\Tor_1^R(N,\kk )_{-d+2}=0$. In turn, given our definition of $N$ as a first syzygy of $I^\perp$, this implies 
$$\Tor_i^R(I^\perp,\kk )_{-d+i}\cong \Tor_{i-1}^R(N,\kk )_{-d+i}=0\qquad\text{for all $i\ge 2$. }$$

We also have 
$$
\Tor_1^R(I^\perp,\kk )_{-d+1}\cong  (N\otimes_R\kk)_{-d+1}\cong L\otimes_R\kk\cong L_A. 
$$
This concludes the proof of (1) and (2). 

To see (3), recall that $I^\perp\cong A^\vee$ and apply \cref{Boij} to (1) in order to obtain $\Tor_i^R(I,\kk)_{i+d+1}=0$ for $i\le n-3$. This observation, in addition to \cref{l:narrow-equiv2}, yields the claim, since $A$ is narrow.

Finally, (4) follows from (3) by taking $i=0$. Indeed, $\Tor_0^R(I,\kk)_{j}=0$ for $j\ge d+1$ implies that $I$ has no generators in degrees greater than $d$. 
\end{proof}

Part (3) of the above result shows that an extremely narrow algebra $A=R/I$ is defined by an ideal $I$ such that the differentials in the resolution of $I$ are given by matrices with linear entries except possibly for the differentials in homological degrees $n-3$ and $n-2$.

The following will be our main criterion for checking that an algebra $A$ is $d$-extremely narrow, assuming that $I$ is generated in a single degree. Note that the hypothesis does not require $A$ to be artinian. 

 \begin{prop}
 \label{strategy}
Let $n\geq 3$, $d\ge 1$ and $a\ge 1$.  Let $A=R/I$ be an algebra quotient of $R$ with $t(I)= d$, and assume the following conditions hold. 
 \begin{enumerate}[\quad\rm(1)]
  \item $\dim_\kk A_d\le a$, or equivalently $\dim_\kk I_d\ge \dim_\kk R_d-a$;  
 \item There exist $F_1, \dots, F_a\in( I^\perp)_{-d}$ linearly independent such that $$L_{F_1, \dots, F_a}\subseteq x{R_0}^a\quad\text{ for some}\quad  x\in R_1\,.$$
 \end{enumerate}
 Then $I^\perp=(F_1, \dots, F_a)+S_{\geqslant -d+1}$  and $A$ is $d$-extremely narrow, with socle polynomial 
 $$(\dim_\kk R_{d-1}-an+\dim_\kk L_A)z^{d-1}+az^d\,.$$ 
 \end{prop}
 \begin{proof}
 Assume (1) and (2) hold. Let $F_1, \dots, F_a\in(I^\perp)_{-d}$ as in (2). We have 
 $$
a= \dim_\kk\langle F_1, \dots, F_a\rangle \le \dim_\kk(I^\perp)_{-d}=\dim_\kk A_d\le a. 
$$
In view of condition (1), we see that the inequality above must be an equality, and hence
 $$
 ( I^\perp)_{-d}=\langle F_1, \dots, F_a\rangle\,,
$$
 implying also $I_d=\Ann_R(F_1, \dots, F_a)_d$ by \eqref{e:duality}. 

Note that we must have $a<\dim_\kk R_d$, since $t(I)=d$. To show $A$ is $d$-extremely narrow we need to show $s(A)\leq d$ (and hence $s(A)=d$).  Set 
 $$J=\Ann_R((F_1, \dots, F_a)+S_{\geqslant -d+1})\quad\text{and}\quad A'=R/J\,.
 $$
 By  \cref{l:narrow-equiv1}, the algebra $A'$ is then narrow  with $s(A')=d$ and $\dim_\kk\Soc(A')_d=a$,  and in particular $t(J)\ge  d$. Since $a<\dim_\kk R_d$, we must have $t(J)=d$ by \cref {l:narrow}(3). Using the hypothesis (2), it follows that $J$ is a $d$-extremely narrow algebra, and in particular $J$ is generated in degree $d$, by \cref{en-props}(4).

 Since $t(I)=d$, we have $S_{\geqslant -d+1}\subseteq I^\perp$  and hence $(F_1, \dots, F_a)+S_{\geqslant -d+1}\subseteq I^\perp$. We have then 
 $$
 I=\Ann_R(I^\perp)\subseteq\Ann_R((F_1,\dots, F_a)+S_{\geqslant -d+1})=J\,.
 $$
 Since $I_d=\Ann_R(F_1, \dots, F_a)_d=J_d$, and since $J$ is generated in degree $d$, whereas $I$ is generated in degrees at least $d$, we have $J\subseteq I$. We conclude that $I=J$ and hence $A=A'$; moreover, $A$ is $d$-extremely narrow. 
 
Finally, the formula for the socle polynomial follows from \cref{l:narrow}(1). 
 \end{proof}
 
\begin{thm}
\label{thm:ses}
Let $d\ge 1$ and let $A=R/I$ be a $d$-extremely narrow algebra with socle polynomial $bz^{d-1}+az^d$. 
By \Cref{l:narrow}~(5) there is a short exact sequence 
\begin{equation}
\label{exact-seq}
0\to I\hookrightarrow\m^d\xrightarrow{\varphi} \kk^a(-d)\to 0\,.
\end{equation}
This sequence induces a long sequence in homology that splits into exact sequences: 
\begin{gather*}
0\to \Tor_i^R(I,\kk )_{i+d}\to \Tor_i^R(\m^d, \kk)_{i+d}\to \Tor_i^R(\kk^a,\kk )_{i}\to 0\qquad \text{for  $i\le n-2$}\\
0\to \Tor_{n-1}^R(I,\kk )_{n-1+d}\to  \Tor_{n-1}^R(\m^d, k)_{n-1+d}\to \Tor_{n-1}^R(\kk^a,\kk )_{n-1}\to \Tor_{n-2}^R(I,\kk )_{n-1+d}\to 0\\
\Tor_n^R(\kk^a,\kk )_{n}\cong\Tor_{n-1}^R(I, k)_{n+d}
\end{gather*}
where all $\Tor_i^R(I,\kk )_j$ that do not appear in one of the sequences above are equal to zero. 
Further, there are isomorphisms 
\begin{gather*}
 \Tor_{n-2}^R(I,\kk )_{n-1+d}\cong \Hom_\kk(L_A, \kk), \qquad 
\Tor_{n-1}^R(I,\kk )_{n-1+d}\cong \Hom_\kk((I^\perp)_{-d+1}, \kk)\cong \kk^b,\\
\Tor_{n-1}^R(I,\kk )_{n+d}\cong \Hom_\kk((I^\perp)_{-d}, \kk)\cong \kk^a. 
\end{gather*}
\end{thm}

\begin{proof}
The fact that all $\Tor_i^R(I,\kk )_j$ that do not appear in one of the sequences above are equal to zero follows from \cref{l:narrow-equiv2} and \cref{en-props}. 

The modules $\kk$ and $\m^d$ are known to have a linear resolution, see \Cref{ex:resk} and \Cref{ex:resmd},  and hence 
$$\Tor_i^R(\kk^a(-d),\kk)_{j+d}=\Tor_i^R(\kk^a, \kk)_j=0=\Tor_i^R(\m^d,\kk)_{j+d}=0\qquad\text{for $j\ne i$.}$$
This implies the claimed splitting of the long exact sequence in homology. 

The last isomorphisms can be seen using the isomorphism $A^\vee\cong I^\perp$, \cref{Boij} and  \Cref{en-props}~(2). 
\end{proof}

The next corollaries are deduced directly from \cref{thm:ses}. They describe the betti numbers and the minimal free resolution of the ideal defining a $d$-extremely narrow algebra. 
 \begin{cor}
 \label{cor:betti}
Let $d\ge 1$ and $a\ge 1$. If $A=R/I$ is a $d$-extremely narrow algebra with $\dim_\kk\Soc(A)_d=a$, then 
$$
 \beta_{i,j}^R(I)=\begin{cases}
 \beta_{i}^R(\m^d)-a\beta_i^R(\kk ) &\text{if $i\le n-2$ and $j=i+d$}\\
  \dim_\kk (L_A) &\text{if $i=n-2$ and $j=i+d+1$}\\
 a &\text{if $i=n-1$ and $j=i+d$}\\
 \dim_\kk R_{d-1}-an+\dim_\kk L_A &\text{if $i=n-1$ and $j=i+d+1$}\\
 0&\text{otherwise. }
 \end{cases}
$$
With $u_i=\beta_i(\m^d)-a\beta_i(\kk)$, $\ell=\dim_\kk(L_A)$ and $b=\dim_\kk R_{d-1}-an+\dim_\kk L_A$, the minimal free resolution $\mathcal F$ of $I$ over $R$ has the following shape: 
$$
R^{u_0}(-d)\leftarrow R^{u_1}(-d-1)\leftarrow\cdots\leftarrow R^{u_{n-3}}(-d-n+3)\leftarrow \substack{R^{u_{n-2}}(-d-n+2)\\\oplus\\ R^{\ell}(-d-n+1)}\leftarrow \substack{R^b(-d-n+1)\\\oplus\\R^{a}(-d-n)}\leftarrow 0\,.
$$
 \end{cor}
 
 \begin{cor} 
 \label{res-as-kernel}
 Let $d\ge 1$. If $A=R/I$ is a $d$-extremely narrow algebra with socle polynomial $bz^{d-1}+az^d$, then the exact sequence \eqref{exact-seq}  induces short exact sequences 
 \begin{align*}
 0\to \Tor_i^R(I,k)&\to \Tor_i^R(\m^d,\kk)\to \Tor_i^R(\kk^a(-d), \kk)\to 0\qquad\text{for all $i\le n-3$}\\
&0\to \Tor_n^R(\kk^a(-d),\kk)\to \Tor_{n-1}^R(I,\kk)\to 0\,.
 \end{align*}
Furthermore, if $L_A=0$, then the first sequence is also exact for $i=n-2$. 

Consequently, if $\mathcal F$, $\mathcal L$, and $\mathcal G$ respectively denote minimal free resolutions of $I$, $\m^d$, and $\kk^a(-d)$, and we extend $\varphi\colon \m^d\to \kk^a(-d)$ to a map of complexes $\varphi\colon \mathcal L\to \mathcal G$, then 
$$
\mathcal F_{\leqslant n-3}\cong \Ker(\varphi_{\leqslant n-3})\,.
$$
Furthermore, if $L_A=0$, then $
\mathcal F_{\leqslant n-2}\cong \Ker(\varphi_{\leqslant n-2})$. 
 \end{cor}
 
\begin{rem}
The resolutions $\mathcal G$, $\mathcal L$ are known, see \Cref{ex:resk} and \Cref{ex:resmd}. Then \cref{res-as-kernel} can  be used to describe the first $n-3$ differentials in $\mathcal F$. Note that the matrices representing these differentials have linear entries. If one wants to describe the matrices corresponding to the last two differentials in the minimal free resolution of a $d$-extremely narrow algebra $A$, it may be useful to think of them as being the transpose to the first two matrices in a minimal free resolution of $A^\vee=I^\perp$ over $R$. 
 \end{rem}

 We end this section with a result that addresses betti numbers of modules over a narrow algebra $A$. In conjunction with the formulas in \cref{cor:betti}, this result allows one to compute the betti numbers of $\kk$ over $A$ when $A$ is a $d$-extremely narrow algebra.

We first need to introduce the necessary concepts. 
Let $A=R/I$ be an algebra quotient of the polynomial ring $R=\kk[x_1, \dots, x_n]$. 
If $M$ is a finitely generated graded $A$-module, then $\beta_n^A(M)=\dim_\kk \Tor_n^A(M,\kk)$ and  the {\it Poincar\'e series} $P_M^A(t)$ of $M$ over $A$ is the generating series of the sequence $\{\beta_n^A(M)\}_n$ of betti numbers of $M$, namely 
\[
P_M^A(t) :=\sum_{n=0}^\infty \beta_n^A(M)t^n\,.
\]
Since minimal free resolutions of modules over a non-regular ring are often  infinite, attention has been given to understanding whether Poincar\'e series are rational. A discussion of this problem can be found in \cite{Av}. In general, one has an inequality
\begin{equation}
\label{e:Golod}
P_\kk^A(t)\le \frac{(1+t)^n}{1-t(P_\kk^R(t)-1)}\,.
\end{equation}
When equality holds in \eqref{e:Golod}, the ring $A$ is said to be {\it Golod}, see \cite[Section 5]{Av}. While the original definition is usually stated for local rings, it translates as usual to the graded case at hand. When $A$ is Golod, not only $P_\kk^A(t)$ is rational, but it is also known that $P_M^A(t)$ is rational, sharing the same denominator, for all finitely generated graded $A$-modules $M$.

Another instance when $P_\kk^A(t)$ is rational is when the algebra  $A$ is {\it Koszul}, meaning that $\beta^A_{i,j}(\kk)=0$ for all $i\ne j$, or, equivalently, that the minimal graded  free resolution of $\kk$ over $A$ is linear. In this case, one has
\[
P_\kk^A(t)=\frac{1}{\H_A(-t)}\,.
\]

\begin{prop}\label{prop:GolodKoszul}
Let $A=R/I$ be an algebra quotient of $R=\kk[x_1, \dots, x_n]$. Then 
\begin{enumerate}[\quad\rm(1)]
\item \label{r:Golod} If $A$ is narrow and $t(I)\ge 3$, then $A$ is Golod. 
\item \label{r:Koszul} If $A$ is $2$-extremely narrow, then $A$ is Koszul.
\end{enumerate}
In either case the Poincar\'e series of all finitely generated graded $A$-modules are rational, sharing the same denominator.
\end{prop}
\begin{proof}
(1) If $A$ is a narrow algebra and $t(I)\ge 3$, then we see that the inequality $s(A)\le 2t(I)-3$ holds, and thus $A$ is a Golod ring by \cite[Observation 5.6]{KSV}. 

(2) Assume $A$ is $2$-extremely narrow with socle polynomial $bz+az^2$. 
Condition (ii) of \cref{def:extremelynarrow} implies $\dim_\kk L_A\le a$. Using  \cref{l:narrow}(1), we have  
\begin{equation}
\label{eq:b}
0\le b=n-an+\dim_\kk L_A\le n-an+a\,.
\end{equation}
If $n>2$, this inequality implies $a=1$, or, in other words,  $\dim_\kk {\m_A}^2= 1$. We also have $a=1$ when $n=2$, because $a=\dim_\kk{\m_A}^2\le \frac{n(n-1)}{2}$ when $A$ is Artinian. From \eqref{eq:b} we also get  $b\le 1$. One can see that $b=0$ when $n=2$.  

We have thus $\dim_\kk{\m_A}^2= 1$ and $\dim_\kk\Soc(A)=a+b\le n-1$.  By \cite[Theorem 4.1, Corollary 4.4]{AIS}, we see that $A$ is Koszul  with $P_\kk^A(t)=(1-nt+t^2)^{-1}$ and the Poincar\'e series of all finitely generated graded $A$-modules are rational, with denominator equal to $1-nt+t^2$. (The results of \cite{AIS} imply that $A$ is in fact {\it absolutely Koszul}, in the sense of \cite{C}.)
\end{proof}

\begin{rem}
Assume that $A=R/I$ is narrow  with  $t(I)=2$, which is equivalent to ${\m_A}^3=0$.
An example  of Anick \cite{An} shows there exist such algebras for which $P_\kk^A(t)$ is not rational. However, this does not happen when $A$ is $2$-extremely narrow, as shown in \Cref{prop:GolodKoszul} \eqref{r:Koszul} above.
\end{rem}

\section{Quadratic principal symmetric ideals}
\label{sec: quadratic symmetric} 

In this section we begin our task of studying principal symmetric ideals. Such an ideal is denoted $(f)_{\sym_n}$, where $f\in R$ is a homogeneous polynomial. As explained in \cref{s:prelim}, such ideals can be parametrized by points in projective space. In this section we focus on the case of quadratic polynomials $f$, i.e. $\mathrm{deg}(f)=2$, in order to gain intuition towards the general case.

 We denote by $\mu(I)$ the minimal number of generators of a homogeneous ideal $I$. We record here a general result on the minimal number of generators of a principal symmetric ideal, needed for our later arguments. 

\begin{lem}
\label{upper-s}{\rm (Upper semicontinuity of the minimal number of generators.)} Let $m\ge 1$, $d\ge 1$.  If there exists a principal symmetric ideal $I'$ generated in degree $d$ such that $\mu(I')\ge m$, then a general principal symmetric ideal ideal $I$ generated in degree $d$ satisfies $\mu(I)\ge m$. 
\end{lem}

\begin{proof}
Using  the notation in \cref{s:prelim}, 
let $I=(f)_{\sym_n}$, with $f=f_c=\sum_{i=1}^N c_im_i$ and $$c=(c_1:\cdots:c_N)\in \P^{N-1}.$$ Note that 
$$
\mu(I)=\dim_\kk I_d=\dim_\kk\langle \sigma\cdot f : \sigma \in \sym_n\rangle
$$
is given by the rank of the $(n!)\times N$ matrix $M$ of coefficients for the set of degree $d$ polynomials  $\sigma\cdot f $ with respect to the monomial basis $m_1, \dots, m_N$ of $R_d$. The locus where $\dim_k I_d <m$ is cut out by the ideal $I_m(M)$ of $m\times m$ minors of $M$, which are homogeneous polynomials in $c_1, \dots, c_N$.  Set $U=\P^N\setminus V(I_m(M))$. This is a Zariski-open set of $\P^N$ which is not empty by the hypothesis on the existence of the ideal $I'$. Thus, if $c\in U$, then $\dim_\kk I_d\ge m$. 
\end{proof}

We start our investigation by looking at a specific case, cf. Example~\ref{ex:Liana} below. We then describe the general behavior, when the degree of the generator is $2$. 

\begin{ex}
\label{ex:Liana}
Assume $n\ge 2$. Let $I=(x_1^2-x_2^2+x_1x_2)_{\sym_n}$ and set $A=R/I$.  Observe 
$$
I= (x_1^2-x_2^2, x_1x_2)_{\sym_n}=(x_i^2-x_j^2, x_kx_\ell\colon  i <  j, k <  \ell)=(x_1^2-x_j^2, x_kx_\ell\colon j\ne 1, k<\ell)\,.
$$
The generators listed on the right are clearly linearly independent, and thus 
$$
\dim_\kk I_2=\dim_\kk R_2-1\,.
$$
It is straightforward to verify that the inverse system of  $I$ is 
\[
I^\perp=( y_1^{(2)}+\cdots+y_n^{(2)} ).
\]
Since $I^\perp$ is a cyclic $R$-module generated by $y_1^{(2)}+\cdots+y_n^{(2)}$, $A$ is Gorenstein; see \Cref{r:cyclic}.
 \end{ex}

The above example is an illustration of the following general result.  

\begin{thm}
\label{thm:soclequadratic}
Assume $n\ge 2$ and the field $\kk$ is infinite. A general principal symmetric ideal $I$ generated by a homogeneous quadratic polynomial yields a quotient $A=R/I$ so that $A$ is artinian Gorenstein with ${\m_A}^3=0\neq {\m_A}^2$.

 Furthermore,  $A$ is $2$-extremely narrow, the Macaulay inverse system $I^\perp$ is a cyclic $R$-module generated by a symmetric quadratic polynomial, and $L_A=0$. 
\end{thm}
\begin{proof}
As explained in \cref{s:prelim}, and with the notation there, we consider $f=f_c$ with $c\in \mathbb P^{N-1}$, where $N=\binom{n+1}{2}$ and $f=c_1m_1+c_2m_2+\cdots +c_Nm_N$. We find it convenient to rename the coefficients $c_k$ as follows: 
\[
c_k=\begin{cases}
a_i&\text{if $m_k=x_i^2$}\\
b_{ij}&\text{if $m_k=x_ix_j$ with $i< j$.}
\end{cases}
\]
With this notation, we have: 
\begin{equation}
\label{e:deg2f}
f =\sum_{i=1}^n a_ix_i^2+ \sum_{1\leq i<j\leq N} b_{ij} x_ix_j\,.
\end{equation}

We first show that for  $f$ general, the hypotheses (1) and (2) in \Cref{strategy} are satisfied. 
To prove (1),  use \Cref{ex:Liana}  and upper semicontinuity (\cref{upper-s}) to  find a non-empty Zariski open set $U$ such that, if $c\in U$, then $\dim_\kk I_2\ge \dim_\kk  R_2-1$. 

To prove (2), define 
\begin{equation}\label{eq:Fdeg2}
F =\beta \left( \sum_{i=1}^n y_i^{(2)} \right)- \alpha \left(\sum_{\scriptscriptstyle 1\leq i<j\leq n}  y_iy_j\right)
\end{equation}
where 
$$\alpha=\sum_{i=1}^n a_i\qquad\text{and}\qquad \beta =\sum_{1\leq i<j\leq n} b_{ij}\,.
$$
A computation shows that $(\sigma\cdot f)\circ F=0$ for each $\sigma\in \sym_n$ and hence $F\in (I^\perp)_{-2}$.

We show next that, for $f$ general, we have $L_F=0$.  Suppose that $\ell=\sum_{i=1}^n e_i x_i$ satisfies $\ell\circ F=0$. This translates as 
\[
\alpha \left( \sum_{\scriptscriptstyle j\neq k}^n e_j \right) - \beta e_k=0 \text{ for all }1\leq k\leq n.
\]
Regarding this as a system of linear equations with indeterminates $e_1, \dots, e_n$, the coefficient matrix of this system is the $n\times n$ matrix
\[
M=\begin{bmatrix}
-\beta & \alpha & \alpha & \cdots & \alpha\\
\alpha & -\beta & \alpha & \cdots & \alpha\\
\alpha & \alpha & -\beta & \cdots & \alpha\\
\vdots & \vdots & \vdots && \vdots \\
\alpha&\alpha& \alpha & \cdots & -\beta
\end{bmatrix}
\]
with 
$$\det(M)=((n-1)\alpha-\beta)(\alpha-\beta)^{n-1}\,.
$$
When $c\in U'=\P^{N-1}\setminus V(\det(M))$ one obtains $\ell=0$ and hence $\dim_\kk L_A=0$. Further, note that, if  $f$ is as in \Cref{ex:Liana}, then $\alpha=0$ and $\beta=1$, and hence $\det(M)\ne 0$ in this case, showing that $U'$ is non-empty. 

Finally, $U\cap U'$ is not empty, since both $U$ and $U'$ are non-empty Zariski open sets and $\kk$ is assumed infinite. 
Thus, for $c\in U\cap U'$, \cref{strategy} shows that the algebra  $A$ is $2$-extremely narrow with socle polynomial $bz+z^2$, and 
$$
b=n-n+\dim_\kk L_A=0\,.
$$
This shows that the socle is 1-dimensional concentrated in degree two, and $I^\perp=(F)$.  The fact that $A$ is Gorenstein follows from \cref{r:cyclic}. 
\end{proof}

We now establish the $\sym_n$-equivariant structure of the resolution of a general quadratic principal symmetric ideal. 

\begin{thm}
\label{thm:resolutionquadratic}
Let $n \geq 2$ and let $\kk$ be an infinite field with ${\rm char}(\kk )=0$. A  general principal symmetric ideal $I=(f_c)_{\sym_n}$ with $\deg(f_c)=2$ 
satisfies \footnote{Recall that if $\alpha$ is a tuple which is  not a partition then $\Sp_\alpha=0$ by our convention.}
\[
\Tor^R_i(R/I,k) = 
\begin{cases}
\Sp_{(n-i, 2,1^{i-2})}\oplus  { \Sp_{(n-i, 1^{i})}}^2\oplus  {\Sp_{(n-i-1, 2,1^{i-1})}}^2 \\ \,\,\oplus \   {\Sp_{(n-i-1, 1^{i+1})}}^2 \oplus \Sp_{(n-i-2, 2,1^{i})} & \text{for } 0\leq i< n\\
\Sp_{(1^n)}\oplus  \Sp_{(2,1^{n-2})} &\text{for } i=n.
\end{cases}.
\]
\end{thm}

\begin{proof}
Let $I$ be so that the conclusion of  \Cref{thm:soclequadratic} holds. In particular, $I^\perp=(F)$, where $F$ is a symmetric quadratic polynomial and $L_A=0$. 
By \Cref{l:narrow}~(5) we have  a short exact sequence
\begin{equation}
\label{eq:quadratic1}
0\to I\to R_{\geqslant 2}\xrightarrow{\varphi} {\kk}(-2)\to 0
\end{equation}
and this exact sequence is $\sym_n$-equivariant, when considering the trivial $\sym_n$-action on $\kk(-2)$. By \Cref{res-as-kernel}, this sequence  induces equivariant short exact sequences as follows: 
\begin{equation}
\label{eq:sesTori}
0\to \Tor_i^R(I,{\kk})\to \Tor_i^R(\m^2, {\kk})\to \Tor_i^R({\kk}(-2), {\kk})\to 0\quad\text{if $i<n-1$}
\end{equation}
\begin{equation}
\label{ses:Torn}
0\to \Tor_{n}^R({\kk}(-2),{\kk})\to \Tor_{n-1}^R(I,{\kk})\to 0.
\end{equation}

It remains to determine the representation-theoretic structure of the modules $\Tor_i^R({{\kk}}, {\kk})$ and $\Tor_i^R(\m^2, {\kk})$ appearing in \eqref{eq:sesTori} and \eqref{ses:Torn}. 
Recall from \Cref{ex:resk} that the Koszul complex yields 
\begin{eqnarray*}
\Tor_i^R({{\kk}}, {\kk})  = \Sp_{(n-i,1^i)} \oplus  \Sp_{(n-i+1,1^{i-1})}.
\end{eqnarray*}
As in \Cref{ex:resmd}, a $\sym_n$-equivariant minimal resolution  $\mathcal{L}$ of $\m^2=R_{\geqslant 2},$ can be described by setting 
$$\mathcal L_i=\coker\left(\bw^{i+2}R_1\to \bw^{i+1} R_1 \otimes_R R_1 \right)\otimes_{\kk} R$$ 
with  differentials induced from the Koszul complex of $R$. 
Thanks to \cref{prop:propertiesofSpecht} (3) we have 
\begin{eqnarray*}
\bw^i R_1 \otimes_R R_1 &=& \left(  \Sp_{(n-i,1^i)} \oplus  \Sp_{(n-i+1,1^{i-1})} \right) \otimes_R \left( \Sp_{(n)}\oplus \Sp_{(n-1,1)} \right) \\
&=& \Sp_{(n-i+2,1^{i-2})} \oplus  \Sp_{(n-i+1,2,1^{i-3})} \oplus  {\Sp_{(n-i+1,1^{i-1})}}^3 \\
&& \oplus \ {\Sp_{(n-i, 2,1^{i-2})}}^2 \oplus {\Sp_{(n-i,1^{i})}}^3 \oplus  \Sp_{(n-i-1,2,1^{i-1})},
\end{eqnarray*}
in view of which we compute
\begin{eqnarray*}
\Tor_i^R(R_{\geqslant 2}, {\kk}) &=& \coker\left(\bw^{i+2}R_1\to \bw^{i+1} R_1 \otimes_R R_1 \right)\\
&=&  \Sp_{(n-i+1,1^{i-1})} \oplus  \Sp_{(n-i,2,1^{i-2})} \oplus   {\Sp_{(n-i,1^{i})}}^3\\
&&{\oplus \Sp_{(n-i-1, 2,1^{i-1})}}^2 \oplus {\Sp_{(n-i-1,1^{i+1})}}^2 .
\end{eqnarray*}
Substituting the above identities into equations \eqref{eq:sesTori} and \eqref{ses:Torn} yield the claim.
\end{proof}

 The goal of the next sections is to generalize the above phenomena, that we have recorded in the degree $2$ case.

\section{Principal symmetric ideals generated in degree $d$: an example}
\label{s:example}

Based on our work in the previous section, it is apparent that an essential part of proving properties of symmetric ideals that hold ``generically'' is to construct a concrete example which exhibits those properties, thus ensuring that an appropriately-defined Zariski-open subset is non-empty.  The proof of \cref{thm:soclequadratic} clearly illustrates this principle.
The main result of this section is \Cref{lem:generalexample}. Its purpose is to construct a polynomial $f$, of  degree $d \geq 2$ and in $n$ variables for $n$ sufficiently large,  which has the property that the principal symmetric ideal it generates is particularly easy to describe, while also exhibiting the characteristic properties of general principal symmetric ideals described in \Cref{introthm}.   We will then use \Cref{lem:generalexample} in later sections to prove results analogous to those in Section~\ref{sec: quadratic symmetric} when $d \geq 2$. 

Throughout the section, we work in a polynomial ring $R=\kk[x_1, \dots, x_n]$, where $n$ will need to be chosen to be large enough in order for the objects of interest to exist. To explain the construction, we need some preliminary notation and definitions. 

\begin{defn}
\label{d:admissible} Fix a partition $\lambda=(\lambda_1,\cdots,\lambda_s)$ of $d$ with $s$ parts, with $d\ge 1$. 
\begin{itemize}
\item We say that a monomial of the form $m=x_{i_1}^{\lambda_1} x_{i_2}^{\lambda_2} \cdots x_{i_s}^{\lambda_s}$, with $i_1, \cdots, i_s$ pairwise distinct, has {\it type} $\lambda$.

\item If $m$ is a monomial of degree $d$, then there exists a unique partition $\lambda\vdash d$ such that $m$ has type $\lambda$, and we write $\type(m)=\lambda$. 

\item We denote by $()$ the  empty partition with $0$ parts and set $\type(1)=()$. 

\item We say that $b\in R$ is a $\lambda$-{\it binomial} if $b=m-m'$, with $m,m'$  distinct monomials of type $\lambda$. 

\item  If $b$ is a $\lambda$-binomial as above, we set $g(b)=\gcd(m,m')$.  We say that $b$ is {\it admissible} if $g(b)$ is relatively prime with both $\frac{m}{g(b)}$ and $\frac{m'}{g(b)}$. 
\end{itemize}
\end{defn}

 Given a monomial $m$ of type $\lambda$ as above, note that the indices $i_1, \dots, i_s$ are not uniquely determined by $m$; only the (unordered) set of these indices is unique.  We usually write $m$ so that $i_k<i_{k+1}$ when $\lambda_k=\lambda_{k+1}$, but rearrangements may be needed in what follows, in order to verify certain conditions. 

\begin{ex}
Let $\lambda=(3,2,2,1)$. Set
\[
b=x_4^3x_1^2x_2^2x_3-x_5^3x_1^2x_2^2x_6\quad\text{and}\quad b'=x_1^3x_2^2x_3^2x_4-x_2^3x_1^2x_5^2x_6\,.
\]
Then  $b$ and $b'$ are $\lambda$-binomials and $g(b)=x_1^2x_2^2=g(b')$. The binomial $b$ is admissible, while $b'$ is not admissible. 
\end{ex}

\begin{defn}
 We say that a partition $\gamma$ is a {\it subpartition} of the partition $\lambda$, and we write $\gamma\subseteq \lambda$, if the multiset of parts of $\gamma$ is a submultiset of the parts of $\lambda$. If $\gamma\subseteq\lambda$, we describe $\gamma$ by indicating which parts of $\lambda$ are present in $\gamma$ as follows. 
 If $\gamma=(\gamma_1, \cdots, \gamma_t)$, then we can choose distinct integers $k_1, \cdots, k_t$ such that  $\gamma_i=\lambda_{k_i}$ for all $i$ with $1\le i\le t$.  Since the choice of such integers is not necessarily unique, we further require that these indices are chosen in order, starting with $k_1$ and ending with $k_t$ and, whenever a choice is to be made, we choose $k_i$ to be the smallest of the available choices. With these rule in place, then the indices $k_i$ are unique and we set $T(\lambda, \gamma) :=\{k_1, \dots, k_t\}$. 
\end{defn}

\begin{ex}
\label{e:div}
If $\lambda=(5,5,2,2,1)$, then the following is a list of all subpartitions $\gamma$ with $\gamma\subsetneq\lambda$:
\begin{gather*}
(5,2,2,1), \quad (5,5,2,1),\quad (5,5,2,2), \quad 
(5,5,2), \quad (5,5,1), \quad (5,2,2), \quad (5,2,1),\\ (2,2,1), \quad 
(5,5), \quad (5,2), \quad (5,1),\quad (2,2),\quad (2,1), \quad 
(5), \quad(2), \quad (1), \quad ().
\end{gather*}
With $\gamma=(5,5,2)$, there are two choices for $k_1$ with $\gamma_1=\lambda_{k_1}$, namely $k_1=1$ or $k_1=2$. Since we are supposed to choose the smallest of the choices,  we must take $k_1=1$. The only remaining choice for $k_2\ne k_1$ with $\gamma_2=\lambda_{k_2}$ is $k_2=2$. Further, there are two choices for $k_3$ with  $\gamma_3=\lambda_{k_3}$, namely $k_3=3$ or $k_3=4$. Our definition dictates $k_3=3$ and thus $T((5,5,2,2,1),(5,5,2))=\{1,2,3\}$.
\end{ex}

\begin{rem}
Let $d\ge 1$ and let  $\lambda=(\lambda_1, \dots, \lambda_s)$ be a partition with $s$ parts. Let $b$ be a $\lambda$-binomial and let $\gamma$ denote the type of $g(b)$.
Using the definition of admissible binomial, it is not hard to see that the following statements are  equivalent:
\begin{enumerate}
\item $b$ is admissible; 
\item $\gamma\subsetneq \lambda$, and there exists a choice of indices with  $i_1, \dots, i_s$ distinct and $j_1, \dots, j_s$ distinct such that  $b=x_{i_1}^{\lambda_1} \cdots x_{i_s}^{\lambda_s} - x_{j_1}^{\lambda_1} \cdots  x_{j_s}^{\lambda_s}$ and
\begin{equation}
\label{e:T}
i_\ell=j_\ell\quad\text{for all $\ell\in T$}\qquad\text{and}\qquad  \{i_\ell \, \mid \, \ell \not \in T\} \cap \{j_\ell \, \mid \, \ell \not \in T\} = \emptyset
\end{equation}
where $T=T(\l,\gamma)$. 
\end{enumerate}

Further, observe that condition (2) above can only hold if $n$ (the number of variables) is sufficiently large, and thus the existence of admissible binomials imposes restrictions on the size of $n$.  If $n$ is sufficiently large, then for any  $\gamma\subsetneq \lambda$ there exists an admissible $\lambda$-binomial $b$ with $g(b)$ of type $\gamma$. Note that such a $b$ may not be unique. 
\end{rem}

\begin{ex}
\label{e:|}
Let $\lambda=(5,5,5,2,1)$ and $\gamma=(5,5,2)$. We have $T=\{1,2,4\}$. To construct an admissible $\lambda$-binomial $b$ such that $g(b)$ has type $\gamma$, we need $n\ge 7$. Then we can choose $(i_1, i_2,i_3,i_4,i_5)=(1,2,3,4,5)$ and $(j_1, j_2,j_3,j_4,j_5)=(1,2,6,4,7)$, giving: 
\[
b=x_1^5x_2^5x_3^5x_4^2x_5-x_1^5x_2^5x_6^5x_4^2x_7. 
\]
The binomial 
\[
b'=x_1^5x_2^5x_3^5x_4^2x_5-x_1^5x_3^5x_6^5x_4^2x_7=x_1^5x_3^5x_2^5x_4^2x_5-x_1^5x_3^5x_6^5x_4^2x_7
\]
is also an example of an admissible $\lambda$-binomial $g(b')$ of type $\gamma$. Note that the multiplication order in $b'$ needs to be rearranged to the form on the right,  in order for \eqref{e:T} to be satisfied. 
\end{ex}

While there may be many choices of admissible $\lambda$-binomials $b$ with $g(b)$ of type $\gamma$ for a fixed $\gamma$, the choice is unique up to the action of $\sym_n$. The following statement is straightforward so we omit the proof. 
\begin{lem}
\label{l:orbit} 
Let $\lambda$ be a partition of $d\ge 2$, If $b$ and $b'$ are both admissible $\lambda$-binomials such that $g(b)$ and $g(b')$ have the same type, then there exists $\sigma\in \sym_n$ such that $\sigma\cdot b=b'$. 
\end{lem}

\begin{constr}
\label{c:f} Fix an integer $d\geq 2$. We proceed to construct a homogeneous polynomial $f\in \kk[x_1, \dots, x_n]$ of degree $d$.  

For each partition $\l\vdash d$, $\l\ne (d)$ with $s$ parts and for each partition $\gamma$ with $\gamma\subsetneq \lambda$,  including $\gamma=()$, we choose an admissible  $\lambda$-binomial $b(\lambda, \gamma)$ with $g(b(\l, \gamma))$ of type $\gamma$ 
and such that for each $i$ with $1\le i\le n$, the variable $x_i$ appears in at most one of the terms of the summation \eqref{eq:large f} below. 
Note that such a collection of binomials exists, provided that the number $n$ of variables is sufficiently large. While the binomial $b(\lambda, \gamma)$ is dependent on a choice, the ideal $(b(\lambda, \gamma))_{\sym_n}$ 

We then define $f$ as follows
\begin{equation}
\label{eq:large f}
f :=x_1^d+\sum_{\lambda\vdash d, \lambda\neq (d)}\, \sum_{\gamma\subsetneq \lambda} b(\lambda, \gamma)\,.
\end{equation}
By construction, $f$ is homogeneous of degree $d$. 
Since $f$ depends on the choice of the binomials $b(\lambda, \gamma)$, it is not uniquely determined. However, in view of \cref{l:orbit} and since there is no overlap in indices between the variables of the summands,  we see that the principal symmetric ideal $(f)_{\sym_n}$ generated by $f$ is independent of the choices made in defining $f$. 
\end{constr}

\begin{ex}\label{ex:d=3}
Let $d=3$. Then the possible partitions $\lambda$ are $(3),(2,1),(1,1,1)$, with $1$ part, $2$ parts and $3$ parts, respectively.  

We list all possible of $\gamma$, $\lambda$ with $\lambda\ne (3)$,  $\gamma\subsetneq \lambda$: 
For $\lambda=(2,1)$, we have $\gamma=()$ or $\gamma=(2)$ or $\gamma=(1)$. 
For $\lambda=(1,1,1)$, we have $\gamma=()$ or $\gamma=(1)$ or $\gamma=(1,1)$. 
Thus, the summation in the formula \eqref{eq:large f} will have $7$ terms. We need to choose the $6$ binomials $b(\lambda, \gamma)$, making sure that no two such binomials share the same variables among themselves and also with $x_1^3$. 
Below is such a choice: 
\begin{equation*}
\begin{split} 
f & = x_1^3 + (x_2^2 x_3 - x_4^2 x_5) + (x_6^2 x_7 - x_6^2 x_8) + (x_9^2 x_{10} - x_{11}^2 x_{10})  \\
 & +   (x_{12} x_{13} x_{14} - x_{15} x_{16} x_{17}) + (x_{18} x_{19} x_{20} - x_{18} x_{21} x_{22})  \\
 & + (x_{23} x_{24} x_{25} - x_{23} x_{24} x_{26}) \\ 
 \end{split} 
 \end{equation*} 
where we need to assume that the ambient polynomial ring has at least $26$ variables. 
\end{ex} 

The main result of this section is to give a concrete description of the principal symmetric ideal generated by the polynomial $f$ constructed in \cref{c:f}, for $n$ sufficiently large. The terminology below will be useful for this purpose. 

\begin{defn}
\label{def:monomialsymmetric}
Although we did not define a multiplication on the dual $S=R^\vee$ defined in \cref{s:duality}, we refer to the elements of $S$ as polynomials in $y_i$.  Let $\lambda = (\lambda_1, \cdots, \lambda_s)$ be a partition of $d$ with $s$ parts. We call a basis element of the form $m=y_{i_1}^{(\lambda_1)}\cdots y_{i_s}^{(\lambda_s)}$ a monomial of type $\lambda$, and we write $\type(m)=\lambda$. 

We define the {\em monomial symmetric polynomial} is $S$ corresponding to $\lambda$ to be 
\[
m_\lambda :=\sum y_{i_1}^{(\lambda_1)}\cdots y_{i_s}^{(\lambda_s)} \in S\,,
\]
where the summation is taken over all distinct monomials of type $\lambda$. 
Observe that $m_\l$ is invariant under the $\sym_n$-action on $S$. 
\end{defn}

\begin{obs}
\label{o:ml-contract}
If $m\in R$ is a monomial of type $\lambda'$ for some $\lambda'\vdash d$, then the definition of the contraction in \eqref{eq:contraction} yields 
$$
m\circ m_\lambda =\begin{cases}
0 &\text{if $\lambda'\ne \lambda$}\\
1 &\text{if $\lambda'=\lambda$.}
\end{cases}
$$
\end{obs}

We now wish to show that a polynomial $f$ as in \cref{c:f} gives rise to a corresponding symmetric ideal which can be concretely described, and moreover, its Macaulay inverse system can also be explicitly computed.  As noted earlier, for $n$ sufficiently large, polynomials $f$ as in \cref{c:f} exist, so the statement below is not vacuous.

\begin{prop}
\label{lem:generalexample}
Let $\kk$ be a field with ${\rm char}(\kk )\neq 2$. Let $d$ be an integer, $d\geq 2$.  Let $n$ be a positive integer with $n>3d$, and additionally assume $n$ is sufficiently large so that \cref{c:f} can be done. Let $f \in R = \kk[x_1,\cdots, x_n]$ as in \cref{c:f},
 and set  $I=(f)_{\sym_n}$. The following hold:
\begin{enumerate}[\quad\rm(1)]
\item The principal symmetric ideal $I$ generated by $f$ may be computed as follows: 
\begin{equation}
\label{eq:gens large f}
I=\left( x_1^d\right)_{\sym_n}+ \sum_{\lambda\vdash d, \lambda\neq (d)}\,\sum_{\gamma\subsetneq \lambda} \left(b(\lambda, \gamma)\right)_{\sym_n}.
\end{equation}
\item 
If $P(d)$ denotes the number of partitions of $d$, then
\[ 
\dim_\kk I_d= \dim_\kk R_d-(P(d)-1)\,.
\]
\item The $(-d)$-degree component of the Macaulay inverse system  of $I$ can be computed as 
\[
(I^\perp)_{-d}=\langle m_\l\mid \l\vdash d, \l\neq (d)\rangle.
\]
\end{enumerate} 
\end{prop}

\begin{proof}
We begin with (1). Since the LHS of~\eqref{eq:gens large f} clearly lies in the RHS, it suffices to show the opposite inclusion, namely, that the RHS of~\eqref{eq:gens large f} lies in the LHS. In particular, 
it suffices to show that each summand in \eqref{eq:large f} belongs to $I$.  To see this, fix a partition $\l\vdash d$, $\l\ne (d)$ with $s$ parts and $\gamma\subsetneq \lambda$ and let 
\[
b(\lambda, \gamma)=x_{i_1}^{\lambda_1} \cdots x_{i_s}^{\lambda_s} - x_{j_1}^{\lambda_1} \cdots  x_{j_s}^{\lambda_s}\,.
\]
such that $T=T(\lambda, \gamma)$ satisfies \eqref{e:T}. In particular, $i_t=j_t$ for $t\in T$ and  $i_t\ne j_t$ for $t\notin T$. Define the permutation
\[\sigma :=\prod_{t\not \in T}(i_t,  j_t)\,.\] 
By the assumptions on $T=T(\l,\gamma)$ in \eqref{e:T}, it follows that the transpositions $(i_t,j_t)$  commute with each other, and thus the multiplication order in $\sigma$ is irrelevant. 
By construction of $f$, each variable that appears in $b(\lambda, \gamma)$ does not appear in any $b(\lambda',\gamma')$ with $(\lambda, \gamma)\ne (\lambda',\gamma')$, and it follows that $\sigma$ fixes $b(\l',\gamma')$. On the other hand, observe that 
\[
\sigma\cdot b(\l,\gamma)=-b(\l,\gamma)\,.
\]
 From this we conclude
\[
b(\l,\gamma)=\frac{1}{2}\left(f-\sigma \cdot f \right)\in I.
 \]
 Since $I$ is $\sym_n$-invariant, if one binomial is contained in $I$ then any permutation of it also lies in $I$, so we conclude $(b(\lambda,\gamma))_{\sym_n}$ is contained in $I$, as desired. Moreover, by subtracting binomials we may also conclude that $x_1^d$ and hence also any $x_i^d$ is contained in $I$, completing the proof of (1).

We now prove (2). 
Let $N_\lambda$ denote the set of all monomials in $R$ of type $\lambda$ and let $M_\lambda$ denote the $\kk$-span of $N_\lambda$. Since $R_d$ is spanned by degree $d$ monomials, and any monomial of degree $d$ is of type $\lambda$ for some partition of $d$, we have 
\begin{equation}\label{eq: Rd decomp}
R_d = \bigoplus_{\lambda\vdash d} M_\lambda. 
\end{equation}
 From (1) above, we know that $I$ can be described as a sum of ideals as in~\eqref{eq:gens large f}. The ideals in question are all homogeneous, so we may restrict to degree $d$ components. Moreover, the ideals $(b(\l,\gamma))_{\sym_n}$ appearing in the RHS of~\eqref{eq:gens large f} lie entirely in $M_\lambda$ for each $\lambda$.  So by the decomposition~\eqref{eq: Rd decomp} we obtain
\begin{equation}\label{eq:DS}
I_d= \left(\left( x_1^d\right)_{\sym_n}\right)_d \oplus  \bigoplus_{\lambda\vdash d, \lambda\neq (d)}\left( \sum_{\scriptscriptstyle\gamma\subsetneq \lambda} \left( b(\l,\gamma)\right)_{\sym_n}\right)_d.
\end{equation}
It is easy to see that $\left(\left( x_1^d\right)_{\sym_n}\right)_d= \langle x_1^d, \ldots, x_n^d \rangle=M_{(d)}$. We next claim that, for each $\l\neq (d)$, the vector space $M_\l$ decomposes as 
\begin{equation}
\label{eq:Mlambda}
M_\l=\left\langle \sum_{\scriptscriptstyle m\in N_\lambda} m \right\rangle  \oplus \left(\sum_{\scriptscriptstyle\gamma\subsetneq \lambda}\left (b(\l,\gamma)\right )_{\sym_n}\right)_d.
\end{equation}
To see this, we consider the decomposition 
\begin{equation}
\label{eq:Mlambdabis}
M_\l=\left\langle \sum_{\scriptscriptstyle m\in N_\lambda} m \right\rangle  \oplus \left \langle m-m' \mid m,m'\in N_\lambda, m\ne m'\right\rangle \, . 
\end{equation}
We want to show that the second summands in \eqref{eq:Mlambda} and \eqref{eq:Mlambdabis} are equal.  Since the inclusion 
$$
\left(\sum_{\scriptscriptstyle\gamma\subsetneq \lambda} \left( b(\l,\gamma)\right)_{\sym_n}\right)_d \subseteq
 \left\langle m-m' \mid m,m'\in N_\lambda, m\ne m' \right\rangle 
$$ 
is apparent from the definitions, it suffices to prove the reverse inclusion. 
For  distinct monomials $m, m'\in N_\l$  with 
\[
m=x_{i_1}^{\lambda_1}\cdots x_{i_s}^{\l_s}\qquad\text{and}\qquad m'=x_{j_1}^{\lambda_1}\cdots x_{j_s}^{\l_s} 
\]
set $T=\{\ell\in [n] \mid i_\ell = j_\ell\}$ and consider a subset of indices $\{k_\ell \mid \ell\not \in T\}$ so that 
\[\{k_\ell \mid \ell\not \in T\}\cap \{i_\ell \mid \ell\not \in T\} =\emptyset\qquad\text{and}\qquad \{k_\ell \mid \ell\not \in T\}\cap \{j_\ell \mid \ell\not \in T\}=\emptyset\,.
\]
It is straightforward to see that such choices of $k_\ell$ for $\ell \not \in T$ always exist provided that $n>3d$. Also,  set $k_\ell=i_\ell=j_\ell$ for $\ell\in T$.  We may write 
\[m-m' =\left(x_{i_1}^{\lambda_1}\cdots x_{i_s}^{\l_s}-x_{k_1}^{\lambda_1}\cdots x_{k_s}^{\l_s}\right) +\left( x_{k_1}^{\lambda_1}\cdots x_{k_s}^{\l_s} - x_{j_1}^{\lambda_1}\cdots x_{j_s}^{\l_s} \right)\,.\]
Let $b_1$ (respectively $b_2$) denote the first (respectively second) binomial in the above decomposition of $m-m'$. Both $b_1, b_2$ are admissible $\lambda$-binomials, with $g(b_i)$ of type $\gamma$ for $i=1,2$ for some $\gamma$ with $\gamma\subsetneq \lambda$. Use then \cref{l:orbit} to conclude that $b_1, b_2\in (b(\lambda, \gamma))_{\sym_n}$, and hence that $m-m'=b_1-b_2\in (b(\lambda, \gamma))_{\sym_n}$. This establishes the desired inclusion and proves \eqref{eq:Mlambda}. In view of (1), it also follows that
\begin{equation}
\label{e:m-m'}
m-m'\in I\qquad\text{for all}\quad m,m'\in N_\lambda, m\ne m', \lambda\vdash d, \l\ne (d)\,.
\end{equation}

Since the first term in the RHS of~\eqref{eq:Mlambda} is a $1$-dimensional subspace, we may conclude from~\eqref{eq:Mlambda} the identity 
\[
\dim_\kk\left( \sum_{\scriptscriptstyle\gamma\subsetneq \lambda} \left( b(\l,\gamma)\right)_{\sym_n}\right)_d=\dim_\kk M_\l -1.
\]

Now, computing dimensions by means of \eqref{eq:DS} and using the relation $R_d=\bigoplus_{\lambda \vdash d} M_\lambda$ yields
\[
\dim_{\kk} I_d=\dim_\kk M_{(d)}+\sum_{\l\neq(d)} \left( \dim_\kk M_{\l}-1\right)=\dim_\kk R_d-(P(d)-1).
\]
This proves (2). 

Finally, we prove (3).  To begin, consider the linear subspace of $S$ defined as follows: 
 \begin{equation}\label{eq: def W} 
 W :=\left\langle m_\l \mid \l\vdash n, \l\neq (d)\right\rangle. 
 \end{equation} 
We now claim that $W = (I^\perp)_{-d}$.

The inclusion $W \subseteq  (I^\perp)_{-d}$ can be seen as follows.  In view of~\eqref{eq:gens large f}, we need to show $x_1^d\circ m_\l=0$ and $b(\l,\gamma)\circ m_\l=0$ for all $\l, \gamma$ with $\l\vdash d$, $\l\ne (d)$ and $\gamma\subsetneq \lambda$. Both these equations follow directly from \cref{o:ml-contract}. 

Observe that $\dim_\kk W=P(d)-1$. In view of \eqref{eq:GIperp} and part (2), we also have 
\[
\dim_\kk (I^\perp)_{-d}=\dim_\kk R_d-\dim_\kk I_d=P(d)-1
\]
and hence $W=(I^\perp)_{-d}$. This completes the proof. 
\end{proof}

By definition of the symmetric polynomials $m_\lambda$ and \cref{l:actionS}, the group $\sym_n$ acts trivially on each $m_\lambda$. The following corollary is therefore immediate from the above proposition.

\begin{cor}
Under the assumptions and notation of Proposition~\ref{lem:generalexample}, the natural $\sym_n$-action on $S$ restricts to be trivial on $(I^\perp)_{-d}$.
\end{cor}

\section{Linear relations of the inverse system}\label{sec: linear relations}

In the previous section, we gave an explicit description of the Macaulay inverse system of a special type of principal symmetric ideal $(f)_{\sym_n}$, where $f$ is constructed in a very specific manner. The main goal of the current section is to use the analysis in Section~\ref{s:example} to compute the space of linear relations on a subspace $W(\underline{t})$ (to be defined precisely below) of the Macaulay inverse system for a more general class of polynomials $f$. The dimension computations in this section then lead directly to our main results in Section~\ref{sec:main}.

In this section, we let $R=\kk[x_1, \dots, x_n]$ with $\kk$ a field, and $S$ denotes the dual $S=R^\vee$, as in \cref{s:duality}.

\begin{notation}
\label{n:W}
Let $d\ge 1$. We introduce scalar parameters $t_\lambda \in \kk$ for each $\lambda \vdash d, \lambda \neq (d)$. 
We assemble these scalars to define the notation $\underline{t} := (t_\lambda)_{\lambda \vdash d, \lambda \neq (d)}$. Next, we define a subspace of $S$ associated to $\underline{t}$ as follows: 
\begin{equation}
\label{eq:W(t)}
W(\underline{t}) :=\langle m_\l -t_\l m_{(d)} \, \mid \, \l\vdash d,  \l\neq (d)\rangle \, . 
\end{equation}

We associate to a homogeneous polynomial $f \in R$ of degree $d$, a list of parameters $\underline{t}=(t_\lambda)_{\lambda \vdash d, \lambda \neq (d)}$ as follows.  For each partition $\l$ of $d$, let $\alpha_\l$ denote the sum of all coefficients of monomials of type $\l$ in the support of $f$. Assume $\alpha_{(d)}\neq 0$. Under this hypothesis, we define, for each $\l\vdash d$, the scalar $t_\l :=\alpha_\l/\alpha_{(d)}$. 
\end{notation}

We start with a lemma which identifies the resulting vector space $W(\underline{t})$ as a subspace of the Macaulay inverse system of $(f)_{\sym_n}$.

\begin{lem}
\label{lem:W(t)}
Let $f \in R_d$ and let $(f)_{\sym_n}$ denote the corresponding principal symmetric ideal. Let $\alpha_\l$ be defined from $f$ as above and assume that $\alpha_{(d)} \neq 0$. Let $\underline{t}$ be constructed from $f$ as in \cref{n:W}. Then the vector space  $W(\underline{t})$ defined by~\eqref{eq:W(t)} satisfies
\[W(\underline{t})\subseteq \left((f)_{\sym_n}\right)^\perp_{-d}\,.
\]
\end{lem}

\begin{proof}
Let $\mathcal M_d$ denote the set of all monomials of degree $d$ in $R$. If $m\in \mathcal M_d$, let $a_m$ denote the coefficient of  $m$ in $f$, so that $f=\sum_{m\in \mathcal M_d} a_mm$. By \cref{o:ml-contract}, we have
$$
f \circ m_\lambda= \sum_{m\in \mathcal M_d\,,\type(m)=\lambda} a_m= \alpha_\lambda
$$
for any $\lambda\vdash d$. Thus, if $\lambda\ne (d)$, we have
\begin{equation}
\label{e:f}
f\circ (m_\lambda - t_\lambda m_{(d)}) = \alpha_\lambda-t_\lambda \alpha_{(d)}=0\,,
\end{equation}
since we  defined $t_\lambda = \frac{\alpha_\lambda}{\alpha_{(d)}}$. 

Next, let $\sigma\in \sym_n$. We claim that $(\sigma\cdot f)\circ  (m_\lambda - t_\lambda m_{(d)})=0$ when $\lambda\ne (d)$. Indeed, we have:
\[
(\sigma \cdot f) \circ (m_\lambda - t_\lambda m_{(d)}) = (\sigma \cdot f) \circ \left(\sigma\cdot (m_\lambda - t_\lambda m_{(d)})\right)= \sigma\cdot (f \circ (m_\lambda-t_\lambda m_{(d)}) = \sigma\cdot 0=0\,.
\]
In the first equality, we used the fact that $m_\lambda$ and $m_{(d)}$ are symmetric, so that they are invariant under the action of $\sigma$. In the second equality, we used the fact that $S\in \mod_{\sym_n}(R)$, as shown in \cref{rem:MGRproperties}(4). In the third equality we used \eqref{e:f}. 

The inclusion $W(\underline{t}) \subseteq (f)^\perp_{\sym_n}$ now follows, in view of the definition of $W(\underline{t})$. 
\end{proof}

It will turn out -- as a consequence of results we obtain in later sections (cf. \Cref{cor:I perp equal W(t)}) -- that, for a general (in a suitable sense) principal symmetric ideal, the containment in \Cref{lem:W(t)} is in fact an equality.

\subsection{Linear relations on $W$}

Recall from \eqref{eq:Lses} that for polynomials $F_1, \ldots, F_r \in S$, the vector space $L_{F_1, \dots, F_r}$ of linear relations on the given polynomials is defined by means of the following exact sequence of $\kk$-vector spaces: 
\begin{equation}
\label{eq:Lsesreprieve}
0\to L_{F_1, \dots, F_a}\hookrightarrow{R_1}^a\xrightarrow{\psi} R_1\circ (F_1, \dots, F_a)_{-d}\to 0
\end{equation}
where $\psi(\ell_1, \ell_2, \dots, \ell_r) :=\sum_{i=1}^r \ell_i\circ F_i$. Here we view the vector space ${R_1}^r$ as a direct sum of $r$ many copies of the standard representation of $\sym_n$. 
If $W$ denotes the vector space spanned by $F_1, \dots, F_a$ we will also write $L_W$ instead of $L_{F_1, \dots, F_a}$.  We begin with the following.

\begin{lem}
\label{lem:psi}
Let $F_1, \ldots, F_r\in S$ and suppose $F_i$ is $\sym_n$-invariant for all $i, 1 \leq i \leq r$.  
Then the map $\psi$ in  \eqref{eq:Lsesreprieve} is $\sym_n$-equivariant and the subspace $L_{F_1, \dots, F_r} \subseteq R_1^r$ is stable under the $\sym_n$-action.
\end{lem}

\begin{proof}
Let $\sigma\in \sym_n$. We compute 
 \begin{align*}
\sigma\cdot \left(\psi(\ell_1, \ell_2, \dots, \ell_r)\right)&=\sum_{i=1}^r \sigma\cdot (\ell_i\circ F_i)=\sum_{i=1}^r (\sigma\cdot \ell_i)\circ (\sigma \cdot F_i) =   \sum_{i=1}^r (\sigma\cdot \ell_i)\circ F_i \\&=\psi\left(\sigma\cdot (\ell_1, \ell_2, \dots, \ell_r)\right)\,.
\end{align*}
where we use Remark~\ref{rem:MGRproperties}(4) in the second equality and the $\sym_n$-invariance of the $F_i$ in the third equality.  This shows that $\psi$ is equivariant and consequently its kernel, $L_{F_1, \dots, F_r}$ is stable under the $\sym_n$-action, as claimed. 
\end{proof}

\begin{rem}
The hypothesis that  $F_i$ are invariant polynomials is necessary in Lemma~\ref{lem:psi}. 
To see this, consider the $n-1$ polynomials $F_{i}=y_i^{(d)}-y_1^{(d)}$ for each $i>1$.  Note that these $F_i$ are not $\sym_n$-invariant. The corresponding $\psi$ then takes an $n-1$-tuple $(f_2,\cdots,f_n) \in R_1^{n-1}$ to $\psi(f_2,\cdots,f_n) = \sum_{i=2}^n f_i \circ F_i$, and it is straightforward to compute that $\psi(x_1,\cdots,x_1) = -(n-1)y_1^{(d-1)}$ whereas $\psi(x_2,\cdots,x_2) = y_2^{(d-1)}$. Taking $\sigma$ to be the simple transposition $(12)$ which swaps $1$ and $2$, we have $\sigma(x_1)=x_2$ and $\sigma \cdot \psi(x_1) \neq \psi(\sigma \cdot x_1)$. Thus $\psi$ is not $\sym_n$-equivariant in this example. 
\end{rem}

 Lemma~\ref{lem:psi} applies in particular to the special case where $F_\l=m_\l$ are the monomial symmetric polynomials in $S_{-d}$ for $\lambda \vdash n, \lambda \neq (d)$. There are $P(d)-1$ many such monomial symmetric functions. In~\eqref{eq: def W} we denoted by $W$  the subspace spanned by them. Applying Lemma~\ref{lem:psi} with $r=P(d)-1$, it follows that  the space of linear relations $L_W$ is a $\sym_n$-representation. Our next goal is to determine this action of $\sym_n$ on $L_{W}$; the answer is given in Proposition~\ref{prop:alternatesymmetry} below.

To prove Proposition~\ref{prop:alternatesymmetry}, we begin with a definition of a total order on the set of partitions of a fixed integer $d$. 
  
 \begin{defn}\label{def: lex order on partitions}
 Fix a positive integer $d$. 
 We define the {\em lexicographic (total) order}, denoted $>_{\lex}$, on the set of partitions of $d$ as follows: for $p=(p_1,p_2,\cdots,p_s)$ and $q = (q_1,\cdots,q_t)$ any two distinct partitions of $d$, we define 
\[p=(p_1, p_2, ..., p_s)>_{\lex} q=(q_1, q_2, ..., q_t)\]
 if there exists an integer $k \geq 1$ such that 
 \[
 p_1=q_1, p_2=q_2, \cdots , p_k=q_k, \, \text{ and } \, p_{k+1}>q_{k+1}. 
 \]
 \end{defn}

The following technical lemma is used in the sequel.

\begin{lem}
\label{lem:xi-xj} 
Let $F\in S_{-d}$ is a nonzero symmetric polynomial that does not contain $y_k^{(d)}$ in its support for any $k$. Then for every pair $i, j \in [n]$ with $i\neq j$, we have $(x_i-x_j)\circ F\neq 0$.
\end{lem}

\begin{proof}
The monomials $m_\l$ with $\l\vdash d$ form a basis for the subspace of symmetric polynomials in $S_{-d}$. Using also the assumption on the support of $F$, we can write
\[F=\sum_{\l\vdash d, \l\ne (d)} a_\l m_\l\qquad\text{with}\quad a_\lambda\in \kk\,.
\]
Define $\mu :=\max_{\lex}\{\l\vdash d \mid a_\l\neq 0\}$ where the maximum is taken with respect to the lex order $>_{\lex}$ of Definition~\ref{def: lex order on partitions} and note that by our assumption we know $\mu<(d)$. Hence, $\mu$ is a partition with at least $2$ parts. 

For the purposes of the proof, we say that the monomial $y_1^{(e_1)}y_2^{(e_2)}\dots y_n^{(e_n)}$ is divisible by $y_1^{(e_1')}y_2^{(e_2')}\dots y_n^{(e_n')}$ if  $e_i\ge e_i'$ for all $i\in [n]$.

Let $i,j \in [n]$ with $i \neq j$. We wish to show $(x_i - x_j) \circ F \neq 0$. Suppose in order to obtain a contradiction that $(x_i - x_j) \circ F = 0$, i.e., $x_i \circ F = x_j \circ F$. By the symmetry of $F$ and because $m_\mu$ appears in $F$ (i.e. $a_\mu\neq 0$), we know that there is a monomial in $F$ divisible by $y_i^{(\mu_1)}y_j^{(\mu_2)}$. From this we see that $x_j \circ F$ contains a monomial divisible by $y_i^{(\mu_1)}y_j^{(\mu_2-1)}$. By assumption $x_i \circ F = x_j \circ F$, so $x_i \circ F$ also contains such a monomial, from which it follows that there must appear in $F$ a monomial which is divisible by $y_i^{(\mu_1+1)}y_j^{(\mu_2-1)}$. But such a monomial has a type $\lambda$ with the property that $\lambda_1 > \mu_1$, which contradicts the choice of $\mu$ as the maximal in lex order. Thus we achieve a contradiction and the claim is proved. 
\end{proof} 

We can now compute the $\sym_n$-action on the space of linear relations $L_W$.  By definition, $L_W$ is a subspace of $R_1^{\oplus P(d)-1}$, the direct sum of $P(d)-1$ many copies of $R_1$, that is equal to the kernel of the map $\psi$ from \eqref{eq:Lsesreprieve}, which, in this case is  described by 
\[
\psi((\ell_{\lambda})_{\lambda\vdash d, \lambda\ne (d)})=\sum_{\lambda\vdash d, \lambda\ne (d)} \ell_\lambda\circ m_\lambda(y_1, \cdots, y_n)\,.
\]
 In what follows,  $(\sum_{i=1}^n x_i) R_0$ denotes the vector subspace of $R_1$ spanned by the linear form $\sum_{i=1}^nx_i$. With this notation in place, we can state and prove the following.

\begin{prop}
\label{prop:alternatesymmetry}
Let $W$ be as defined in~\eqref{eq: def W} and let $L_W$ denote the kernel of $\psi$ as above. 
Then $L_W\subseteq \left( (\sum_{i=1}^n x_i) R_0 \right)^{\oplus P(d)-1} \subseteq R_1^{\oplus P(d)-1}$. 
In particular, $L_{W}$ is a trivial $\sym_n$-representation. 
\end{prop}

\begin{proof}
Any element in $L_W \subseteq R_1^{\oplus P(d)-1}$ is a $(P(d)-1)$-tuple of linear polynomials, so for $\ell \in L_W$, we may write 
\[
\ell=\left(\sum_{i=1}^n c_{i,\l}x_i\right)_{\l\vdash d, \l\neq (d)}\in L_W\qquad\text{for}\quad c_{i,\l} \in \kk\,.
\]
To prove the first claim, it would suffice to show that $c_{i,\l}=c_{j,\l}$ for all $1\leq i,j\leq n$.

To see this, recall from Lemma~\ref{lem:psi} that $L_W$ is stable under the $\sym_n$-action. Thus for any $i,j \in [n]$, $i \neq j$, we have that $(i\, j) \cdot \ell\in L_W$. In other words, $\psi(\ell) = 0$ and $\psi((i\, j) \cdot \ell)=0$. We now compute each explicitly. 

We have 
\begin{equation}
\label{eq: psi ell}
\psi(\ell)=\sum_{\l\vdash d , \l \neq(d)}\left(\sum_{k=1}^n c_{k,\l}x_k\circ m_\l\right)=0.
\end{equation}
We may also compute 
\begin{equation} 
\label{eq: psi ij ell}
\psi((i \, j)\cdot \ell)=
\sum_{\l\vdash d , \l \neq(d)}\left( \left( \sum_{\scriptscriptstyle k\neq i,j } c_{k,\l}x_k\circ m_\l \right) +c_{j,\l}x_i\circ m_\l+c_{i,\l}x_j\circ m_\l\right) = 0.
\end{equation} 
Subtracting \eqref{eq: psi ij ell} from \eqref{eq: psi ell}, we obtain 
\[
\sum_{\l\vdash d , \l \neq(d)}(c_{i,\l}-c_{j,\l})(x_i-x_j)\circ m_\l=0\,,
\]
which is equivalent to 
\[
(x_i-x_j)\circ\left(\sum_{\scriptscriptstyle\l\vdash d , \l \neq(d)}(c_{i,\l}-c_{j,\l}) m_\l\right)=0\,.
\]
Note that the polynomial being contracted by $x_i-x_j$ above is a symmetric polynomial that does not contain $x_i^d$ in its support. It follows by \Cref{lem:xi-xj} that this must be the zero polynomial. This implies that $c_{i,\lambda}=c_{j,\lambda}$ for all $\lambda \vdash d, \lambda \neq (d)$. Moreover, the above argument is valid for any pair of $i,j \in [n]$ with $i \neq j$, so the claim is proved. Finally, the linear polynomial $\sum_{i=1}^n x_i$ is $\sym_n$-invariant, so the $\kk$-vector space it spans is the trivial $\sym_n$-representation, as is its direct sum $\left( (\sum_{i=1}^n x_i) R_0 \right)^{\oplus P(d)-1}$. Any subspace of a trivial representation is still trivial, so we may conclude $L_W$ is also the trivial $\sym_n$-representation. This concludes the proof. 
\end{proof}

\subsection{Linear relations on $W(\underline{t})$}

We now turn our attention to the general case of the subspace $W(\underline{t})$ from \eqref{eq:W(t)}. The subspace $W$ studied in the previous subsection is the special case when $t_\lambda=0$ for all $\lambda$.   It will turn out that our analysis of the linear relations $L_W$ in this special case will allow us to compute the dimension of $L_{W(\underline{t})}$ for $\underline{t} = (t_\lambda)\in \mathbb{A}^{P(d)-1}$ general (in a suitable sense to be made precise in Proposition~\ref{prop: general Aprime}), and in addition, it will allow us to show that for such $\underline{t}$, the subspace $L_{W(\underline{t})}$ carries a trivial $\sym_n$-action -- giving us a  natural generalization of \Cref{prop:alternatesymmetry}.  These statements are contained in Corollary~\ref{cor:L_A}. 

We first need some notation.  For an arbitrary tuple of non-negative integers $\alpha=(\alpha_1, \ldots, \alpha_s) \in \Z_{\geq 0}^n$ we denote by $\pa(\alpha)$ the tuple with the same entries as $\alpha$ but with entries re-ordered to be non-increasing, so that $\pa(\alpha)$ is a partition. 

\begin{notation}
For a partition $p=(p_1,\ldots, p_s)$ denote by $\#p$ the length of the partition, so $\#p = s$ in the example given.
For any $i$ with $1 \leq i \leq \#p+1$, define the notation 
\[
p_{\uparrow i} := 
\begin{cases}
\pa(p_1, \ldots, p_{i-1}, p_i+1, p_{i+1}, \ldots, p_s) & \text{ if }1\leq i\leq \#p=s \\
(p_1, \ldots, p_s, 1) & \text{ if }\, i=\#p+1 = s+1. 
\end{cases}
\]
In particular, $p_{\uparrow i}$ is undefined for any $i$ with $i \not \in [\#p+1]$. 
\end{notation}

\begin{ex}
Assume $d>4$. Let $p=(d-3, 1,1)$.  Then  
\[
p_{\uparrow 1}= (d-2, 1,1), \quad p_{\uparrow 2}= p_{\uparrow 3}= (d-3, 2,1), \quad  p_{\uparrow 4}=(d-3, 1,1,1).
\]
\end{ex}

Our second piece of notation involves the support of vectors in $\Z_{\geq 0}^n$. Let $\alpha\in \Z_{\geq 0}^n$. Recall that the {\em support} of $\alpha$ is the set 
\[
\Supp(\alpha)=\{i \mid \alpha_i\neq 0\} \subseteq [n]. 
\]
We also denote by $e_1, \ldots, e_n$ the standard basis of $\Z_{\geq 0}^n$.

\begin{defn}
\label{def:diff}
Let $\alpha\in \Z_{\geq 0}^n$ with $\pa(\alpha)=p$.  We define a function 
\[
\diff_\alpha:\Supp(\alpha)\to \{1,\ldots, \#p\}
\]
such that  if $i \in \Supp(\alpha)$, then $\diff_\alpha(i)$ is the unique index  $1\leq \diff_\alpha(i)\leq \#p$ such that 
\[ \pa(\alpha)_{\diff_\alpha(i)}\neq \pa(\alpha+e_i)_{\diff_\alpha(i)}.
\]
\end{defn}

\begin{ex}
Let $\alpha=(1, d-2, 1, 0^{n-3})$. Then $\Supp(\alpha)=\{1,2,3\}$. Assume $d>3$. Then $\pa(\alpha)=(d-2, 1,1)$ and we have $\alpha+e_1 = (2,d-2,1,0^{n-3})$ and $\pa(\alpha+e_1)=(d-2,2,1)$ so $\diff_\alpha(1) = 2$. Similarly $\alpha+e_2=(1,d-1,1,0^{n-3})$ so $\pa(\alpha+e_2)=(d-1,1,1)$ and $\diff_\alpha(2)=1$. A similar computation yields $\diff_\alpha(3)=2$. 
\end{ex}

The following is a straightforward consequence of definitions and is left to the reader. 

\begin{lem}\label{lem: uparrow p}
Let $\alpha \in \Z_{\geq 0}^n$, and let $\pa(\alpha)=p$. 
\begin{enumerate} [\quad\rm(1)]
\item For $i \not \in \Supp(\alpha)$, we have $\pa(\alpha+e_i) = p_{\uparrow \#p+1}$. 
\item For $i \in \Supp(\alpha)$, we have $\pa(\alpha+e_i) = p_{\uparrow \diff_\alpha(i)}$. 
\item The multisets $\{p_{\uparrow \diff_\alpha(i)} \, \mid \, i \in \Supp(\alpha)\}$ and $\{p_{\uparrow i} \, \mid \, 1 \leq i \leq \#p\}$ are equal. 
\end{enumerate} 
\end{lem}

We may also succinctly express the action of a monomial on elements of $W$ using this notation.

\begin{lem}
\label{lem:actbyalpha}
 Let $p$ be a partition of $d-1$. Let $\alpha\in \Z_{\geq 0}^n$ with $\pa(\alpha)=p$. For each $i \in [n]$ and $\l \vdash d$, let $c_{i,\l}$ be an element of $\kk$.  Then
\begin{equation}
\label{eq:support}
x^\alpha \circ \left(\sum_{\scriptscriptstyle,\lambda\vdash d} c_{i,\lambda}x_i\circ m_\lambda\right)= \sum_{i\notin \Supp(\alpha)} c_{i, p_{\uparrow \#p+1}}+\sum_{i\in \Supp(\alpha)} c_{i, p_{\uparrow \diff_\alpha(i)}}.
\end{equation}
\end{lem}

\begin{proof}
For any $\beta \in \Z_{\geq 0}^n$ and corresponding monomial $x^{\beta} := x_1^{\beta_1} x_2^{\beta_2} \cdots x_n^{\beta_n}$ in $R = \kk[x_1,\cdots,x_n]$, use \cref{o:ml-contract} to see that $x^\beta \circ m_\lambda$ is non-zero if and only if $\type(x^\beta)=\lambda$, and this occurs if and only if $\pa(\beta)=\lambda$. Moreover, if $x^\beta \circ m_\lambda\ne 0$, then $x^\beta \circ m_\lambda=1$. 

From the above paragraph, we see that in order to compute the LHS (and show that it is equal to the RHS) of~\eqref{eq:support} it suffices to find conditions on $\alpha$ such that $x^\alpha x_i$ has type $\lambda$.  Notice that in order for this to occur, it is necessary that $\pa(\alpha)$ be a partition of $\lvert \lambda \rvert -1 = d-1$ (as we have assumed). We take cases according as to whether $x_i$ appears in $x^\alpha$, i.e., whether or not $i$ is in $\Supp(\alpha)$. 

Suppose $i\notin \Supp(\alpha)$. Then $\pa(\alpha+e_i)=p_{\uparrow \#p+1}$ and thus 
\[
x^\alpha x_i\circ m_\l= 
\begin{cases}
1 &\text{ if and only if } p_{\uparrow \#p+1}=\l \\
0 &\text{ else}.
\end{cases}
\]

If $i\in \Supp(\alpha)$ then $\pa(\alpha+e_i)=p_{\uparrow \diff_\alpha(i)}$ by Lemma~\ref{lem: uparrow p}(2) and thus 
\[
x^\alpha x_i\circ m_\l= 
\begin{cases}
1 &\text{ if and only if } p_{\uparrow \diff_\alpha(i)}=\l \\
0 &\text{ else}.
\end{cases}
\]
The claim of the lemma now follows. 
\end{proof}

Recall that $L_{W(\underline{t})}$ is defined to be the kernel of the map in \eqref{eq:Lsesreprieve}, given by $$\psi: (\ell_\lambda)_{\lambda \vdash d, \lambda \neq (d)} \in R_1^{\oplus P(d)-1} \mapsto \sum_{\lambda \vdash d, \lambda \neq (d)} \ell_\lambda \circ (m_\l - t_\l m_{(d)}).$$
Writing $\ell_\lambda = \sum_{i=1}^n c_{i,\l} x_i$ for each $\l$, we then see that $L_{W(\underline{t})}$ consists of tuples $(\sum_{i=1}^n c_{i,\l}x_i)_{\l\neq(d)}$ where the $(c_{i,\lambda})$ must satisfy the equation 
\begin{equation}\label{eq: general linear relation with c}
\sum_{i, \lambda \vdash d, \lambda \neq (d)} c_{i,\lambda} x_i \circ (m_\lambda - t_\lambda m_{(d)}) = 0. 
\end{equation} 
Note that the LHS of~\eqref{eq: general linear relation with c} is a homogeneous polynomial of degree $d-1$. Given an element $g \in S$, we have that $x^\alpha \circ g$ is the coefficient in $g$ of the monomial with exponent vector $\alpha$, for any such $\alpha$.  Thus~\eqref{eq: general linear relation with c} is equivalent to the statement that 
\begin{equation}\label{eq: general linear relation with x alpha}
x^\alpha \circ \left( \sum_{\scriptscriptstyle i, \lambda \vdash d, \lambda \neq (d)} c_{i,\lambda} x_i \circ (m_\lambda - t_\lambda m_{(d)}) \right) = 0 \, \textup{ for all } \, \alpha \in \N^n \, \textup{ with } \, \lvert \alpha \rvert = d-1.
\end{equation}

For a given $\alpha$, let $q$ be defined by $q := \pa(\alpha)$. Note that if $q \neq (d-1)$, then $x^\alpha  x_i$ is not of the form $x_i^d$ for any $i \in [n]$, so $x^\alpha \circ (x_i \circ m_{(d)})=0$ in that case. Hence if $\pa(\alpha)=q \neq (d-1)$ the equation~\eqref{eq: general linear relation with x alpha} is equivalent to 
$$
x^\alpha \circ \left( \sum_{\scriptscriptstyle i, \lambda \vdash d, \lambda \neq (d)} c_{i,\lambda} x_i \circ m_\lambda \right) = 0 
$$
which in turn is equivalent to 
\begin{equation}\label{eq: D most q}
 \sum_{i \not \in \Supp(\alpha)} c_{i, q_{\uparrow \#q+1} } + \sum_{i \in \Supp(\alpha)} c_{i, q_{\uparrow \diff_\alpha(i)}} = 0 
\end{equation}
by \Cref{lem:actbyalpha}. 
 On the other hand, when $\pa(\alpha)=q=(d-1)$, which means $x^\alpha = x_k^{d-1}$ for some $k \in [n]$, then~\eqref{eq: general linear relation with x alpha} is equivalent to the equations 
\begin{equation}\label{eq: D last q} 
\sum_{1 \leq i \leq n, i \neq k} c_{i, (d-1,1)} - \sum_{\lambda \vdash d, \lambda \neq (d)} c_{k,\lambda} t_\lambda = 0.
\end{equation} 
Note that the equations \eqref{eq: D last q} are the only ones in which the parameters $t_\lambda$ appear.

One of the goals of this section is 
to find the dimension of the space of solutions $\{c_{i,\lambda}\}$ of the equations~\eqref{eq: D most q} and~\eqref{eq: D last q}, which computes the dimension of $L_{W(\underline t)}$.  For the convenience of the reader, we now give an overall sketch of our strategy before handling the details. First, we introduce some notation. We let $D(\underline{t})$ denote the dimension of the space of solutions of the system of equations~\eqref{eq: D most q} and~\eqref{eq: D last q}, and note that $D(\underline t)=\dim_\kk L_{W(\underline t)}$. We denote by $C(\underline{t})$  the dimension of the space of solutions to~\eqref{eq: D most q} and~\eqref{eq: D last q} under the additional condition that the solution $\{c_{i,\lambda}\}$ must be symmetric in the $i$, that is, $c_{i,\lambda}=c_{j,\l}$ for all $1\leq i,j\leq n$. We also let $D_{generic}$, respectively $C_{generic}$, denote the minimum value attained by $D(\underline t)$, respectively $C(\underline{t})$ for $\underline t\in \mathbb A^{P(d)-1}$. The reason for the use of the word ``generic'' is that the entries of coefficient matrices of the respective systems of equations are linear functions in the parameters $\underline t\in \mathbb A^{P(d)-1}$, and thus these matrices attain maximal rank (and hence the solution spaces of the corresponding systems have minimal dimension) when $\underline t$ is in a non-empty open subset of $\mathbb A^{P(d)-1}$. We also set $C_0=C(0)$ and $D_0=D(0)$; these are the dimensions of the two solutions sets that are obtained when $\underline t=0$, meaning $t_\l=0$ for all $\l\vdash d$, $\l\ne (d)$. Our proof can now be summarized through the following steps:
\begin{itemize}
\item In \Cref{prop:alternatesymmetry} we proved that if $\underline t=0$, then any solution to~\eqref{eq: D most q} and~\eqref{eq: D last q} must satisfy that the $c_{i,\lambda}$ are independent of $i$. This implies $C_0=D_0$.
\item Observe that $C(\underline{t}) \leq D(\underline{t})$ for all $\underline t\in \mathbb A^{P(d)-1}$, since the system corresponding to $C(\underline t)$ is obtained from the system corresponding to $D(\underline t)$ by adding additional linear equations that symmetrize the solutions. We have thus inequalities
\begin{align}
\label{eq:cd}
C(\underline t)&\leq D(\underline t)\le D_0=C_0\\
\label{eq:inequalities}
C_{generic} &\leq D_{generic} \leq D_0 = C_0.
\end{align}
\item We explicitly compute $C_{generic}$ and we prove $C_{generic}=C_0$. This implies that equalities must hold in \eqref{eq:inequalities}. More precisely, \eqref{eq:cd} gives that if $V$ is the nonempty open subset $V\subseteq \mathbb A^{P(d)-1}$ on which $C(\underline t)=C_0$, then 
$D(\underline t)=C(\underline t)=C_0$ for all $\underline t\in V$.  A direct consequence of the equality $D(\underline t)=C(\underline t)$ is that, when $\underline t\in V$, any solution $\{c_{i,\lambda}\}$ for the equations~\eqref{eq: D most q} and~\eqref{eq: D last q} satisfies $c_{i,\lambda}=c_{j,\lambda}$ for all $i,j\in [n]$; this is equivalent to the fact that $L_{W(\underline{t})}$ is a subspace of $ \left(\sum_{i=1}^n x_i\right) R_0^{P(d)-1}$. 
\end{itemize}

To fill in the details of the sketch-of-proof outlined above -- namely, to compute $C_{generic}$ and prove that $C_0 = C_{generic}$ -- we need some preliminaries. We will first analyze and rewrite the equations~\eqref{eq: D most q} and~\eqref{eq: D last q} under the additional assumption that $c_{i,\lambda}$ is independent of $i$. In this case, we may use the simplified notation $c_\lambda := c_{i,\lambda}$. Moreover, by Lemma~\ref{lem: uparrow p}(3), we know (using the $c_\lambda$ notation just introduced) 
$$
\sum_{i \in \Supp(\alpha)} c_{q_{\uparrow \diff_\alpha(i)}} = \sum_{j=1}^{\# q} c_{q_{\uparrow j}}.
$$
Therefore, under the symmetry assumption on the coefficients, the equation~\eqref{eq: D most q} becomes 
\begin{equation}\label{eq: Eq final} 
 (n- \#q) c_{q_{\uparrow \#q+1} } + \sum_{j=1}^{\#q}  c_{q_{\uparrow j}} = 0 
\end{equation} 
for each $q \vdash (d-1)$ with $q \neq (d-1)$. 
For $q=(d-1)$, the equation~\eqref{eq: D last q} becomes  
\begin{equation} \label{Eqlast}
\left(n-1-t_{(d-1,1)}\right)c_{(d-1,1)}- \sum_{\lambda\vdash d, \lambda\ne (d), (d-1,1)}t_{\lambda}c_{\lambda}=0. 
\end{equation} 
As suggested above, we now view these equations as linear equations in the variables $c_\l$, and the coefficients of these linear equations depend on the parameters $\underline{t} \in \mathbb{A}^{P(d)-1}$. 

We illustrate the corresponding matrix $A$ of coefficients associated to the above equations~\eqref{eq: Eq final} and~\eqref{Eqlast} in a specific instance of $d$ below. Indeed, for this specific case, we can also compute the (generic) rank of $A$, and the techniques used in the example below illustrates the general argument we give below in \Cref{prop: general Aprime}.  Our reasoning utilizes the lexicographic order in \Cref{def: lex order on partitions}.

\begin{ex}\label{ex:d=5} 
Assume ${\rm char}(\kk)=0$ and assume  $n\geq d=5$. In this case, there are 6 partitions $\lambda$ of $d=5$ with $\l \neq (5)$, namely 
$$
\{(1,1,1,1,1), (2,1,1,1), (2,2,1), (3,1,1), (3,2), (4,1)\}
$$
listed in increasing lex order. There are $4$ partitions $q$ of $d-1=4$ with $q \neq (d-1)=(4)$: 
$$
\{(1,1,1,1), (2,1,1), (2,2), (3,1)\}
$$ 
listed again in increasing lex order. Thus by varying $q \vdash 4, q \neq (4)$ in~\eqref{eq: Eq final} above, we obtain the following 4 equations: 
\begin{align*}
&(n-4)c_{(1,1,1,1,1)}+4c_{(2,1,1,1)}=0\\
&(n-3)c_{(2,1,1,1)}+2c_{(2,2,1)}+c_{(3,1,1)}=0\\
&(n-2)c_{(2,2,1)}+2c_{(3,2)}=0\\
&(n-2)c_{(3,1,1)}+c_{(3,2)}+c_{(4,1)}=0.
\end{align*}
Finally, taking $q=(4)$ we rearrange~\eqref{Eqlast} so the coefficients $c_\lambda$ appear in increasing lex order and we obtain
$$
- t_{(1,1,1,1,1)}c_{(1,1,1,1),1}-t_{(2,1,1,1)}c_{(2,1,1,1)}-t_{(2,2,1)}c_{(2,2,1)}-t_{(3,1,1)}c_{(3,1,1)}-t_{(3,2)}c_{(3,2)}+(n-1-t_{(4,1)})c_{(4,1)}=0. 
$$
The corresponding matrix $A$, with columns indexed by $\lambda \vdash 5$ with $\lambda \neq (5)$ listed in increasing lex order, and rows indexed by $q \vdash (4)$ in increasing lex order, is as follows:  
\[
A = \begin{bmatrix}
(n-4) & 4 &0&0&0&0\\
0& (n-3) & 2 &1 &0&0 \\\\
0& 0& (n-2)& 0 &2 &0\\\\
0& 0&0 & (n-2) & 1 &1\\
- t_{(1,1,1,1,1)} & -t_{(2,1,1,1)} &  -t_{(2,2,1)} & -t_{(3,1,1)}  &-t_{(3,2)}&(n-1-t_{(4,1)})
\end{bmatrix}. 
\]
We now argue that for a generic choice of $(t_\lambda)$, this $5 \times 6$ matrix has full rank. In order to see this, we consider the $5 \times 5$ minor of $A$ obtained by deleting the column corresponding to partitions $\lambda \vdash (5)$ that do not end with a $1$. Since partitions of $\lambda \vdash (5)$ that end with a $1$ (e.g. $(2,2,1)$) are in bijective correspondence with partitions $q$ of $(4)$ (e.g. $(2,2)$) by deleting the last ``$1$'' entry, it is straightforward to see that the resulting minor is indeed square, of size $5 \times 5 = P(4) \times P(4)$. It is also straightforward to see that this construction holds for general $d>4$, resulting in a $P(d-1) \times P(d-1)$ matrix. In this case of $d=5$, the resulting minor is: 
\begin{equation}\label{eq: Aprime}
A' = \begin{bmatrix}
(n-4) & 4 &0&0&0\\\\
0& (n-3) & 2 &1&0 \\
0& 0& (n-2)& 0 &0\\\\
0& 0&0 & (n-2) &1\\
- t_{(1,1,1,1,1)} & -t_{(2,1,1,1)} &  -t_{(2,2,1)} & -t_{(3,1,1)}  &(n-1-t_{(4,1)})\\
\end{bmatrix}. 
\end{equation}

We see that $A'$ is almost upper-triangular, in the sense that any entry below the diagonal and above the bottom row are equal to $0$. Since $n=d \geq 5$ we also know that the first $4$ entries along the main diagonal are positive. It follows, by performing basic row operations which clear the leftmost $4$ entries along the bottom row (using the first $4$ entries along the diagonal), that 
\begin{equation}\label{eq:det}
\det(A') = (n-4)(n-3)(n-2)^2 \cdot g(\underline{t})
\end{equation}
where $g(\underline{t})$ is a polynomial in the $t_\lambda$. In fact, using the standard definition of a determinant as a sum over permutations $\sigma$ in $S_5$ of the products $(A')_{1,\sigma(1)}(A')_{2,\sigma(2)} \cdots (A')_{5,\sigma(5)}$ of matrix entries, that in fact the polynomial $g(\underline{t})$ must be linear in the variables $t_\lambda$. Moreover, $g(\underline{t})$ is not identically $0$, since it is evident from~\eqref{eq: Aprime} that $g(0) = n-1$. Thus, for any tuple $(t_\lambda) \in \mathbb{A}^5 \setminus \mathbb{V}(g)$ where $\mathbb{V}(g)$ denotes the vanishing locus of the non-zero linear polynomial $g(\underline{t})$, we have $\det(A') \neq 0$ and hence the matrix $A$ is full rank, as desired. 
In the terminology of the discussion above, the non-empty Zariski-open subset $V$ with $C(\underline{t})$ minimal is precisely $V := \mathbb{A}^5 \setminus \mathbb{V}(g)$.

From the fact that $A$ is full rank for $(t_\lambda) \in V$, and the fact that $A$ is of size $P(4) \times (P(5)-1)$, we conclude that for such $\underline{t} = (t_\lambda)$, we have $C(\underline{t}) = C_{generic} = P(5) - P(4)-1$. 
\end{ex}

The above example clearly illustrates the idea of the general argument, which we formalize with the next result. Equation \eqref{eq:det} also illustrates the necessity of the assumption on the characteristic of the base field, if we wish to avoid situations in which the determinant of $A'$ (or $A$) is identically $0$.

\begin{prop}\label{prop: general Aprime} 
Suppose ${\rm char}(\kk )=0$ and fix integers $n \geq d \geq 2$. Let $A$ denote the $P(d-1) \times (P(d)-1)$ matrix of coefficients, which depends on $(t_\lambda)$, corresponding to the linear equations~\eqref{eq: Eq final} and~\eqref{Eqlast}.  Then there exists a non-empty Zariski-open subset $V \subseteq \mathbb{A}^{P(d)-1}$ such that for $(t_\lambda) \in V$, the matrix $A$ has full rank. Moreover, 
the origin $(t_\lambda) = (0)$ is contained in $V$, and 
$$
C_0 = C_{generic} = P(d)-P(d-1)-1.
$$
\end{prop} 

\begin{proof} 
The general argument follows closely the argument given for $d=5$ in the example above. Let $A'$ denote the $P(d-1) \times P(d-1)$ minor of $A$ obtained by deleting the columns corresponding to $\lambda \vdash (d)$ that do not end with a $1$. This results in a $P(d-1) \times P(d-1)$ minor since the partitions $\lambda \vdash d$ that do end with a $1$ are in bijective correspondence with partitions $q \vdash d-1$, by deleting the last $1$. In particular, each such $\lambda \vdash d$ can be written as $\lambda = q_{\uparrow \#q+1}$ for a unique $q \vdash d-1$. Note also that this bijective correspondence respects the lexicographic  order, as can be checked. 

We next claim that $A'$ is almost-upper-triangular in the sense that the $(q, \lambda = q'_{\uparrow \# q' + 1})$-th entry of $A'$ equals $0$ if $q>q'$ and $q \neq (d-1)$. In other words, $A'$ is upper-triangular except for the bottom row. To see this, suppose $q \vdash d-1$ and $q \neq (d-1)$. Then by the form of equation~\eqref{eq: Eq final} it follows that the coefficient $A_{q,\lambda}$ is non-zero if and only if either 
$$
\lambda = q'_{\uparrow \# q' +1} = q_{\uparrow \#q+1}
$$
or 
$$
\lambda = q'_{\uparrow \# q' +1} = q_{\uparrow \diff_q(i)}
$$ 
for $1 \leq i \leq \#q$, and $q_{\uparrow \diff_q(i)}$ ends with a $1$. In the first case, we have $q=q'$.  In the second case, by definition $q_{\uparrow \diff_q(i)}$ is a partition obtained from $q$ by increasing one of its entries by $1$ and then re-ordering to make it a partition. In particular, by the definition of lex order it follows that $q_{\uparrow \diff_q(i)} > q_{\uparrow \#q+1}$ and since $q'$ is defined by $q_{\uparrow \diff_q(i)} = q'_{\uparrow \#q'+1}$ it follows that $q'>q$. Thus we have shown that $A'$ is upper-triangular except for the bottom row, as claimed. Moreover, since $n \geq d$ we also know from~\eqref{Eqlast} that the diagonal entries $A_{q, q_{\uparrow \#q+1}}$ for $q \vdash (d-1)$ and $q \neq (d-1)$ are non-zero, since $n-\#q > 0$ for all such $q$. 

Now, from the form of~\eqref{Eqlast}, it immediately follows that every entry in the bottom row of $A'$ (which corresponds to $q=(d-1)$) is non-zero, and each non-diagonal entry is a single variable $t_\lambda$.  Moreover, each $t_\lambda$ which appears along the row appears exactly once. As argued above for in \Cref{ex:d=5}, by elementary row operations it now follows that 
$$
\det(A') = \left( \prod_{\scriptscriptstyle q \vdash d-1, q \neq (d-1)} (n- \# q) \right) \cdot g(\underline{t}) 
$$
for some linear polynomial $g(\underline{t})$ in the $t_\lambda$ parameters. Moreover, since the bottom right corner entry of $A'$ is $n-1-t_{(d-1,1)}$, it is evident that if all $t_\lambda=0$ then 
$$
\det(A') =  \left(  \prod_{\scriptscriptstyle q \vdash d-1, q \neq (d-1)} (n- \# q) \right) \cdot (n-1) \neq 0.
$$
In other words, $g(\underline{0}) \neq 0$. Now let $V := \mathbb{A}^{P(d-1)} \setminus \mathbb{V}(g)$ where $\mathbb{A}^{P(d-1)}$ is the space of tuples $(t_\lambda)$. Arguing as in \Cref{ex:d=5}  above, it follows that $V$ is a non-empty Zariski-open set containing $(0)$. Thus for $(t_\lambda) \in V$ we conclude $\det(A') \neq 0$ and hence the original matrix $A$ is full rank. The dimension of the space of solutions is therefore $P(d)-1-P(d-1)$ as desired, and the rest of the claims follows as in Example~\ref{ex:d=5}. 
\end{proof}

Returning now to the discussion before Example~\ref{ex:d=5}, we observe that Proposition~\ref{prop: general Aprime} computes that 
$$
C_{generic} = C_0 = P(d)-P(d-1)-1. 
$$
Recall that, by the previous discussion, this implies 
$$
C_{generic} = D_{generic} = D_0 = C_0 = P(d)-P(d-1)-1.
$$
We have thus obtained the following Corollary. 

\begin{cor} 
\label{cor:L_A}
Assume ${\rm char}(\kk )=0$ and fix $n \geq d \geq 2$. There exists a non-empty Zariski-open subset $V \subseteq \mathbb{A}^{P(d)-1}$ of tuples $(t_\lambda)$ such that $0\in V$ and, for $(t_\lambda) \in V$, the following hold: 
\begin{enumerate} [\quad\rm(1)]
\item  The dimension of the solution set of the equations~\eqref{eq: D most q} and~\eqref{eq: D last q} in the variables $c_{i,\lambda}$ is equal to $P(d)-P(d-1)-1$, i.e., $$\dim_\kk L_{W(\underline{t})}=P(d)-P(d-1)-1$$ for $\underline{t} \in V$. 
\item Any such solution must satisfy the property that the $c_{i,\lambda}=c_{j,\lambda}$ for all $i,j\in [n]$, i.e., $$L_{W(\underline{t})}\subseteq  \left(\sum_{i=1}^n x_i\right) R_0^{P(d)-1}.$$ In particular, $L_{W(\underline{t})}$ is a trivial $\sym_n$-representation.
\end{enumerate}

\end{cor} 

The above arguments show an open set of parameters $t_\lambda$ with the desired properties. In the next section, we will translate to a Zariski open subset of the parameter space for the polynomials $f$ that define the principal symmetric ideals of interest for this paper. This set will then be used in describing the generic condition on  these ideals that is  needed towards  the main theorem of this paper. 

\section{Main theorem, and the case of cubic principal symmetric ideals}\label{sec:main}

The present section contains the statement and proof of our main result, \Cref{introthm}. Some preliminary results are needed for the proof, since the argument relies on the theory developed for extremely narrow algebras in \cref{s:narrow}.  Thus we first prove Theorem~\ref{thm:mainnarrow}, which shows that the quotient algebras corresponding to general principal symmetric ideals are extremely narrow algebras.  We also need Theorem~\ref{thm:equivariantTors}, which gives us an understanding of the $\sym_n$-equivariant structure of the minimal resolutions which arise.  Finally, in Subsection \ref{subsec: cubic case}, we illustrate in some detail the case of cubic ideals, where $\mathrm{deg}(f) = 3$. In part, this serves as an illustration of our main results, but in addition, we present this case because in this situation we can prove some additional facts and give a more detailed analysis than in the general case. (Recall that the case $\mathrm{deg}(f) = 2$ was fully treated in \cref{sec: quadratic symmetric}.)

\subsection{Main theorems}\label{subsec: main theorems} 

We begin with the following result, which shows that for general principal symmetric ideals, the corresponding quotient is extremely narrow. 
Recall that $P(d)$ denotes the number of partitions of $d$.

\begin{thm}
\label{thm:mainnarrow}
Assume $\kk$ is infinite with ${\rm char}(\kk )=0$ and fix integers $d\geq 2$ and $n\gg0$.
If $I=(f)_{\sym_n}$ is a general principal symmetric ideal generated by a homogeneous polynomial $f$ of degree $d$, then the quotient algebra $A=R/I$ is a $d$-extremely narrow algebra. Moreover, the socle polynomial of $A$ is 
$$\left(\binom{n+d-2}{d-1}-(n-1)P(d)-P(d-1)+n-1\right)z^{d-1}+(P(d)-1)z^d.
$$
\end{thm}

\begin{proof}
Let  $c\in \P^{N-1}$ and let  $f_c\in R$ of degree $d$ be the corresponding polynomial, as defined in \cref{s:prelim}. Let $\alpha_{(d)}$ denote the sum of the coordinates of $c$ corresponding to coefficients of the monomials $x_i^d$ appearing in $f_c$. 
Since we are interested in general $f_c$, we may restrict to working with $c$ in the Zariski-open set where the sum $\alpha_{(d)}$ is nonzero -- the complement of the hyperplane $\mathbb{V}(\alpha_{(d)})$.

For each $\l\vdash d$, let $\alpha_\l$ denote the sum of all coefficients of monomials of type $\l$ in the support of $f_c$. For $c\in U$ and for $\l\vdash d$ set $t_\l=\alpha_\l/\alpha_{(d)}$. This gives a morphism
\[
\Psi:U\to \A^{\P(d)-1}, \quad c\mapsto t=(t_\l).
\]

To show that $A$ is $d$-extremely narrow, we need to verify conditions (1) and (2) in \cref{strategy}. To prove (1), use \Cref{lem:generalexample} and upper semicontinuity (\cref{upper-s}) to conclude 
$$
 \dim_\kk I_d\ge \dim_\kk R_d-(P(d)-1)
$$
for $f=f_c$ with $c\in U$, with $U$ a nonempty Zariski open set. 

For each partition $\lambda$ of $d$, $\lambda\ne (d)$, define polynomials
\begin{equation*}
\label{eq:Flambda}
F_\lambda=\alpha_{(d)}m_\lambda-\alpha_\lambda m_{(d)}.
\end{equation*}
and note that $\langle F_\lambda \mid \l\vdash d, \l\neq (d) \rangle=W(\underline{t})$ in the notation of \eqref{eq:W(t)}. 
Since the polynomials $m_\lambda$ have disjoint support, they are linearly independent and consequently the polynomials $F_\lambda$ are also linearly independent.
In \Cref{cor:L_A} we have identified a nonempty Zariski-open set $V\subseteq \A^{P(d)-1}$ so that $0\in V$. Take $V'=\Psi^{-1}(V)$. Then $V'$ is Zariski-open and non-empty as we have seen that $\Psi(c)=0$ for the polynomial $f_c$ in \eqref{eq:large f}, so that $V'$ is non-empty.

\Cref{cor:L_A} shows that for $c\in U\cap V'$ one has
\begin{equation}
\label{eq:dimL}
\dim_\kk L_{W(\underline{t})}=P(d)-P(d-1)-1\qquad\text{and}\qquad L_{W(\underline{t})}\subseteq \left(\sum_{i=1}^n x_i \right)R_0^{P(d)-1}\,.
\end{equation}
In particular, the inclusion above implies that $L_{W(\underline{t})}$ carries a trivial $\sym_n$-structure. 

Let $c\in U\cap V'$. Since $ \{F_\l \mid \l\neq(d)\}$ is linearly independent and is contained in $(I^\perp)_{-d}$ by \Cref{lem:W(t)}, equation \eqref{eq:dimL} allows us to apply 
  \cref{strategy}  with $a=P(d)-1$, and conclude that $A$ is $d$-extremely narrow with 
socle polynomial $bz^{d-1}+(P(d)-1)z^d$, where 
\begin{eqnarray*}
b &=& \dim_\kk R_{d-1}-(P(d)-1)n +\dim_\kk L_{F_\lambda, \lambda\ne (d)}\\
&=& \dim_\kk R_{d-1}-(P(d)-1)(n-1)+P(d-1). \qedhere
\end{eqnarray*}
\end{proof}

\begin{cor}
\label{cor:I perp equal W(t)} 
Under the assumption of \Cref{thm:mainnarrow} for a general principal sym\-me\-tric ideal $(f)_{\sym_n}$ we have an equality in \Cref{lem:W(t)}. Namely, 
\[
\left((f)_{\sym_n}^\perp\right)_{-d}=W(\underline{t})
\]
for the tuple $\underline{t}$ defined therein. In particular, this vector space carries a trivial $\sym_n$-action.
\end{cor}
\begin{proof}
Since \Cref{lem:W(t)} provides a containment, the desired equality follows by establishing equality for the dimensions of the two vector spaces. By \Cref{thm:mainnarrow} we have $\dim_\kk\left((f)_{\sym_n}^\perp\right)_{-d}=P(d)-1$. Since the spanning set of $W(\underline{t})$ in \eqref{eq:W(t)} is a basis, due to the linear independence of the monomial symmetric functions, we conclude that $\dim_\kk W(\underline{t})=P(d)-1$ as well. Finally, since $W(\underline{t})$ is spanned by $\sym_n$-invariant polynomials we conclude it is a trivial representation.
\end{proof}

For our main result we also need to understand the equivariant structure of the resolution of a principal symmetric ideal. This is summarized in the following result.

\begin{thm}
\label{thm:equivariantTors}
Assume $\kk$ is infinite of ${\rm char}(\kk)=0$ and fix an integer $d\geq 2$. Set $a=P(d)-1$ and $\ell=P(d)-P(d-1)-1$. For $n$ sufficiently large, a general principal symmetric ideal $I=(f)_{\sym_n}$ generated by a homogeneous polynomial $f$ of degree $d$ yields a quotient $A=R/I$ so that for $1\leq i \leq n-1$ we have 
\begin{gather*}
\Tor_i(A,\kk)_{i+d-1} \cong \bigoplus\limits_{\substack{|\l|\leq i+d\\\l_1+|\l|\leq n\\\l\neq (1^i), \l\neq(1^{i-1})}}\left(\Sp_{\l(n)}\right)^{a_{(d, 1^{i})}^{\l(n)}} \oplus \left(\Sp_{(n-i+1,1^{i-1})}\right)^{a_{(d, 1^{i})}^{(n-i+1, 1^{i-1})}-a} \\ \oplus
\hfill\left(\Sp_{(n-i,1^{i})}\right)^{a_{(d, 1^{i})}^{(n-i, 1^{i})-a}}
\end{gather*}
and all the remaining nonzero components of $\Tor_i^R(A,\kk)$ with $1\leq i$ are given by\\
\begin{gather*}
\Tor_{n-1}^R(A,\kk )_{n-1+d} \cong \Sp_{(1^n)}^\ell
\\
\Tor_{n}(A,\kk)_{n-1+d} \cong \bigoplus\limits_{\substack{\l_1+|\l|\leq n\\\l(n)\neq (1^n)\\ \l(n)\neq(2,1^{i-2})}}\left(\Sp_\l\right)^{a_{(d, 1^{i})}^{\l(n)}} \oplus \left(\Sp_{(1^n))}\right)^{a_{(d, 1^{i})}^{(1^n)}+P(d-1)+1} \oplus 
\hfill\left(\Sp_{(2,1^{n-2})}\right)^{a_{(d, 1^{i})}^{(2, 1^{n-2})}-a} 
\\
\Tor_{n}^R(A,\kk )_{n+d}\cong  \Sp_{(1^n)}^a.
\end{gather*}

The formulas above utilize exponents  given by \eqref{eq:plethysm}.
\end{thm}
\begin{proof}
By \Cref{l:narrow}~(5) we have  a short exact sequence
\begin{equation}
\label{eq:quadratic2}
0\to I\to R_{\geqslant d}\xrightarrow{\varphi} {\kk}^a(-d)\to 0.
\end{equation}
Recall that map $\varphi$ is given by $\varphi(r)=(r\circ F_1, \ldots, r\circ F_a)$ where $F_1, \ldots, F_a$ form a basis for $(I^\perp)_{-d}$. Furthermore, by \eqref{eq:W(t)} and \cref{cor:I perp equal W(t)},   the basis $F_1, \ldots, F_a$ can be chosen so that  it is indexed by $\l\vdash d, \l\neq(d)$ and 
$$F_\l=m_\l-t_\l m_{(d)}.$$
Note that these polynomials are invariant under the action of the symmetric group on $S$.  By \cref{l:narrow}~(5), the exact sequence \eqref{eq:quadratic2} is $\sym_n$-equivariant, when considering the trivial $\sym_n$-action on $\kk^a(-d)$.

By \Cref{thm:ses} and \Cref{lem:A5} (3), the $\sym_n$-equivariant sequence \eqref{eq:quadratic2} induces the following equivariant short exact sequences relating the nonzero graded components of the relevant homology modules: 
\begin{gather}\label{eq:8.3}
0\to \Tor_i^R(I,\kk )_{i+d}\to \Tor_i(\m^d, \kk)_{i+d}\to \Tor_i(\kk^a,\kk )_{i}\to 0\qquad \text{for  $i\le n-2$}\\
\label{eq:8.4}
0\to \Tor_{n-1}^R(I,\kk )_{n-1+d}\to  \Tor_{n-1}^R(\m^d, k)_{n-1+d}\to \Tor_{n-1}^R(\kk^a,\kk )_{n-1}\to \Tor_{n-2}^R(I,\kk )_{n-1+d}\to 0\\
\label{eq:last}
\Tor_n^R(\kk^a,\kk )_{n}\cong\Tor_{n-1}^R(I, \kk)_{n+d}
\end{gather}
Focusing on $i\leq n-2$ and internal degree  $i+d$ and substituting  \eqref{eq:tormdequivsmalli} and  \eqref{eq:equivTorkk}  into \eqref{eq:8.3} we obtain the short exact sequence below which yields the first item of the claim
\[
0\to \Tor_i^R(I,\kk )_{i+d}\to  \bigoplus_{\substack{|\l|\leq i+d\\\l_1+|\l|\leq n}}\left(\Sp_{\l(n)}\right)^{a_{(d, 1^{i})}^{\l(n)}}\to \Sp_{(n-i,1^i)}^a \oplus  \Sp_{(n-i+1,1^{i-1})}^a\to 0.
\]
 For $i=n-2$,  we have isomorphisms in $\MG R$ 
\begin{eqnarray}
 \Tor_{n-2}^R(I,\kk )_{n-1+d} &=& \Tor_{n-1}^R(R/I,\kk )_{n-1+d}\cong  (\Tor_{1}^R(I^\perp,\kk )_{d-1})^\vee\otimes_{\kk}\Sp_{(1^n)} \nonumber \\
 &\cong&(L_{F_1,\cdots, F_a})^\vee\otimes_{\kk}\Sp_{(1^n)}=\Sp_{(n)}^\ell \otimes_{\kk}\Sp_{(1^n)}=\Sp_{(1^n)}^\ell.\label{eq:8.5}
 \end{eqnarray}
The first isomorphism comes from \Cref{equivariantBoij} and \Cref{lem:torbalance}, the second isomorphism comes from  \cref{en-props} (noting that it is indeed an isomorphism in $\MG R$ due to the $\sym_n$-structure on $L_{F_1,\cdots, F_a}$) and the second to  last equality comes from \eqref{eq:dimL}. 
 
 Substituting  \eqref{eq:tormdequivsmalli},  \eqref{eq:equivTorkk} and \eqref{eq:8.5} into \eqref{eq:8.4} yields 
 \[
 0\to \Tor_{n-1}^R(I,\kk )_{n-1+d}\to  \bigoplus_{\substack{\l_1+|\l|\leq n}}\left(\Sp_{\l(n)}\right)^{a_{(d, 1^{n-1})}^{\l(n)}}\to \Sp_{(1^n)}^a \oplus \Sp_{(2,1^{n-2})}^a \to \Sp_{(1^n)}^\ell\to 0.
 \]
 whence the second item of the claim follows. Finally, \eqref{eq:last} establishes the last claim.
 \end{proof}
 
 We now recall and prove our main theorem.
 
 \begin{thm}
 \label{introthm}Suppose $\kk$ is infinite with ${\rm char}(\kk)=0$ and fix an integer $d\geq 2$. For sufficiently large  $n$, a general principal symmetric ideal $I$ of ${\kk}[x_1,\ldots, x_n]$ generated in degree $d$ has the following properties: 
\begin{enumerate}[\quad\rm(1)]
\item the Hilbert function of $A=R/I$ is given by
\[
 \HF_A(i):=\dim_\kk A_i = \begin{cases}\dim_\kk R_i &\text{if $i\le d-1$.}\\
 P(d)-1&\text{if $i=d$}\\
 0 &\text{if $i>d$,}\
 \end{cases}
\]
\item the betti table of $A$ has the form
\[
\begin{matrix}
     &0&1&2&\cdots&i &\cdots & n-2 &n-1&n\\
     \hline
     \text{total:}&1&u_1&u_2&\cdots&u_i&\cdots&u_{n-2} &u_{n-1}+\ell&a+b\\
     \hline
     \text{0:}&1&\text{.}&\text{.}&\text{.}&\text{.}&\text{.}&\text{.}&\text{.}&\text{.}\\
     \text{\vdots}&\text{.}&\text{.}&\text{.}&\text{.}&\text{.}&\text{.}&\text{.}&\text{.} &\text{.}\\
     \text{d-1:}&\text{.}&u_1&u_2&\cdots&u_i&\cdots&u_{n-2} &u_{n-1}&b\\
     \text{d:}&\text{.}&\text{.}&\text{.}&\text{.}&\text{.}&\text{.}&\text{.}&\ell & a \\
     \end{matrix}
\]
with 
\begin{eqnarray*}
a &=&P(d)-1,\\
 b&=&
 \dim_\kk R_{d-1}-(P(d)-1)(n-1)+P(d-1) \\
  u_{i+1} &=& \binom{n+d-1}{d+i}\binom{d+i-1}{i}-\left(P(d)-1\right)\binom{n}{i} \\
    \ell &=& P(d)-P(d-1)-1, 
  \end{eqnarray*}
\item the graded minimal free resolution of $A$ has $\sym_n$-equivariant structure described by the $\sym_n$-irreducible decompositions for the modules $\Tor_i^R(A,\kk)$ given in 
\Cref{thm:equivariantTors}.
\end{enumerate}

Moreover,  the Poincar\'e series of all finitely generated graded $A$-modules are rational, sharing a common denominator. When $d>2$, $A$ is Golod. When $d=2$, $A$ is Gorenstein and Koszul. 
\end{thm}

 \begin{proof}
 The first two items follow from \Cref{thm:mainnarrow} by means of the formulas developed in \Cref{l:narrow} for the Hilbert function and in \Cref{cor:betti} for betti numbers of extremely narrow algebras, respectively. The third item is \Cref{thm:equivariantTors}. Based on the value $\beta_n(A)=a+b$ established in part (2) as well as the evident  inequality $a\geq 1$, it can be deduced that $A$ is Gorenstein if and  only if $P(d)-1=a=1$ and $b=0$.  This yields $d=2$. The case of $d=2$ is discussed in detail in \Cref{thm:soclequadratic}, where it is also established that $A$ is Gorenstein provided the principal symmetric ideal is quadratic.

 The rationality of Poincar\' e series and the fact that $A$ is Golod when $d>2$ and is Koszul when $d=2$ follow from \Cref{prop:GolodKoszul}. 
 \end{proof}

\subsection{Cubic principal symmetric ideals}\label{subsec: cubic case}

While the general results in the previous section hold for $n$ sufficiently large, we want to make the point in this subsection that for the cubic case ($d=3$) specifically, the general behavior described in  \Cref{thm:cubic}  holds whenever $n\ge 5$, and this precise lower bound on $n$ is not covered by the previous results.  (See \Cref{ex:d=3}, which shows that our previous results would only apply to $n\geq 26$ when $d=3$.)  Additionally, the lower bound $n\geq 5$ is sharp, that is, it is the smallest embedding dimension in which the betti numbers of cubic principal symmetric ideals behave as predicted in \Cref{introthm} (2). This is illustrated by the Macaulay2 computations recorded in \Cref{ex:fewvar} below.

\begin{ex}
\label{ex:fewvar}
 Computations performed using the computer algebra system Macaulay2 \cite{M2} indicate that with very high probability for $\kk=\Q$ a general cubic principal symmetric ideal $I$ has the following betti table in low embedding dimensions $n\leq 4$:
\[
\begin{tabular}{|c|c|c|c|c|}
\hline
$n$ & 1 & 2 & 3& 4\\
\hline
 \text{betti table of} $R/I$ 
 &
 $\begin{matrix}
     &0&1\\\text{0:}&1&\text{.}\\\text{1:}&\text{.}&\text{.}\\\text{2:}&\text{.}&1\\\end{matrix}
     $
  &
 $\begin{matrix}
     &0&1&2\\\text{0:}&1&\text{.}&\text{.}\\\text{1:}&\text{.}&\text{.}&\text{.}\\\text{2:}&\text{.}&2&\text{.}\\\text{3:}&\text{.}&\text{.}&\text{.}\\\text{4:}&\text{.}&\text{.}&1\\\end{matrix}$
&
 $\begin{matrix}
     &0&1&2&3\\\text{0:}&1&\text{.}&\text{.}&\text{.}\\\text{1:}&\text{.}&\text{.}&\text{.}&\text{.}\\\text{2:}&\text{.}&6&4&\text{.}\\\text{3:}&\text{.}&\text{.}&3&1\\\text{4 :}&\text{.}&\text{.}&\text{.}&1\\\end{matrix}$
&
 $\begin{matrix}
     &0&1&2&3&4\\\text{0:}&1&\text{.}&\text{.}&\text{.}&\text{.}\\\text{1:}&\text{.}&\text{.}&\text{.}&\text{.}&\text{.}\\\text{2:}&\text{.}&15&26&10&\text{.}\\\text{3:}&\text{.}&\text{.}&\text{.}&4&1\\\text{4:}&\text{.}&\text{.}&\text{.}&\text{.}&1\\\end{matrix}$ \\
\hline
\end{tabular}
\]
\end{ex}

We now present an example that requires $n\geq 5$ variables. 

\begin{ex}
\label{cubic-ex}
Assume $n\ge 5$ and $\ch(\kk)\neq 2$. Take 
\[I=(x_1^3-x_2^3+x_1^2x_3+x_2x_3x_4-x_2x_3x_5)_{\sym_n}.
\] 
\begin{claim}
\begin{align*}
I&=(x_i^3-x_j^3,\,\, x_k^2x_\ell,\,\, x_px_qx_r-x_ux_vx_w\mid i<j, k\ne \ell, p<q<r, u<v<w)\\
&=(x_1^3-x_j^3,\,\, x_k^2x_\ell,\,\, x_1x_2x_3-x_ux_vx_w\mid j, k\ne \ell, u<v<w, \{1,2,3\}\ne \{u,v,w\})\,,
\end{align*}
\end{claim}
Indeed, let $f=x_1^3-x_2^3+x_1^2x_3+x_2x_3x_4-x_2x_3x_5$. Notice that 
$$
g:=\frac{1}{2}\left(f+(4\, 5)\cdot f\right)=x_1^3-x_2^3+x_1^2x_3\in I\qquad\text{and}\qquad h:=f-g=x_2x_3x_4-x_2x_3x_5\in I\,.
$$
If $\sigma\in\sym_5$ satisfies $\sigma(2)=1$ and $\sigma(4)=2$ we have
$$
s:= g-(2\, 4)\cdot g=-x_2^3+x_4^3\in I\qquad \text{and} \qquad g+\sigma \cdot s=x_1^2x_3\in I
$$
and 
$$
t:=h-(1\,5)(3\, 4)\cdot h=x_2x_3x_4-x_2x_3x_5-(x_2x_3x_4-x_2x_4x_1)=x_1x_2x_4-x_2x_3x_5\in I\,.
$$
By applying suitable permutations, $t$ and $h$ yield all binomials $ x_1x_2x_3-x_ux_vx_w$ where $\{1,2,3\}\cap \{u,v,w\}$ has two, respectively one, element(s). If $n=5$ this recovers all the binomials $ x_1x_2x_3-x_ux_vx_w$ with $1\leq u,v,w\leq 5$. 

If $n\ge 6$, then we also get 
$$
t-(1\,3)(2\,4)(5\,6)\cdot h=x_1x_2x_4-x_2x_3x_5-(x_1x_2x_4-x_1x_4x_6)=x_1x_4x_6-x_2x_3x_5\in I
$$
The Claim follows from here, as all listed generators of $I$ can be obtained by acting with a permutation on one of the monomials or binomials that we established to be in $I$. 
The generators listed in the second line of the Claim are clearly linearly independent, and thus a count yields
$$
\dim_{\kk}I_3=\dim_\kk R_3-2\,.
$$
A computation in Macaulay2 with $\kk=\mathbb Q$ yields that the betti table of the ideal described in this example for $n=5$ is
\begin{equation}
\label{e:betti d=3}
\begin{matrix}
     &0&1&2&3&4&5\\\text{0:}&1&\text{.}&\text{.}&\text{.}&\text{.}&\text{.}\\\text{1:}&\text{.}&\text{.}&\text{.}&\text{.}&\text{.}&\text{.}\\\text{2:}&\text{.}&33&95&    106&50&5\\\text{3:}&\text{.}&\text{.}&\text{.}&\text{.}&\text{.}&2\\\end{matrix}
\end{equation}
\end{ex}

Notice that this resolution is {\it almost linear}, in the sense of \cite{Iar}, meaning that the matrices giving the differential have linear entries, except for the first and the last. We show the essential features of \Cref{cubic-ex} persist in embedding dimension $n\geq 5$.

\begin{thm} 
\label{thm:cubic}     
Assume $n\ge 5$ and $\kk$ is infinite with $\ch(\kk)\ne 2$. A general principal symmetric ideal $I$ generated by a homogeneous cubic polynomial yields a quotient $A=R/I$ which is a $3$-extremely narrow algebra, thus compressed, with permissible socle polynomial $\frac{n(n-3)}{2}z^2+2z^3$. Moreover, $A$ has an almost linear resolution over $R$ with betti numbers given  in  \Cref{introthm}. Moreover, if $\text{char}(\kk)=0$, the $\sym_n$-equivariant structure on the resolution of $A$ is as described in \cref{thm:equivariantTors}. 
\end{thm}

\begin{proof}
As explained in \cref{s:prelim}, and using the notation therein, we consider $f=f_c$ with $c\in \mathbb P^{N-1}$, where $N=\binom{n+1}{2}$ and $f=c_1m_1+c_2m_2+\cdots +c_Nm_N$. We find it convenient to rename the coefficients $c_k$ as follows: 
\[
c_k=\begin{cases}
a_i &\text{if $m_k=x_i^3$}\\
b_{ij}  &\text{if $m_k=x_i^2x_j$ with $i\ne j$}\\
e_{ijk} &\text{if $m_k=x_ix_jx_k$ with $i<j<k$.}
\end{cases}
\]
With this notation, we have: 
\[
f=\sum_{i=1}^n a_ix_i^3+\sum_{1\le i, j\le n, i\ne j}b_{ij}x_i^2x_j+\sum_{1\le i<j<k\le n}e_{ijk}x_ix_jx_k\,.
\]
We will verify conditions (1) and (2) of \cref{strategy}, for $c$ general. 

To show (1),  use \Cref{cubic-ex}  and upper semicontinuity (\cref{upper-s}) to  find a non-empty Zariski open set $U$ such that, if $c\in U$, then $\dim_\kk I_3\ge \dim_\kk  R_3-2$. 

Set 
$$
\alpha=\sum_{i=1}^n a_i\,, \qquad \beta=\sum_{1\le i, j\le n, i\ne j}b_{ij}\qquad\text{and}\qquad  \epsilon=\sum_{1\le i<j<k\le n}e_{ijk}
$$
To show (2), we define 
\begin{gather*}
F=\beta \left(\sum_{i=1}^n y_i^{(3)}\right)-\alpha\cdot \left(\sum_{\scriptscriptstyle1\le i, j\le n, i\ne j}y_i^{(2)}y_j\right)\\
G=\beta\cdot \left( \sum_{\scriptscriptstyle1\le i<j<k\le n}y_iy_jy_k\right)-\epsilon\cdot \left(\sum_{\scriptscriptstyle 1\le i, j\le n, i\ne j}y_i^{(2)}y_j\right)
\end{gather*}
A computation shows that $f\circ F=0=f\circ G$. Since $F$ and $G$ are invariant under the action of  $\sym_n$, it follows that  $(\sigma\cdot f)\circ F=0=(\sigma\cdot f)\circ G$ for all $\sigma\in \sym_n$,  and hence $F, G\in (I^\perp)_{-3}$. 
We show  that, for $f=f_c$ general, $F$ and $G$ are linearly independent. Indeed, this is satisfied when the matrix 
$$
M'=\begin{bmatrix}
\beta&-\alpha&0\\
0&-\epsilon &\beta
\end{bmatrix}
$$
has maximal rank. The locus where the rank is not maximal is described by the ideal $J$ generated by $\beta\epsilon$, $\alpha\beta$ and $\beta^2$, which are homogeneous polynomials in the coefficients of $f$. Set $U'=\P^{N-1}\setminus V(J)$. This is a Zariski open set of $\P^{N-1}$. Observe that $U'$ is non-empty for the specific choice of  $c$ that corresponds to \cref{cubic-ex}. Indeed, in this case we have $\alpha=\epsilon=0$ and $\beta=1$ and $\beta^2\ne 1$. 

 We now  show 
$
\dim_\kk L_{F,G}=0
$
for general $f$. Assume 
$$\sum_{i=1}^n u_i(x_i\circ F)+v_i(x_i\circ G)=0\qquad\text{with}\quad  u_i,v_i\in \kk\,.
$$
We need to show $u_i=v_i=0$ for all $i$. 
Observe that 
\begin{align*}
x_i\circ F&=\beta y_i^{(2)}-\alpha\sum_{j\ne i} y_j^{(2)}-\alpha\sum_{j\ne i}y_iy_j\\
x_i\circ G&=\beta \sum_{j<k, j,k\ne i} y_jy_k-\epsilon \sum_{j\ne i} y_j^{(2)}- \epsilon\sum_{j\ne i}y_iy_j
\end{align*}
We have thus
$$
\sum_i\left( u_i\beta y_i^{(2)}-u_i\alpha\sum_{j\ne i} y_j^{(2)}-u_i\alpha\sum_{j\ne i}y_iy_j+v_i\beta \sum_{j<k, j,k\ne i} y_jy_k-v_i\epsilon \sum_{j\ne i} y_j^{(2)}- v_i\epsilon\sum_{j\ne i}y_iy_j\right)=0
$$
The degree $-2$ monomials in $S$ are linearly independent, and, so, equating the coefficients of each of these monomials, one obtains a system of linear equations with variables $u_i$, $v_i$. The matrix $M''$ of this system is of size $h\times 2n$, where $h=\dim_\kk R_2=\frac{n(n+1)}{2}$. For $c$ general, we need this matrix to have maximal rank. The locus where the rank is not maximal is cut out by the ideal $I_{2n}(M'')$ of maximal minors of $M''$, which are homogeneous polynomials in the coefficients of $f$. Set $U''=\P^{N-1}\setminus V(I_{2n}(M'))$. This is a Zariski open set of $\P^{N-1}$. We will show that $U''$ is non-empty for the specific choice of  $c$ that corresponds to \cref{cubic-ex}. In this case, we have $\alpha=\epsilon=0$ and $\beta=1$, and hence we need to verify that the equation 
$$
\sum_i\left( u_i y_i^{(2)}+v_i\sum_{j<k, j,k\ne i} y_jy_k\right)=0
$$
only has the trivial solution $u_i=v_i=0$. This can be easily verified.  

The Zariski open set $U\cap U'\cap U''$ is non-empty, as $\kk$ is assumed infinite and we have shown that each of the sets $U$, $U'$, $U''$ is non-empty. 
Then, with $c\in U\cap U'\cap U''$,  \cref{strategy} shows that the algebra $A$ is $3$-extremely narrow with socle polynomial $bz^2+2z^3$, where  
 \begin{align*}
b&=\dim_\kk R_2-2\dim_\kk R_1+\dim_\kk L_A=\frac{(n+1)n}{2}-2n+\dim_\kk L_A=\frac{n(n-3)}{2}\,.
 \end{align*}
The fact that $A$ has permissible socle type and has an almost linear resolution is a consequence of the fact  that $L_A=0$, in view of  \cref{cor:betti} and  \cref{l:narrow} (4). 

Since we established that $A$ is $3$-extremely narrow, one can proceed as in the proofs of \Cref{introthm} and \cref{thm:equivariantTors} to deduce that the formulas given there for the betti numbers and the $\sym_n$-equivariant structure of the resolution apply. 
\end{proof}

\appendix
\renewcommand{\MG}{\operatorname{mod}_G}
	\section{Brief review of the category $\MG R$}\label{appendix}
	
	Let $R = \bigoplus_{i \in \mathbb{Z}} R_i$ be a commutative $\mathbb{Z}$-graded $\kk$-algebra. We assume $R_0 \cong \kk$.  Let $G$ be a group which acts on $R$ by degree-preserving $\kk$-algebra automorphisms, that is, via a group homomorphism $G\to \Aut(R)$. When $R={\kk}[x_1,\ldots, x_n]$ equipped with the standard $\mathbb{Z}$-grading, we restrict attention to subgroups $G$ of $\GL_n(\kk)$, where $\GL_n(\kk)$ acts on $R_1=\langle x_1,\cdots,x_n\rangle$ in the standard way and acts trivially on $\kk$, and this action naturally induces an action on $R_i$ for $i>0$. The special case $G=\sym_n$, viewed as the subgroup of permutation matrices in $\GL_n(\kk)$, is utilized in the bulk of this paper.

For a graded ring $R$ as above, $M$ is said to be a (${\mathbb{Z}}$-)graded $R$-module if $M = \bigoplus_{j \in \mathbb{Z}} M_j$ is an $R$-module and the module structure respects the grading, i.e. $R_i \times M_j$ maps to $M_{i+j}$. In this Appendix we denote the $R$-module structure as $(r,m) \mapsto rm$ for $r \in R, m \in M$. Given two graded $R$-modules $M$ and $N$, a morphism in the category of graded $R$-modules is a map $f: M \to N$ which is a homomorphism of $R$-modules and also preserves gradings, i.e. $f(M_i) \subseteq N_i$ for all $i$. On the other hand, as explained in \cite{BH}, for the purposes e.g. of constructing the graded version of the $\textup{Ext}$ functor, it is necessary to consider a broader class of morphisms, and in this manuscript we denote by $\Hom_R(M,N) := \bigoplus_{i \in \mathbb{Z}} \Hom_i(M,N)$ the collection of homogeneous module morphisms, where $\varphi \in \Hom_i(M,N)$ is said to be homogeneous of degree $i$ if $\varphi$ is an $R$-module homomorphism, and in addition, $\varphi(M_n) \subseteq N_{n+i}$ for all $n$. 	It is not hard to see that $\Hom_R(M,N)$ is itself a graded $R$-module, defined for $\varphi \in \Hom_R(M,N)$ by $(r\varphi)(m) := r(\varphi(m))$ for all $r \in R, m \in M$.

In what follows, we restrict attention further to graded $R$-modules which have the additional data of a $G$-action which is compatible with the graded $R$-module structure as well as the $G$-action on the coefficient ring $R$, as follows.

	\begin{defn}\label{MGR}
The category of finitely generated ($\mathbb{Z}$-)graded $R$-modules with $G$-action $\MG R$ has as objects ${\mathbb{Z}}$-graded $R$-modules $M = \bigoplus_i M_i$ as above, endowed with an action of $G$ so that $g \in G$ acts as a degree-preserving $k$-vector space automorphism of $M_i$ for all $g \in G$ and $i \in \mathbb{Z}$, and
\begin{equation}\label{eq:def MGR}
g\cdot (r m)=(g\cdot r) (g\cdot m)\qquad\text{for all $g\in G, r\in R$, $m\in M$.}
\end{equation} 
The morphisms in $\MG R$ are {\em  $G$-equivariant} graded $R$-module homomorphisms $f:M\to N$. By $G$-equivariance we mean that a morphism $f: M \to N$ must satisfy $f(g\cdot m)=g\cdot f(m)$ for all $g\in G$, $m\in M$. 
\end{defn}

\begin{rem} \label{rem:MGRproperties}
We list some facts about $\MG R$. 
\begin{enumerate}
\item
The category of $\MG R$ is an abelian category. 
\item
 If $M$ and $N$ are in $\MG R$, then $G$ acts on $M \otimes_R N$  by $g\cdot (v\otimes w) = g \cdot v\otimes g \cdot w$ and $M\otimes_RN\in \MG R$. 
 \item 
 If $M,N\in \Mod_G(R)$ and $\dim_\kk(M_i)<\infty$ for all $i \in \mathbb{Z}$, then $G$ acts on $\Hom_R(M,N)$ via $(g \cdot f)(m) = g\cdot f(g^{-1}\cdot m)$ and $\Hom_R(M,N) \in \MG R$. 
 \item Consider $\kk$ with the trivial $R$-module structure, where $R_0 \cong \kk$ acts by usual multiplication in $\kk$ and $R_i$ for $i \neq 0$ acts on $\kk$ by sending everything to $0$. For $M\in \Mod_G(R)$, we define $M^\vee :=\Hom_\kk(M,\kk)$. We equip $M^\vee$ with an $R$-module structure defined by 
 $$
 (r \phi)(m) :=\phi(rm)
 $$
  for $r\in R$, $\phi\in M^\vee$ and $m\in M$.  Moreover, taking $N=\kk$ with trivial $G$-action in (3) above, $M^\vee$ is also equipped with a $G$-action, given by  
 \begin{equation}\label{eq: g phi m}
 (g\cdot \phi)(m) :=\phi(g^{-1}\cdot m) \qquad\text{for $g\in G$, $\phi\in M^\vee$, $m\in M$}. 
 \end{equation}
 We next observe that the above $R$-action and $G$-action are compatible in the sense that $M^\vee\in \MG R$ as well. To see this, we must check that~\eqref{eq:def MGR}
 holds. This can be seen by the following computation: 
\begin{align*}
\left((g \cdot r) (g \cdot \phi)\right)(m)&=(g\cdot \phi)\left((g\cdot r)m\right)=
\phi(g^{-1}\cdot (g\cdot r)m)= \phi(r(g^{-1}\cdot m))\\
&= (r \phi)(g^{-1}\cdot m)
=(g\cdot (r \phi))(m)\,
\end{align*} 
where the first and fourth equations follow from the definition of the $R$-module structure of $M^\vee$, the second and fifth equation are from~\eqref{eq: g phi m}, and the third equation is implied by~\eqref{eq:def MGR}, applied to $M\in \MG R$. 
 \end{enumerate}

\noindent  If additionally the category of finite dimensional graded representations of $G$ over $\kk$ is semisimple, then the following hold:
\begin{enumerate}\addtocounter{enumi}{4}
 \item Any surjective map $M\twoheadrightarrow N$ in $\MG R$ admits a section as a map of graded $G$ representations. If moreover $M$ is projective as an $R$-module then the map admits a section in $\MG R$. 
 \item Every object in $\MG R$ has a free resolution in $\MG R$ \cite[Proposition 2.4.9 and Remark 2.4.10]{Galetto1}.
 \end{enumerate}
\end{rem}

The following well-known isomorphisms of $R$-modules have equivariant counterparts.

\begin{lem}\label{lem:freeisos}
If $F,U$ are objects in $\MG R$ and $F$ is a finitely generated free $R$-module the following are natural isomorphisms in $\MG R$
\begin{enumerate}[\quad\rm(1)]
\item \label{L1} $\Hom_R(F,U)\cong F^*\otimes_R U$; 
\item ${F^*}^*\cong F$.
\end{enumerate}
\end{lem}
\begin{proof}
(1) The isomorphism is given for  $f\in F^*, u\in U$ by $f \otimes u \mapsto \psi$, where $\psi(v) = f(v)u$. For $g\in G$ we have 
\[
g\cdot(f\otimes u)=(g\cdot f)\otimes (g\cdot u) \mapsto \phi\,,\qquad  \text{ where }
\]
\[
\phi(v) = (g\cdot f)(v)(g\cdot u)=\left(g\cdot f(g^{-1}\cdot v)\right)(g\cdot u)=g\cdot f(g^{-1}\cdot v)u=(g\cdot \psi)(v)
\]
This shows that the isomorphism is $G$-equivariant.

(2) The isomorphism is given for $x\in F$ by $x\mapsto \phi$ where $\phi(f)=f(x)$ evaluates $f\in F^*$ at $x$.
Now $g\cdot x\mapsto \psi$ where $\psi(f)=f(g\cdot x)$ while 
\[(g\cdot \phi)(f)=g\cdot \phi(g^{-1}\cdot f)=gg^{-1}\cdot f(g\cdot x)=f(g\cdot x)\,.
\]
 Therefore $\psi=g\cdot \phi$ and the isomorphism is equivariant as desired.
\end{proof}

\begin{rem}\label{rem}
Note that the action of a group element $g\in G$ on a module $M$ in $\MG R$ is not $R$-linear, which makes it inconvenient to work with directly. We may repackage this in the category of $R$-modules as follows: consider the ring homomorphism $g:R\to R$ and let $g^*:\mod(R)\to\mod(R)$ denote the restriction of scalars functor via $g$. Specifically, if $M$ is an $R$-module then $g^*M$ is $ M$ as a vector space, with $R$-module structure defined by $r m=(g\cdot r)m$ for all $r\in R, m\in g^*(M)$,  and if $f$ is a homomorphism then $g^*f=f$. Since it preserves maps, restriction of scalars is an exact functor. 

If $M$ is an object in $\MG R$ then  for each $g\in G$ the map $g:M\to g^*M, g(m)=g\cdot m$ is an $R$-module homomorphism.
 \end{rem}

Let $M,N$ be objects in $\MG R$. We now discuss the canonical structure of $\Tor_i^R(M,N)$  as objects of $\MG R$, which we shall always consider when discussing these modules.

\begin{lem}\label{lem:A5}
 If  the category of finite dimensional graded representations of $G$ over $\kk$ is semisimple, then the following hold. 
 \begin{enumerate}[\quad\rm(1)]
 \item The derived functor $\Tor_i^R(M,N)$  applied to objects in $\MG R$ produce objects with canonical structure in $\MG R$. 
 \item If $M,N$ are objects in $\MG R$ then the $G$ action on $\Tor_i^R(M,N)$ does not depend on the choice of projective  resolution for $M$ or for $N$. 
 \item Short exact sequences
$
0\to M' \to M \to M'' \to 0 \text{ and } 0\to N' \to N \to N'' \to 0
$
induce $G$-equivariant long exact sequences in homology.
 \end{enumerate} 
\end{lem}
\begin{proof}
(1) If  $P$ is a projective resolution for an object $M$ of $\MG R$  then for every $g\in G$ applying \Cref{rem}  one constructs an $R$-linear chain map $g:P\to g^*P$ lifting  $g:M\to g^*M$. If $F:\MG R\to \MG R$ is a right exact functor, one obtains induced $R$-linear maps in homology $g:H(F(P))\to H(F(g^*P))$. Since the functor $g^*$ is exact we have $H(F(g^*P))=g^*H(F(P))$ and we interpret $g:H(F(P))\to g^*H(F(P)$ as an action of $G$ on $H(F(P))$, which endows $H(F(P))$ with a canonical structure in $\MG R$. 

(2) If  $P$ and $Q$ are projective resolutions for an object $M$ of $\MG R$ and $F:\MG R\to \MG R$ is a right exact functor, then one constructs comparison maps $f:P\to Q$ and $h:Q\to P$ in $\MG R$ lifting the identity on $M$. Let $g\in G$. Applying the functor $g^*$ yields an $R$-linear chain map  $g^*f:g^*P\to g^*Q$ which fits into a commutative diagram (since the comparison maps are natural)
\[
\begin{tikzcd} 
  P \arrow[r, "f"] \arrow[d, "g"] & Q \arrow[d, "g"] \\
g^*P    \arrow[r,  "g^*f" ]&  g^*Q.
 \end{tikzcd}
 \implies
 \begin{tikzcd} 
  H(F(P)) \arrow[r, "H(f)"] \arrow[d, "g"] & H(F(Q)) \arrow[d, "g"] \\
g^* H(F(P))   \arrow[r,  "H(g^*f)" ]&  g^*H(F(Q))
 \end{tikzcd}
 \]

Since $hf$ and the identity map are both chain maps $P\to P$ lifting the identity, they are homotopic, thus induce the same map (the identity) on $H(F(P))$. This shows that the map $H(f):H(F(P))\to H(F(Q))$ induced by $f$ is an $R$-module isomorphism and the commutative diagram above shows it is a $G$-equivariant.

The above considerations applied to $F=-\otimes_RN$ and $G=M\otimes_R -$ yield the claim regarding the structure of the respective derived functors  in $\MG R$ and their independence on resolutions.

(3) It is well known that the functor that takes short exact sequences of complexes to long exact sequences in homology is natural. Applying this functor to the map of complexes
\[
\begin{tikzcd} 
0 \arrow[r, ] &  F(P')\arrow[r, ] \arrow[d, "g"] & F(P) \arrow[r, ] \arrow[d, "g"] & F(P'') \arrow[r, ] \arrow[d, "g"] &0\\
0 \arrow[r, ] &  g^*F(P') \arrow[r,  "" ]&  g^*F(P) \arrow[r, ] & g^*F(P'') \arrow[r, ] &0 
 \end{tikzcd}
\]
where $P', P, P''$ are projective resolutions for $M', M, M''$ induces a chain map in homology which shows the corresponding long exact sequence in homology is $G$-equivariant.
\end{proof}

\begin{lem}\label{lem:torbalance}
 If  the category of finite dimensional graded representations of $G$ over $\kk$ is semisimple, then the derived functor $\Tor_i^R(-,-)$  applied to objects in $\MG R$ does not depend on which argument is resolved.
 \end{lem}
 \begin{proof}
 Consider equivariant projective resolutions $F \xrightarrow{\simeq} M$ and $G \xrightarrow{\simeq}N$. Then there are equivariant quasi-isomorphisms $F \otimes_R N \xrightarrow{\simeq} F  \otimes_R G \xleftarrow{\simeq} M\otimes_R G$, which induce for all $i$ isomorphisms $H_i(F \otimes_R N) \cong H_i( F  \otimes_R G)_i \cong H_i(M\otimes_R G)$ in $\MG R$.
 \end{proof}

\end{document}